\documentclass[10pt,pdf]{article}

\usepackage{customizations}

\begin{document}

\title{Interpolation for normal bundles of general curves}

\author{
  Atanas Atanasov,
  Eric Larson, and
  David Yang
}
\date{}

\maketitle

\begin{abstract}
Given $n$ general points $p_1, p_2, \ldots, p_n \in \Pbb^r$,
it is natural to ask when there exists a curve $C \subset \Pbb^r$, of degree $d$
and genus $g$, passing through
$p_1, p_2, \ldots, p_n$.
In this paper, we give a complete answer to this question
for curves $C$ with nonspecial hyperplane section.
This result is a consequence of our main theorem, which states that
the normal bundle $N_C$ of a general nonspecial curve
of degree $d$ and genus $g$ in $\Pbb^r$ (with $d \geq g + r$)
has the property of \emph{interpolation} (i.e.\ that
for a general effective divisor $D$ of any degree on $C$, either
$H^0(N_C(-D)) = 0$ or $H^1(N_C(-D)) = 0$), with exactly three exceptions.
\end{abstract}


\tableofcontents

\newcommand{\versionedtext}[2]{#1}

\section{Introduction}
\label{S:introduction}

The study of curves in projective space is one of the major topics in modern algebraic geometry. It has also served as a central example in the broader interest in moduli spaces which has flourished during the past half-century. The goal of the present article is to address the following fundamental question about incidence conditions for curves.

\begin{named}{Main question}
  When is there a (smooth) curve of degree $d$ and genus $g$ passing through $n$ general points in $\Pbb^r$?
\end{named}

Several cases of this question,
and of closely related questions we shall discuss below, have been previously studied in the literature.
For example, the case of rational curves ($g = 0$) was answered independently by both Sacchiero \cite{sacchiero}
and Ran \cite{ran}, and partial results for space curves ($r = 3$) were obtained independently by both Perrin \cite{perrin}
and Atanasov~\cite{Atanasov-interpolation}.

There are also several generalizations worth mentioning. For example, given values $d$, $g$, $r$, and $n$, we can ask for the dimension of the space of appropriate curves which satisfy the incidence conditions for a general collection of $n$ points. Alternatively, we can also replace the points with with higher dimensional linear spaces, or even other subvarieties in projective space. It turns out that the main question and its generalizations are all related to a property of vector bundles over curves we call \emph{interpolation}. If the normal bundle of a curve satisfies interpolation, we deduce a statement about the deformation theory of the curve, which in turn can lead to an answer of the main question.

Before going any further, we will elaborate the connection between our main question and interpolation of normal bundles. Our references are \cite{Stevens} and \cite{perrin}. Let $\Hc_{d, g, r}$ and $\Pc_{n, r}$ respectively denote the Hilbert schemes of curves of degree $d$ and genus $g$ in $\Pbb^r$, and $n$ points in $\Pbb^r$. There is an incidence correspondence $\Sigma \subset \Pc_{n, r} \x \Hc_{d, g, r}$ (a flag Hilbert scheme) whose points are pairs $([D], [C])$ such that $D \subset C$.
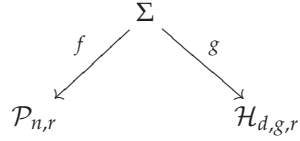
\[\xymatrix{
  & \Sigma \ar[dl]_-{f} \ar[dr]^-{g} \\
  \Pc_{n, r} && \Hc_{d, g, r}
}\]
Choose a point $([D], [C])$ such that $C$ is an lci curve and $D \subset C$ is a Cartier divisor. There is an identification of tangent spaces $T_{[C]} \Hc_{d, g, r} \cong \H^0(N_C)$ and similarly for $D$. Then the tangent space $T = T_{([D], [C])} \Sigma$ fits in the following Cartesian diagram.
\[\xymatrix{
  T \ar^-{d g}[r] \ar_-{d f}[d] & \H^0(N_C) \ar[d] \\
  \H^0(N_D) \ar[r] & \H^0(N_C|_D)
}\]

\begin{theorem}[Kleppe]
  \label{T:Kleppe}
  Let $([D], [C])$ be a geometric point of $\Sigma$. If $[C] \in \Hc_{d, g, r}$ is a smooth point, and the restriction morphism $\H^0(N_C) \rarr \H^0(N_C|_D)$ is surjective, then $f$ is smooth at the point $([D], [C])$. In particular, the image of $f$ contains an open neighborhood of $[D]$.
\end{theorem}

If the hypotheses of \cref{T:Kleppe} are satisfied, then we can give a positive answer to the main question. Consider the short exact sequence
\[\xymatrix{
  0 \ar[r] &
  N_C(-D) \ar[r] &
  N_C \ar[r] &
  N_C|_D \ar[r] &
  0,
}\]
whose cohomology sequence reads
\[\xymatrix{
  0 \ar[r] &
  \H^0(N_C(-D)) \ar[r] &
  \H^0(N_C) \ar[r] &
  \H^0(N_C|_D) \ar[r] &
  \H^1(N_C(-D)) \ar[r] &
  \H^1(N_C) \ar[r] &
  0.
}\]
If $\H^1(N_C(-D)) = 0$, then $\H^0(N_C) \rarr \H^0(N_C|_D)$ is surjective and $[C]$ is a smooth point of $\Hc_{d, g, r}$ (because $\H^1(N_C) = 0$), so we can apply \cref{T:Kleppe}. Note that if $N_C$ is nonspecial, then $\H^1(N_C(-D)) = 0$ is equivalent to
\begin{equation}
  \label{E:interpolation-simple}
  \h^0(N_C(-D)) =
  \h^0(N_C) - (r - 1) \deg(D).
\end{equation}
This discussion naturally leads us to the definition of interpolation (see \cref{T:interpolation}). In this particular case, the bundle $N_C$ satisfies interpolation if
\begin{enumerate}
\item
  for all $n \leq \h^0(N_C) / (r-1)$, there exists a degree $n$ divisor $D$ satisfying \cref{E:interpolation-simple}, and
\item
  for all $n > \h^0(N_C) / (r-1)$, there exists a degree $n$ divisor $D$ such that $\h^0(N_C(-D)) = 0$.
\end{enumerate}

Given a general curve $C$ of genus $g$, and a general line bundle $\mathcal{L}$
on $C$ of degree $d$,
it is well-known that
there exists a linear series on $C$ attached to $\mathcal{L}$
defining a map to $\Pbb^r$ if and only if
\[d \geq g + r.\]
Moreover, in this range, there is a unique component of the Hilbert scheme
corresponding to such curves; this component is distinguished
by the fact that a general curve in this component has a nonspecial hyperplane
section (which we will refer to as a ``nonspecial curve'' for brevity).
Our main result determines when the normal bundle of a general nonspecial
curve satisfies interpolation:

\begin{theorem} \label{thm:intro}
Let $C$ be a general nonspecial curve of degree $d$ and genus $g$
in $\Pbb^r$ (where $d \geq g + r$).
Then the normal bundle $N_C$ satisfies interpolation, unless:
\[(d, g, r) \in \{(5, 2, 3), (6, 2, 4), (7, 2, 5)\}.\]
\end{theorem}

The condition of interpolation is equivalent for rational curves
(and analogous in some sense for curves of higher genus)
to the conditions of semistability and section-semistability
(see Section~3 of~\cite{Atanasov-interpolation}), although we shall not make use of these analogies here.
However, we will remark that the analog of \cref{thm:intro}
for semistability of the normal bundle is known in the case of
rational curves ($g = 0$) as mentioned earlier \cite{ran, sacchiero},
as well as in the case of linearly normal elliptic curves
($g = 1$ and $d = r + 1$)
by work of Ein and Lazarsfeld \cite{ein-lazarsfeld}.

As a consequence, we answer the main question posed at
the beginning of the introduction for nonspecial curves:

\begin{corollary} \label{through-n-general} There exists a nonspecial curve $C$ of degree $d$ and genus $g$
in $\Pbb^r$ (with $d \geq g + r$),
passing through $n$ general points,
if and only if
\[\begin{cases}
(r - 1) n \leq (r + 1) d - (r - 3)(g - 1) & \text{if $(d, g, r) \notin \{(5, 2, 3), (7, 2, 5)\}$;} \\
\phantom{(r - 1)} n \leq 9 & \text{if $(d, g, r) \in \{(5, 2, 3), (7, 2, 5)\}$.} \\
\end{cases}\]
\end{corollary}

To prove \cref{thm:intro}, we will argue by inductively degenerating $C$
to a reducible nodal curve $X \cup Y$. We use results of Hartshorne and Hirschowitz \cite{Hartshorne-Hirschowitz} 
to guarentee the existance of particular such degenerations,
and to give descriptions of the restrictions $N_{X \cup Y}|_X$ and $N_{X \cup Y}|_Y$.
However, in order to reduce interpolation
for $N_C$ to statements about $N_{X \cup Y}|_X$ and $N_{X \cup Y}|_Y$, we need to
have a geometric description of the gluing data:
\begin{equation} \label{glue}
H^0(N_{X \cup Y}|_X) \rightarrow H^0(N_{X \cup Y}|_{X \cap Y}) \leftarrow H^0(N_{X \cup Y}|_Y).
\end{equation}

The key observation that makes it possible to approach \cref{thm:intro}
is the existence --- in the case when $Y = L$ is a line --- of certain
geometrically-defined line subbundles
$\Lc \subseteq N_{X\cup Y}$, which taken together enable
us to give an essentially-complete geometric description of the gluing data
in \cref{glue}.

For example, suppose that $Y = L$ meets $X$ in a single point $u$;
write $v \in L$ for some point on $L$ distinct from $u$. Then writing $S = \overline{v \cdot X}$
for the cone over $X$ with vertex $v$, the normal bundle $\mathcal{L}$ of $X \cup L$ in $S$ gives such a bundle.
We will see in \cref{S:reducible} that $\mathcal{L}|_L$ gives the positive subbundle of
$N_{X \cup L}|_L$; using this, we will reduce interpolation for $N_{X \cup Y}$
to interpolation for the vector bundle on $X$ given by the kernel of the natural map
\begin{equation} \label{modification-example}
N_{X \cup Y}|_X \to (N_{X \cup Y}|_X / \mathcal{L})|_u.
\end{equation}

\subsection*{Summary}

We begin the paper in \cref{S:modifications-arbitrary} and \cref{S:modifications-curves}
by studying \emph{modifications} of vector bundles, which are generalizations of the
above bundle defined on $X$ --- where $N_{X \cup Y}|_X$ is replaced an arbitrary vector
bundle on $X$, and $\mathcal{L}$ by an arbitrary subbundle of $N_{X \cup Y}|_X$.
That is, the modification $\mathcal{E}[D \to \mathcal{F}]$ of $\mathcal{E}$
along $\mathcal{F}$ at a Cartier divisor $D$ is simply the kernel of the natural map
\[\mathcal{E} \to (\mathcal{E} / \mathcal{F})|_D.\]
The main results of these sections are tools for dealing with
\emph{multiple modifications}
\[\mathcal{E}[D_1 \to \mathcal{F}_1][D_2 \to \mathcal{F}_2] \cdots [D_n \to \mathcal{F}_n],\]
which correspond to the bundles on $X$
that we would obtain by, say, iteratively applying the construction outlined above;
our ability to handle multiple modifications will allow us to inductively degenerate $C$, peeling off lines
one (or sometimes two) at a time.
Our study of modifications is divided into two sections: We begin in \cref{S:modifications-arbitrary}
by studying modifications of vector bundles on arbitrary varieties; and further study
the special case of curves in \cref{S:modifications-curves}. This is necessary since
we will need to apply results on modifications to the total space of a family of curves.

Our next topic in \cref{S:sequences} is interpolation and its interaction with modifications.
For example, under certain conditions we show that if a given vector bundle $\mathcal{E}$,
a sub-bundle $\mathcal{F}$, and the quotient $\mathcal{E} / \mathcal{F}$, all satisfy interpolation,
then so does the modification $\mathcal{E}[D \to \mathcal{F}]$.

In \cref{S:normal-bundles} and \cref{S:normal-bundles-examples},
we respectively define, and calculate, important examples of, certain 
sub-bundles of normal bundles of curves in projective spaces.
These bundles will include the bundle $\mathcal{L}$ appearing in \cref{modification-example},
as well as the necessary generalizations thereof
(which are necessary, say, when $L$ meets $X$ at two points instead of just one).

In \cref{S:specialization}, we prove the necessary ingredients to degenerate
$C$ to a reducible curve (e.g.\ we prove that the conclusion of Theorem~\ref{thm:intro}
is an open condition in the Hilbert scheme parameterizing curves of degree $d$
and genus $g$ in $\Pbb^r$).

The heart of the paper is \cref{S:reducible}, where all of the previous work
enables us to carry out the analysis described above
(c.f.\ \cref{modification-example}). We consider not only the case where $L$ meets $X$
once, but also where $X$ meets $L$ twice, as well as several other variants (e.g.\ $X$
contained in a hyperplane, simultaneously adding two lines, etc.).

We then move forward with our inductive argument: First, in \cref{S:hypothesis}, we define a certain class of modifications of normal bundles of curves which we will
inductively study. Then in \cref{S:fixed}, we show how the results of \cref{S:reducible}
allow us to reduce interpolation for
certain cases of this class of modifications of normal bundles to other
``simpler'' cases. In \cref{S:base}, we directly prove
that certain modified normal bundles satisfy interpolation; these form
the base of our inductive argument.

To finish the proof, we need an intricate
combinatorial argument to
show that the collection of inductive
arguments of \cref{S:fixed}, together with the base cases of \cref{S:base},
imply \cref{thm:intro}. This is briefly summarized in \cref{S:combinat-summary},
and detailed in \cref{S:combinat}

Finally, in \cref{S:exceptional}, we further explore the
three exceptional cases occurring in \cref{thm:intro}, understanding
geometrically why curves of degree $r + 2$ and genus $2$ in $\Pbb^r$
do not satisfy interpolation for $r \in \{3, 4, 5\}$. The reason is essentially
that the sub-bundle $N_{C/S}$ has too many sections,
where $S$ is the surface obtained by taking the union of all lines joining
pairs of points $\{p, q\} \subset C$ which are conjugate under the hyperelliptic involution.
Using this construction, we also establish \cref{through-n-general}.

\subsection*{Conventions}
\noindent
  Unless otherwise noted, we will consistently make the following conventions.
  \begin{itemize}
  \item
    We will work over an algebraically closed field $K$ of characteristic $0$.
  \item
    All varieties are reduced, separated, finite type schemes over $K$.
  \item
    All curves are connected and locally complete intersection (lci); all families of curves have connected lci fibers.
  \item
    All vector bundles are locally free sheaves of finite constant rank.
  \item
    A subbundle refers to a vector subbundle with locally free quotient.
  \item
    All divisors are Cartier.
  \item
    We will call a vector bundle nonspecial if it has no higher cohomology.
  \end{itemize}


\subsection*{Acknowledgments}

The authors would like to thank Joe Harris for
introducing us to each other, and for his guidence
throughout this research.
We are also grateful to Bjorn Poonen for many helful conversations.
We would like to thank the members of the Harvard
and MIT mathematics departments, in particular Gabriel Bujokas,
Francesco Cavazzani, Anand Patel, Alexander Perry, and Eric Riedl;
and to thank the anonymous referee for helpful suggestions.

The first author would like to acknowledge the generous support
of the National Science Foundation, grant DMS-1308244, for
``Nonlinear Analysis on Sympletic, Complex Manifolds,
General Relativity, and Graphs.''
The second author would also like to acknowledge the generous
support both of the Fannie and John Hertz Foundation,
and of the Department of Defense
(NDSEG fellowship).


\section{Elementary modifications in arbitrary dimension}
\label{S:modifications-arbitrary}

Elementary modifications of vector bundles are a classical topic. Most sources focus on reduced divisors over curves, but the applications we have in mind require to us to relax these hypotheses. The goal of this section is to define an appropriate notion of modification and develop its properties.

Let $X$ be a variety and $E$ a vector bundle on it. Given an effective Cartier divisor $D \subset X$ and a subbundle $F \subset E|_U$ defined over an open $U$ containing the support of $D$, we consider the composition
\[\xymatrix{
  E \ar[r] &
  E|_D \ar[r] &
  (E/F)|_D
}\]
of the restriction to $D$ followed by a quotient. Both parts are surjective, hence so is the composition. We will call the kernel of the composition the \emph{(elementary) modification of $E$ at $D$ along $F$} and denote it by $E[D \to F]$. Our notation is inspired by the fact that sections of $E[D \to F]$ can be identified with sections of $E$ which point along $F$ when restricted to $D$:
\[
\H^0(E[D \to F]) = \{ \sigma \in \H^0(E) \;|\; \sigma|_D \in \H^0(F|_D) \}.
\]

The defining exact sequence of a modification $E[D \to F]$ is
\begin{equation}
  \label{E:modification-ses}
  \vcenter{\xymatrix{
      0 \ar[r] &
      E[D \to F] \ar[r] &
      E \ar[r] &
      (E/F)|_D \ar[r] &
      0.
    }}
\end{equation}
The inclusion $E[D \to F] \rarr E$ becomes an isomorphism when restricted to the complement $X \setminus \Supp(D)$. This is true since the cokernel $(E/F)|_D$ is supported on $D$.

    There is a second sequence, which can also be very handy:
    \begin{equation}
      \label{E:modification-ses2}
      \vcenter{\xymatrix{
          0 \ar[r] &
          E(-D) \ar[r] &
          E[D \to F] \ar[r] &
          F|_D \ar[r] &
          0.
        }}
    \end{equation}
    This is a consequence of the Snake Lemma applied to the following diagram with exact rows.
    \[\xymatrix{
      0 \ar[r] &
      0 \ar[r] \ar[d] &
      E \ar@{=}[r] \ar[d] &
      E \ar[r] \ar[d] &
      0 \\
      0 \ar[r] &
      F|_D \ar[r] &
      E|_D \ar[r] &
      (E/F)|_D \ar[r] &
      0
    }\]
    The third useful sequence
    \begin{equation}
      \label{E:modification-ses3}
      \vcenter{\xymatrix{
          0 \ar[r] &
          F \ar[r] &
          E[D \to F]|_U \ar[r] &
          (E|_U/F)(-D) \ar[r] &
          0
        }}
    \end{equation}
    is a corollary of the following diagram.
    \[\xymatrix{
      0 \ar[r] &
      0 \ar[r] \ar[d] &
      E|_U \ar@{=}[r] \ar[d] &
      E|_U \ar[r] \ar[d] &
      0 \\
      0 \ar[r] &
      (E|_U/F)(-D) \ar[r] &
      E|_U/F \ar[r] &
      (E|_U/F)|_D \ar[r] &
      0
    }\]
Note that sequences \labelcref{E:modification-ses,E:modification-ses2} are valid over the entire variety $X$, while \labelcref{E:modification-ses3} makes sense only over the open $U$.

\begin{remark}
  \label{T:modification-splitting}
  The inclusion $F \rarr E[D \to F]|_U$ in \labelcref{E:modification-ses3} splits if the inclusion $F \rarr E|_U$ splits. We can then write
  \[
  E[D \to F]|_U = F \oplus (E|_U/F)(-D).
  \]
  In particular, if $U$ is affine, then both inclusions split.
\end{remark}

\begin{remark}
  \label{T:only-neighborhood}
  The modification $E[D \to F]$ only depends on the restriction of $F$ to $D$. Put differently, if $F$ and $F'$ are subbundles of $E$ such that $F|_D = F'|_D$, then
  \[
  E[D \to F] = E[D \to F'].
  \]
      For example, if the support of $D$ is an irreducible variety $Y \subset X$ and $D = n Y$, then $E[D \to F]$ only depends on $F$ in an $n$-th order neighborhood of $Y$.
\end{remark}

    Under the hypotheses we made, elementary modifications are vector bundles.

    \begin{proposition}
      \label{T:modification-vector-bundle}
      If $F \subset E|_U$ is a subbundle and $D$ is a Cartier divisor on $X$, then $E[D \to F]$ is a vector bundle.
    \end{proposition}

    \begin{proof}
      Since $E[D \to F]|_{X \setminus \Supp(D)} \cong E|_{X \setminus \Supp(D)}$, we can pass to an open neighborhood of $D$. For example, take the locus $U$ where $F$ is defined.
      
      Note that $E[D \to F]$ is finitely presented and it suffices to show it is flat. We will use the local criterion of flatness. Let $\Ac$ be a coherent sheaf over $X$. If we apply $\Torc_\bullet(-, \Ac)$ to sequence \labelcref{E:modification-ses3}, then $\Torc_1(E[D \to F], \Ac)$ sits between $\Torc_1(F, \Ac) = 0$ and $\Torc_1((E/F)(-D), \Ac) = 0$, so it must also be zero. This proves that $E[D \to F]$ is flat, hence locally free.
    \end{proof}

\begin{example}
  \label{T:modification-basic-examples}
  Taking $F = 0$ and $F = E$ gives rise to two basic examples:
  \[
  E[D \to 0] = E(-D),
  \qquad\qquad
  E[D \to E] = E.
  \]
      Similarly, when $D = \emptyset$, then
      \[
      E[D \to F] = E
      \]
      is independent of $F$.
\end{example}

Consider a morphism of varieties $f \cn Y \rarr X$. Let $D$ be an effective Cartier divisor on $X$ such that its support contains no component of the image of $f$. Under this hypothesis, the pullback divisor $f^\ast D$ is well-defined and modifications respect pullbacks.

\begin{proposition}
  \label{T:modification-pullback}
  Let $f \cn Y \rarr X$ be a morphism of varieties and $D$ an effective divisor on $X$ such that its support does not contain any component of the image of $f$. If $E$ is a vector bundle on $X$ and $F \subset E$ a subbundle, then there is a natural isomorphism
  \[
  f^\ast E[D \to F] \cong
  (f^\ast E)[f^\ast D \to f^\ast F].
  \]
\end{proposition}

\begin{proof}
  This follows by pulling back the defining sequence \labelcref{E:modification-ses}.
\end{proof}

\begin{remark}
  \label{T:modification-affine}
  Any vector bundle can be decomposed as a gluing of bundles over affine opens which themselves intersect in affines. The compatibility of modifications and pullbacks, open embeddings in particular, allows us to reduce various statements about modifications over general varieties to statements about affine varieties. We will use this technique in several of the arguments that follow.
\end{remark}

\begin{remark}
  \label{T:linearly-independet}
  Suppose we have a vector bundle $E$ over a variety $X$ and a collection of subbundles $F_i \in E$ indexed by $i \in I$. Stating that $\{ F_i \}$ are linearly independent means that for all $x \in X$ the fibers $\{ F_i|_x \}$ are linearly independent in $E|_x$ as vector spaces.

  There is an alternative formulation of this statement. The individual inclusions $F_i \rarr E$ induce a morphism
  \[
  \phi \cn \bigoplus_{i \in I} F_i \rarr E.
  \]
  Then $\{ F_i \}$ are linearly independent if and only if $\phi$ is injective and has locally free cokernel, that is, $\bigoplus_{i \in I} F_i \subset E$ is a subbundle. This restatement is convenient since it allows us to deal with linear independence in a global fashion.
\end{remark}

There are correspondences between certain classes of subbundles of $E$ and $E[D \to F]$. If $X$ is a curve this is true more generally without any restrictions on the subbundles in consideration (see \cref{S:modifications-curves}).

To state our result, consider a subbundle $F \subset E|_U$ and an effective divisor $D$ on $X$ whose support is contained in $U$. We define four sets of subbundles of $E$:
\begin{align*}
  S_1(E, F, D)
  &= \{ G \subset E \textrm{ subbundle} \;|\; \textrm{$G|_V \subset F|_V$ for some neighborhood $V \subset U$ of $D$} \}, \\
  S_2(E, F, D)
  &= \{ G \subset E \textrm{ subbundle} \;|\; \textrm{$F|_V \subset G|_V$ for some neighborhood $V \subset U$ of $D$} \}, \\
  S_3(E, F, D)
  &= \{ G \subset E \textrm{ subbundle} \;|\; \textrm{$G|_D$ and $F|_D$ are linearly independent} \}, \textrm{ and} \\
  S(E, F, D)
  &= S_1(E, F) \cup S_2(E, F, D) \cup S_3(E, F, D).
\end{align*}
    Direct inspection shows that
    \begin{align*}
      S_1(E, F, D) \cap S_2(E, F, D)
      &= \{ G \subset E \textrm{ subbundle} \;|\; \textrm{$G|_V = F|_V$ for some neighborhood $V \subset U$ of $D$} \}, \\
      S_1(E, F, D) \cap S_3(E, F, D)
      &=
        \begin{cases}
          \{ G \subset E \textrm{ subbundle} \} & \textrm{if $D = \emptyset$}, \\
          \{ 0 \} & \textrm{otherwise, and}
        \end{cases} \\
      S_2(E, F, D) \cap S_3(E, F, D)
      &=
        \begin{cases}
          \{ G \subset E \textrm{ subbundle} \} & \textrm{if $F|_D = 0$}, \\
          \emptyset & \textrm{otherwise}.
        \end{cases}
    \end{align*}
    In particular, if $D \cap X' \neq \emptyset$ for every irreducible component $X' \subset X$, then
    \begin{align*}
      S_1(E, F, D) \cap S_2(E, F, D)
      &=
        \begin{cases}
          \{ \overline{F} \} & \textrm{if $F$ extends to $\overline{F}$ defined over $X$}, \\
          \emptyset & \textrm{otherwise, and}
        \end{cases} \\
      S_1(E, F, D) \cap S_3(E, F, D)
      &= \{ 0 \}.
    \end{align*}
\begin{proposition}
  \label{T:modification-subbundles-bijection}
  Let $F \subset E|_U$ be a subbundle and $D$ an effective divisor on $X$ whose support is contained in $U$. Note that we can also treat $F$ as a subbundle of the modification $E[D \to F]|_U$ by sequence \labelcref{E:modification-ses3}. Then there are bijections
  \begin{align*}
    \phi_1 \cn S_1(E, F, D) & \longrightarrow S_1(E[D \to F], F, D), \\
    \phi_2 \cn S_2(E, F, D) & \longrightarrow S_2(E[D \to F], F, D), \\
    \phi_3 \cn S_3(E, F, D) & \longrightarrow S_3(E[D \to F], F, D), \textrm{ and} \\
    \phi \cn S(E, F, D) & \longrightarrow S(E[D \to F], F, D),
  \end{align*}
  such that
  \begin{enumeratea}
  \item
    \label{T:modification-subbundles-bijection-a}
    $\phi|_{S_i(E, F, D)} = \phi_i$ for $i = 1,2,3$,
  \item
    \label{T:modification-subbundles-bijection-b}
    $\phi$ is compatible with pullbacks,
  \item
    \label{T:modification-subbundles-bijection-c}
    given $G_1, G_2 \in S(E, F, D)$, then $G_1 \subset G_2$ in a neighborhood of $D$ if and only if $\phi(G_1) \subset \phi(G_2)$ in a neighborhood of $D$,
  \item
    \label{T:modification-subbundles-bijection-d}
    given $\{ G_i \;|\; i \in I \} \subset S_1(E, F, D)$, then $\{ G_i \}$ are linearly independent along $D$ if and only if $\{ \phi(G_i) \}$ are linearly independent along $D$,
  \item
    \label{T:modification-subbundles-bijection-e}
    given $\{ G_i \;|\; i \in I \} \subset S_3(E, F, D)$, then $\{ F \} \cup \{ G_i \}$ are linearly independent along $D$ if and only if $\{ F \} \cup \{ \phi(G_i) \}$ are linearly independent along $D$, and
  \item
    \label{T:modification-subbundles-bijection-f}
    if $D = \emptyset$, then $E[D \to F] \cong E$ induces an identification of $S(E, F, D)$ and $S(E[D \to F], F, D)$ such that $\phi$ becomes the identity map.
  \end{enumeratea}
\end{proposition}

\begin{proof}
  Without loss of generality, assume that $X$ is connected. If not, we can use the morphisms $\phi_i$ and $\phi$ defined for each connected component to assemble their global versions.

  Let us start by considering $\phi_1$. We will bootstrap our way up by first constructing a simpler bijection $\phi_1' \cn S_1(E|_U, D) \rarr S_1(E[D \to F]|_U, D)$. Note that $F$ sits in both $E|_U$ and $E[D \to F]|_U$, so we can send $G \subset F \subset E|_U$ to $G' = G \subset F \subset E[D \to F]|_U$. To verify that this map sends subbundles to subbundles, we observe that given $F \subset E|_U$ is a subbundle, then $G \subset E|_U$ is a subbundle if and only if $G \subset F$ is a subbundle. An analogous statement is true for the inclusions $G' \subset F \subset E[D \to F]|_U$. To define $\phi_1$ in terms of $\phi_1'$, it suffices to note that $\phi_1'$ is compatible with open embeddings (more generally, it is compatible with pullbacks) and $E[D \to F]|_{X \setminus \Supp(D)} \cong E|_{X \setminus \Supp(D)}$. Simply put, the image $G' = \phi_1(G)$ is glued from $\phi_1'(G|_U)$ and $G|_{X \setminus \Supp(D)}$ along $U \setminus \Supp(D)$. It is easy to see that both $\phi_1'$ and $\phi_1$ are bijections.

  The compatibility with pullbacks follows from the fact sequences \labelcref{E:modification-ses,E:modification-ses2,E:modification-ses3} are preserved by pullbacks as long as the hypotheses of \cref{T:modification-pullback} are satisfied.

  The second and third morphisms, while still induced by the inclusion $E[D \to F] \rarr E$, are a little more interesting. For example, the main issue with $\phi_2$ is that given a subbundle $G' \subset E[D \to F]$, its image in $E$ is no longer a subbundle. One way to solve the problem is via saturation, but this makes it hard to understand the resulting subbundle $G \subset E$.

      Consider the second morphism $\phi_2$. Given a subbundle $G \subset E$ which contains $F$, we can show that the inclusion $G \rarr E$ lifts to an inclusion $G[D \to F] \rarr E[D \to F]$. We have constructed the following diagram with exact rows.
      \[\xymatrix{
        0 \ar[r] &
        G[D \to F] \ar[r] \ar[d]_{\iota} &
        G \ar[r] \ar[d] &
        (G/F)|_D \ar[r] \ar[d] &
        0 \\
        0 \ar[r] &
        E[D \to F] \ar[r] &
        E \ar[r] &
        (E/F)|_D \ar[r] &
        0
      }\]
      All vertical maps are injective, and the Snake Lemma produces the short exact sequence
      \[\xymatrix{
        0 \ar[r] &
        \Coker \iota \ar[r] &
        E/G \ar[r] &
        (E/G)|_D \ar[r] &
        0,
      }\]
      which identifies
      \[\Coker \iota \simeq (E/G)(-D).\]
      It follows that the inclusion $G[D \to F] \rarr E[D \to F]$ is a subbundle since we identified its cokernel with $(E/G)(-D)$. The morphism $\phi_2$ can thus be defined by sending $G$ to $G[D \to F]$.
  
  To construct the backward direction of $\phi_2$, we start with $G' \subset E[D \to F]$ which contains $F$. Similarly to \labelcref{E:modification-ses2}, there is a morphism $E \rarr E[D \to F](D)$. We define $G = \phi_2^{-1}(G')$ as the kernel of the composition
  \[\xymatrix{
    E \ar[r] &
    E[D \to F](D) \ar[r] &
    E[D \to F](D)/G'(D),
  }\]
  which is a subbundle of $E$. Following our construction, it is easy to check that $\phi_2$ is a bijection and it is compatible with pullbacks.

  We proceed to construct $\phi_3$. Let us start by constructing the forward direction first. Take the composition
  \[\xymatrix{
    G \ar[r] &
    E \ar[r] &
    (E/F)|_D,
  }\]
  where $G \subset E$ is such that $G|_D$ and $F|_D$ are linearly independent. By first restricting to $D$, we can identify its image $((F + G)/F)|_D$ with $G|_D$, so we obtain a morphism
  \[
  \xymatrix{
    G(-D) = \Ker(G \rarr G|_D) \ar[r] &
    E[D \to F].
  }\]
  Sending $G$ to $G' = G(-D)$ furnishes the forward direction of the bijection. To see that $G(-D) \subset E[D \to F]$ is a subbundle, observe that we have constructed the following diagram with exact rows.
  \[\xymatrix{
    0 \ar[r] &
    G(-D) \ar[r] \ar[d]_{\iota} &
    G \ar[r] \ar[d] &
    (F|_D + G|_D)/F|_D \ar[r] \ar[d] &
    0 \\
    0 \ar[r] &
    E[D \to F] \ar[r] &
    E \ar[r] &
    (E/F)|_D \ar[r] &
    0
  }\]
  Applying the Snake Lemma, we obtain the short exact sequence
  \[\xymatrix{
    0 \ar[r] &
    \Coker \iota \ar[r] &
    E/G \ar[r] &
    (E/(F + G))|_D \ar[r] &
    0,
  }\]
  which identifies
  \[\Coker \iota \simeq (E/G)[D \to (F + G)/G].\]
  In particular, the cokernel of the first vertical map is a vector bundle by \cref{T:modification-vector-bundle}. A similar analysis constructs the backward direction of the second map which sends $G' \subset E[D \to F]$ to $G = G'(D) \subset E$. Again, all diagrams are preserved by appropriate pullbacks (see \cref{T:modification-pullback}).

  To construct $\phi$ it suffices to note that $\phi_i$ agree on all pairwise intersections of their domains. This also ensures part (\ref{T:modification-subbundles-bijection-a}) is true. On a similar note, part (\ref{T:modification-subbundles-bijection-f}) follows immediately from the constructions of $\phi_i$.

  Next, we focus on part (\ref{T:modification-subbundles-bijection-c}). Without loss of generality, we may replace $X$ with an irreducible component which intersects $D$ non-trivially. Since $E$ and $E[D \to F]$ are isomorphic over $X \setminus \Supp(D)$, in particular, they are isomorphic over the generic point $\eta \in X$. What is more interesting is that if we identify $E|_\eta$ and $E[D \to F]|_\eta$, then $\phi_1$, $\phi_2$, $\phi_3$, and $\phi$ become the identity map. Since containment is a closed property, part (\ref{T:modification-subbundles-bijection-c}) follows immediately from the observations we made.

  We are left to demonstrate parts (\ref{T:modification-subbundles-bijection-d}) and (\ref{T:modification-subbundles-bijection-e}) of our claim. Consider a subset $\{ G_i \;|\; i \in I \} \subset S_1(E, F, D)$. Assume that $\{ G_i \}$ are linearly independent in a neighborhood $V \subset U$ of $D$. We can replace $X$ with $V$ so $\{ G_i \}$ are linearly independent everywhere. Note that $F$ sits in both $E$ and $E[D \to F]$ as a subbundle. Furthermore, all subbundles of $F$ remain unchanged by $\phi$. Since all $G_i$ are contained in $F$, part (\ref{T:modification-subbundles-bijection-d}) follows immediately.

  Finally, consider a subset $\{ G_i \;|\; i \in I \} \subset S_3(E, F, D)$ such that $\{ F \} \cup \{ G_i \}$ are linearly independent in a neighborhood $V$ of $D$. Recall that $\phi_3(G_i) = G_i(-D) \subset E[D \to F]$. After replacing $X$ with $V$, we have a subbundle $F \oplus G \rarr E$ where $G = \bigoplus_i G_i$. This inclusion lifts to a morphism $G(-D) \rarr E[D \to F]$ which fits in the following diagram with exact rows.
  \[\xymatrix{
    0 \ar[r] & F \oplus G(-D) \ar[r] \ar[d] & F \oplus G \ar[r] \ar[d] & G|_D \ar[r] \ar[d] & 0 \\
    0 \ar[r] & E[F \to D] \ar[r] & E \ar[r] & (E/F)|_D \ar[r] & 0
  }\]
  All vertical morphisms are injective, so the Snake Lemma identifies the cokernel of $F \oplus G(-D) \rarr E[D \to F]$ with
  \[
  (E/(F + G))(-D)
  \]
  which is a vector bundle. We have thus shown that $\{ F \} \cup \{ \phi(G_i) \}$ are linearly independent in $E[D \to F]$. The backward implication has an analogous proof, so we will omit that.
\end{proof}

Our discussion so far has only handled single modifications. This is insufficient for our purposes, and we would like to be handle more than one modification at a time. If the underlying variety $X$ is a curve, there is a recursive definition which utilizes the curve-to-projective extension theorem \cite[I.6.8]{Hartshorne} and works in full generality (see \cref{S:modifications-curves}). In higher dimensions, one needs to be much more careful. The following notions formalize multi-modifications. Later, we will relate these to the recursive definition for curves.

\begin{definition}
  \label{T:tree-like}
  Let $\{ F_i \subset E \;|\; i \in I \}$ be a collection of subbundles. We will say that $\{ F_i \}$ is \emph{tree-like} at a point $x \in X$ if for all $I' \subset I$ either
  \begin{enumeratea}
  \item
    the set of subspaces $\{ F_i|_x \;|\; i \in I' \}$ is linearly independent in $E_x$, or
  \item
    there is a distinct pair $i, j \in I'$ and an open $U \subset X$ containing $x \in X$ such that $F_i|_U \subset F_j|_U$.
  \end{enumeratea}
  We will use $\TL_X(\{ F_i \})$ to denote the set of tree-like points in $X$. When there is no ambiguity, we may write $\TL(\{ F_i \}) = \TL_X(\{ F_i \})$. We will say that $\{ F_i \}$ is tree-like along $Y \subset X$ if $Y \subset \TL(\{ F_i \})$.
\end{definition}

\begin{remark}
  \label{T:tree-like-local-open}
  Note that being tree-like is a local property. Let $U \subset X$ be an open, and $x \in U$ is a point in it. Being local means that $\{ F_i \}$ is tree-like at $x$ if and only if $\{ F_i|_U \}$ is tree-like at $x$. More strongly, being tree-like is also preserved by pullbacks.

  On a similar note, since linear independence is an open property, then being tree-like is also open.
\end{remark}

\begin{remark}
  \label{T:tree-like-explanation}
  The definition of being tree-like is inspired by the following observation. Let $E$ be a vector bundle over a variety $X$, and $\{ F_i \subset E \}$ a collection of subbundles. In addition, we consider the inclusion graph of $\{ F_i \} \cup \{ E \}$. The collection $\{ F_i \}$ is tree-like over $X$ if and only if the following two conditions are satisfied:
  \begin{enumeratea}
  \item
    the inclusion graph is a tree, and
  \item
    the children of each node are linearly independent.
  \end{enumeratea}
\end{remark}

The definition of tree-like was crafted so we can transfer multiple subbundles and entire modification data through modifications, similarly to \cref{T:modification-subbundles-bijection}. To simplify the statement of the following result, set
\[
S^{\textrm{set}}(E, F, D) =
\{ \{ F_i \subset E \textrm{ subbundle} \} \;|\; \textrm{$\{ F \} \cup \{ F_i \}$ is tree-like along $D$} \}.
\]

\begin{proposition}
  \label{T:modification-tree-like-bijection}
  Let $F \subset E|_U$ be a subbundle and $D$ an effective divisor on $X$ whose support is contained in $U$. Then there is a bijection
  \begin{align*}
    \phi^{\textrm{set}} \cn S^{\textrm{set}}(E, F, D) & \longrightarrow S^{\textrm{set}}(E[D \to F], F, D) \\
    \{ F_i \} & \longmapsto \{ \phi(F_i) \}
  \end{align*}
  such that
  \begin{enumeratea}
  \item
    \label{T:modification-tree-like-bijection-a}
    $\phi^{\textrm{set}}$ is compatible with pullbacks, and
  \item
    \label{T:modification-tree-like-bijection-b}
    if $D = \emptyset$ and we identify $S^{\textrm{set}}(E, F, D)$ and $S^{\textrm{set}}(E[D \to F], F, D)$, then $\phi^{\textrm{set}}$ becomes the identity map.
  \end{enumeratea}
\end{proposition}

\begin{proof}
  As long as we show that $\phi^{\textrm{set}}$ is a well-defined bijection, then parts (\ref{T:modification-tree-like-bijection-a}) and (\ref{T:modification-tree-like-bijection-b}) follow from \cref{T:modification-subbundles-bijection}.

  Consider a collection of subbundles $\{ F_i \subset E \;|\; i \in I \}$ such that $\{ F \} \cup \{ F_i \}$ is tree-like along $D$. For convenience, set $F_0 = F$ and $\overline{I} = \{ 0 \} \sqcup I$. To verify that $\phi^{\textrm{set}}$ is well-defined, we first need to show that all $F_i$ are in $S(E, F, D)$, the domain of $\phi$. Fix an index $i$, and take $I' = \{ 0, i \}$. By the definition of being tree-like, we know that one of the following is true:
  \begin{enumerate}
  \item
    $F_i \subset F$ in a neighborhood of $D$,
  \item
    $F \subset F_i$ in a neighborhood of $D$, or
  \item
    $F$ and $F_i$ are linearly independent along $D$.
  \end{enumerate}
  These cases correspond to $S_1(E, F, D)$, $S_2(E, F, D)$, and $S_3(E, F, D)$ respectively, so $F_i \in S(E, F, D)$. Next, we need to show that $\{ F \} \cup \{ \phi(F_i) \}$ is tree-like along $D$. This follows immediately from the fact that $\phi$ respects inclusions and linear independence.

  We have demonstrated that $\phi^{\textrm{set}}$ is a well-defined map. To conclude our proof, we need to demonstrate it is a bijection. It suffices to note we can construct an inverse $(\phi^{\textrm{set}})^{-1}(\{ F_i' \}) = \{ \phi^{-1}(F_i') \}$.
\end{proof}

After establishing transfer for sets of subbundles, the next step in our bootstrapping program is to define modification data and show how to transfer them.

\begin{definition}
  \label{T:modification-datum}
  A \emph{modification datum for $E$} is an ordered collection of triples
  \[
  M = \{ (D_i, U_i, F_i) \;|\; i \in I \}
  \]
  such that for each $i$:
  \begin{enumeratea}
  \item
    $D_i$ is an effective Cartier divisor on $X$,
  \item
    $U_i \subset X$ is an open containing the support of $D_i$, and
  \item
    $F_i \subset E|_{U_i}$ is a subbundle
  \end{enumeratea}
  In addition, we will call a datum $M$ \emph{tree-like} if for all subsets $I' \subset I$, there is an inclusion
  \[
  \bigcap_{i \in I'} \Supp(D_i) \subset \TL_{U_{I'}}(\{ F_i|_{U_{I'}} \;|\; i \in I' \}),
  \]
  where $U_{I'} = \bigcap_{i \in I'} U_i$. Put differently, for all $x \in X$ the collection of subbundles $\{ F_i \;|\; x \in D_i \}$ is tree-like at $x$.
\end{definition}

To simplify the transfer statement for modification data, set
\[
S^{\textrm{md}}(E, F, D) = \{ M = \{ (D_i, U_i, F_i) \} \;|\; \{ (D, U, F) \} \cup M \textrm{ is a tree-like modification datum} \}.
\]

\begin{proposition}
  \label{T:modification-data-bijection}
  Let $F \subset E|_U$ be a subbundle and $D$ an effective divisor on $X$ whose support is contained in $U$. Then there is a bijection
  \begin{align*}
    \phi^{\textrm{md}} \cn S^{\textrm{md}}(E, F, D) & \longrightarrow S^{\textrm{md}}(E[D \to F], F, D) \\
    \{ (D_i, U_i, F_i) \} & \longmapsto \{ (D_i, U_i, \phi(F_i)) \}
  \end{align*}
  such that
  \begin{enumeratea}
  \item
    \label{T:modification-data-bijection-a}
    $\phi^{\textrm{md}}$ is compatible with pullbacks, and
  \item
    \label{T:modification-data-bijection-b}
    if $D = \emptyset$ and we identify $S^{\textrm{md}}(E, F, D)$ and $S^{\textrm{md}}(E[D \to F], F, D)$, then $\phi^{\textrm{md}}$ becomes the identity map.
  \end{enumeratea}
\end{proposition}

\begin{proof}
  Continuing our build-up, we will repeatedly refer to \cref{T:modification-tree-like-bijection} in this proof. First, parts (\ref{T:modification-data-bijection-a}) and (\ref{T:modification-data-bijection-b}) follow immediately once we establish that $\phi^{\textrm{md}}$ is a well-defined bijection.

  Fix an element $M = \{ (D_i, U_i, F_i) \;|\; i \in I \} \in S^{\textrm{md}}(E, F, D)$. As before, we set
  \[
  F_0 = F, \qquad
  U_0 = U, \qquad
  D_0 = D, \qquad
  \overline{I} = \{ 0 \} \cup I, \qquad
  \overline{M} = \{ (D, U, F) \} \cup M.
  \]
  Given a subset $I' \subset \overline{I}$, we know that the intersection $\bigcap_{i \in I'} D_i$ lies in the set of tree-like points
  \[
  V_{I'} = \TL_{U_{I'}}(\{ F_i \;|\; i \in I' \}).
  \]
  Applying \cref{T:modification-tree-like-bijection} to $\{ F_i|_{V_I'} \;|\; i \in I' \}$, we conclude that
  \[
  V_{I'} = \TL_{U_{I'}}(\{ \phi(F_i) \;|\; i \in I' \}).
  \]
  We have demonstrated that $\{ (D, U, F) \} \cup \{ (D_i, U_i, \phi(F_i)) \}$ is a tree-like modification datum, so $\phi^{\textrm{md}}(M) \in S^{\textrm{md}}(E[D \to F], F, D)$ and $\phi^{\textrm{md}}$ is a well-defined map. To see that it is a bijection, it suffices to note we can construct an inverse using $(\phi^{\textrm{set}})^{-1}$.
\end{proof}

We are now ready to provide a general definition of vector bundle modifications. The main idea is to recursively use the transfer of modification data (\cref{T:modification-data-bijection}).

\begin{definition}
  \label{T:modification-generic}
  Let $X$ be a variety, $E$ a vector bundle over $X$, and $M$ a tree-like modification datum for $E$. If $M$ is empty, then we define $E[\emptyset] = E$. On the other hand, if $M = \{(D, U, F)\} \cup M'$, then
  \[
  E[M] = E[D \to F][\phi^{\textrm{md}}(M')],
  \]
  where $\phi^{\textrm{md}} \cn S^{\textrm{md}}(E, F, D) \rarr S^{\textrm{md}}(E[D \to F], F, D)$ is the transfer map described in \cref{T:modification-data-bijection}. When
  \[
  M = \{ (D_1, U_1, F_1), \dots, (D_m, U_m, F_m) \},
  \]
  we will allow ourselves to write
  \[
  E[M] = E[D_1 \to F_1] \cdots [D_m \to F_m].
  \]
\end{definition}

After establishing the language of multi-modifications, we are ready to describe some of its basic properties. First, we note that modifications respect pullbacks. This is a direct consequence of \cref{T:modification-pullback}.

\begin{corollary}
  \label{T:multi-modification-pullback}
  Let $f \cn Y \rarr X$ be a morphism of varieties and $E$ a vector bundle on $X$. If $M = \{ (D_i, U_i, F_i) \}$ is a tree-like modification datum for $E$ such that $\bigcup_i \Supp(D_i)$ does not contain any component of the image of $f$, then the pullback datum
  \[
  f^\ast M = \{ (f^\ast D_i, f^{-1}(U_i), f^\ast F_i) \}
  \]
  is tree-like, and there is a natural isomorphism
  \[
  f^\ast E[M] \cong (f^\ast E)[f^\ast M].
  \]
\end{corollary}

Next, note that we defined a modification datum as an ordered collection of triples (see \cref{T:modification-datum}). While the order plays a crucial point in our formulation, it turns out to be irrelevant for the final result $E[M]$ as long as $M$ is a tree-like modification datum.

\begin{proposition}[Commuting modifications]
  \label{T:communting-modifications}
  Let $E$ be a vector bundle over a variety $X$, and $M$ a tree-like modification datum. If $M'$ is a datum obtained by reordering $M$, then there is a natural isomorphism $E[M] \cong E[M']$ compatible with pullbacks.
\end{proposition}

\begin{proof}
  Since any symmetric group is generated by transpositions, it suffices to consider the case
  \[
  M = \{ (D_1, U_1, F_1), (D_2, U_2, F_2) \}, \qquad
  M' = \{ (D_2, U_2, F_2), (D_1, U_1, F_1) \}.
  \]
  We also need to know that $\phi_{F_1, D_1}^{\textrm{md}} \circ \phi_{F_2, D_2}^{\textrm{md}} = \phi_{F_2, D_2}^{\textrm{md}} \circ \phi_{F_1, D_1}^{\textrm{md}}$ for the subset of the domain where this composition makes sense. If we assume there is an isomorphism $E[M] \cong E[M']$, this statement is automatically true if we pass to any the generic point. But subbundles which agree on all generic points must be the same, so this issue is resolved.
  
  We proceed by making several reductions. First, there is a natural isomorphism $E[M] \cong E[M']$ over $X \setminus (\Supp(D_1) \cap \Supp(D_2))$, so it suffices to focus on a neighborhood of $\Supp(D_1) \cap \Supp(D_2)$. Next, we can cover this locus by affine opens $U$ which fall in one of the following three categories: (1) $F_1|_U \subset F_2|_U$, (2) $F_2|_U \subset F_1|_U$, or (3) $F_1$ and $F_2$ are linearly independent over $U$. Since cases (1) and (2) are analogous, so we will demonstrate (1) and (3). For simplicity, we can also replace $X$ with $U$.

  Assume that $F_1 \subset F_2$. Since we are working over an affine space, there are splittings $F_2 = F_1 \oplus F_1'$ and $E = F_2 \oplus F_2'$. Then
  \begin{align*}
    E
    &= F_1 \oplus F_1' \oplus F_2', \\
    E[D_1 \to F_1]
    &= F_1 \oplus F_1'(-D_1) \oplus F_2'(-D_1), \textrm{and} \\
    E[D_2 \to F_2]
    &= F_1 \oplus F_1' \oplus F_2'(-D_2).
  \end{align*}
  Using these splittings, we can perform the second modification to arrive at
  \begin{align*}
    E[M]
    &= E[D_1 \to F_1][D_2 \to F_2] \\
    &= F_1 \oplus F_1'(-D_1) \oplus F_2'(- D_1 - D_2) \\
    &= E[D_2 \to F_2][D_1 \to F_1] \\
    &= E[M'].
  \end{align*}
  
  In the third case, we assume $F_1$ and $F_2$ are linearly independent which leads to a splitting $E = F_1 \oplus F_2 \oplus F$. A similar computation demonstrates that
  \begin{align*}
    E[D_1 \to F_1]
    &= F_1 \oplus F_2(-D_1) \oplus F(-D_1), \textrm{and} \\
    E[D_2 \to F_2]
    &= F_1(-D_2) \oplus F_2 \oplus F(-D_2),
  \end{align*}
  and
  \[
  E[M] =
  F_1(-D_2) \oplus F_2(-D_1) \oplus F(- D_1 - D_2) =
  E[M'].
  \qedhere
  \]
\end{proof}

\begin{proposition}[Commuting modifications and twists]
  \label{T:commuting-modifications-twists}
  Let $E$ be a vector bundle over a variety $X$, $F \subset E$ a subbundle, and $M = \{ (D_i, U_i, F_i) \}$ a tree-like modification datum. If $D$ is a Cartier divisor (not necessarily effective) and we define the datum $M(D) = \{ (D_i, U_i, F_i(D)) \}$ for $E(D)$, then $M(D)$ is tree-like and there is a natural isomorphism
  \[
  E[M](D) = E(D)[M(D)]
  \]
  compatible with pullbacks.
\end{proposition}

\begin{proof}
  To see that $M(D)$ is tree-like, it suffices to note that vector bundle inclusion and linear independence are preserved by twisting.
  
  First, assume we know the desired isomorphism exists for negative effective divisors. Given a divisor $D$, we can always decompose it as $D = D^+ - D^-$ where $D^+$ and $D^-$ are effective. Using the pair $E(D^+)$ and $M(D^+)$ with divisor $-D^+$, we deduce
  \begin{align*}
    E[M]
    &= E(D^+ - D^+)[M(D^+ - D^+)] \\
    &\cong E(D^+)[M(D^+)](-D^+).
  \end{align*}
  Next, we apply the same result for $E(D^+)$, $M(D^+)$ with divisor $-D^-$:
  \begin{align*}
    E[M](D)
    &= E[M](D^+ - D^-) \\
    &\cong E(D^+)[M(D^+)](-D^-) \\
    &\cong E(D^+ - D^-)[M(D^+ - D^-)] \\
    &= E(D)[M(D)].
  \end{align*}

  We are left to furnish an isomorphism in the case of negative effective divisors. For simplicity, replace $D$ with its negative, so it is effective. Let $U$ be a neighborhood of $\Supp D$. Note that if $M$ is a tree-like datum, then $M' = M \cup \{ (D, U, 0) \}$ is also tree-like. Since $E[D \rarr 0] \cong E(-D)$, then the associated morphism $\phi^{\textrm{md}}$ maps the datum $M$ to $M(-D)$. Commutativity implies
  \begin{align*}
    E[M](-D)
    &\cong E[M][D \rarr 0] \\
    &\cong E[M'] \\
    &\cong E[D \rarr 0][\phi^{\textrm{md}}(M)] \\\
    &\cong E(-D)[M(-D)],
  \end{align*}
  which concludes our argument.
\end{proof}

\begin{remark}
  When it is clear that $M$ is a modification datum for $E$, we will allow ourselves to write $M$ instead of $M(D)$. Then the statement of \cref{T:commuting-modifications-twists} becomes
  \[
  E[M](D) = E(D)[M],
  \]
  so we say that modifications and twists commute.
\end{remark}

If we focus on the case of two modifications with identical base divisors, there are two more results mentioning.

\begin{proposition}[Combining modifications]
  \label{T:combining-modifications}
  Let $E$ be a vector bundle over a variety $X$. Consider a tree-like modification datum $M = \{ (a D, U, F_1), (b D, U, F_2) \}$ for $E$, where $a$, $b$ is a pair of non-negative integers.
  \begin{enumeratea}
  \item
    \label{T:combining-modifications-a}
    If $F = F_1 = F_2$, then
    \[
    E[a D \to F][b D \to F] \cong E[(a+b)D \to F].
    \]
  \item
    \label{T:combining-modifications-b}
    If $F_1$, $F_2$ are linearly independent and $a = b = 1$, then
    \[
    E[D \to F_1][D \to F_2] \cong E[D \to F_1 + F_2](-D).
    \]
  \end{enumeratea}
  In addition, both isomorphisms are compatible with pullbacks.
\end{proposition}

\begin{proof}
  Following \cref{T:modification-affine}, we can assume $X$ is affine. For part (\ref{T:combining-modifications-a}), there is a splitting $E = F \oplus E/F$, and we compute
  \begin{align*}
    E[a D \to F][b D \to F]
    & \cong (F \oplus (E/F)(-a D))[b D \to F] \\
    & \cong F \oplus (E/F)(-a D)(-b D) \\
    & \cong F \oplus (E/F)(-(a+b) D) \\
    & \cong E[(a+b)D \to F].
  \end{align*}

  In part (\ref{T:combining-modifications-b}), consider a splitting $E = F_1 \oplus F_2 \oplus F_3$. Then
  \begin{align*}
    E[D \to F_1][D \to F_2]
    & \cong (F_1 \oplus F_2(-D) \oplus F_3(-D))[D \to F_2(-D)] \\
    & \cong F_1(-D) \oplus F_2(-D) \oplus F_3(-2 D) \\
    & \cong (F_1 \oplus F_2 \oplus F_3(-D))(-D) \\
    & \cong E[D \to F_1 + F_2](-D).
      \qedhere
  \end{align*}
\end{proof}


\section{Elementary modifications for curves}
\label{S:modifications-curves}

While \cref{S:modifications-arbitrary} introduces vector bundle modifications in a very general setting, the applications we have in mind use curves and families of curves. The present section will explain more concretely how modifications manifest themselves for curves, and provide several simple consequences.

A substantial part of bootstrapping the definition of multiple modifications consisted of transfer statements. It turns out that curves allow for a simpler transfer statement for subbundles which extends \cref{T:modification-subbundles-bijection}. In particular, this allows us to extend multi-modifications beyond tree-like data at the expense of sacrificing some of the properties we already established (e.g., commutativity).

To state our result, define
\[
\overline{S}(E) =
\{ G \subset E \textrm{ subbundles} \},
\]
where $E$ is a vector bundle over a curve $C$.

\begin{proposition}
  \label{T:curves-modification-subbundles-bijection}
  Let $E$ be a vector bundle over a curve $C$. Given a subbundle $F \subset E$ and a divisor $D$ whose support is contained in the smooth locus of $C$, there is a bijection
  \[
  \overline{\phi} \cn
  \overline{S}(E) \rarr
  \overline{S}(E[D \to F]),
  \]
  such that
  \begin{enumeratea}
  \item
    \label{T:curves-modification-subbundles-bijection-a}
    $\overline{\phi}|_{S(E, F, D)} = \phi$ where $S(E, F, D)$ and $\phi$ are as in \cref{T:modification-subbundles-bijection},
  \item
    \label{T:curves-modification-subbundles-bijection-b}
    $\overline{\phi}$ is compatible with pullbacks,
  \item
    \label{T:curves-modification-subbundles-bijection-c}
    given $G_1, G_2 \in \overline{S}(E)$, then $G_1 \subset G_2$ in a neighborhood of $D$ implies $\overline{\phi}(G_1) \subset \overline{\phi}(G_2)$ in a neighborhood of $D$, and
  \item
    \label{T:curves-modification-subbundles-bijection-d}
    if $D = \emptyset$, then $E[D \to F] \cong E$ induced an identification of $\overline{S}(E)$ and $\overline{S}(E[D \to F])$ such that $\overline{\phi}$ becomes the identity map.
  \end{enumeratea}
\end{proposition}

\begin{proof}
  We start by constructing the map $\overline{\phi}$. Given a subbundle $G \subset E$ of rank $r$, we can produce a section $\sigma$ of the Grassmannian bundle $\Gr(r, E)$ of $E$.
  \[
  \xymatrix{
    \Gr(r, E) \ar[d] \\
    C \ar@/_/[u]_-{\sigma}
  }
  \]
  The natural inclusion $E[D \to F]$ is an isomorphism over $U = C \setminus \Supp(D)$, so we also have an isomorphism $\Gr(r, E)|_U \cong \Gr(r, E[D \to F])|_U$. It follows that we can treat $\sigma|_U$ as a section of the second Grassmannian bundle over $U$. The curve-to-projective extension theorem \cite[I.6.8]{Hartshorne} implies there is a unique section $\sigma' \cn C \rarr \Gr(r, E[D \to F])$ which extends $\sigma|_U$. The new section gives rise to a subbundle $\overline{\phi}(G) = G' \subset E[D \to F]$.

  For part (\ref{T:curves-modification-subbundles-bijection-a}), start by picking a bundle $G \in S(E, F, D)$. If we identify $E|_U$ and $E[D \to F]|_U$, then $\overline{\phi}(G)|_U = G|_U = \phi(G)|_U$. Since both $\phi(G)$ and $\overline{\phi}(G)$ are subbundles, and $U \subset C$ is dense, it follows that $\phi(G) = \overline{\phi}(G)$.

  Note that it makes sense to consider the pullback by a morphism $f \cn C' \rarr C$ only if the pullback divisor $f^\ast D$ is well-defined. This happens exactly when no component of $C'$ is contracted to a point which lies in the support of the divisor $D$ on $C$ (see \cref{T:modification-pullback}). In particular, the condition is always satisfied for finite morphisms $f$. Once we understand this limitation, running through the section extension definition of $\overline{\phi}$, it is clear that $\overline{\phi}$ is compatible with pullbacks.

  Finally, the proofs of (\ref{T:curves-modification-subbundles-bijection-c}) and (\ref{T:curves-modification-subbundles-bijection-d}) are identical to the arguments we gave in \cref{T:modification-subbundles-bijection}.
\end{proof}

\begin{remark}
  \label{T:curves-modification-subbundles-bijection-linear-independence}
  Note that $\overline{\phi}$ satisfies all properties $\phi$ does except it does not preserve linear dependence and independence. To illustrate the point, take $C = \Abb^1$ with a coordinate $x$ on it, $p = 0$ is the origin, and $E = \Oc_C \oplus \Oc_C$. Set
  \begin{align*}
    F & = \langle (1, 1) \rangle, &
                                    G_1 &= \langle (1, 0) \rangle, &
                                                                     G_2 &= \langle (0, 1) \rangle, \\
    F' & = \langle (1, 0) \rangle, &
                                     G_1' &= \langle (1, x) \rangle, &
                                                                       G_2' &= \langle (1, -x) \rangle.  
  \end{align*}
  Then $G_1$ and $G_2$ are linearly independent in $E$, while $\overline{\phi}(G_1)$ and $\overline{\phi}(G_2)$ coincide over $p$ in $E[p \to F]$. On the other hand, $G_1'$ and $G_2'$ are linearly dependent at $p$, but their transfers $\overline{\phi}(G_1'), \overline{\phi}(G_2') \subset E[p \to F']$ are linearly independent at $p$.

  In summary, it is possible to modify curves along modification data which are not tree-like, but we need to be careful about switching the order of modifications. Unless otherwise stated, all modifications will be tree-like.
\end{remark}

    Finally, we present a result which relates the Euler characteristics of a modified bundle and the original one.

    \begin{proposition}[The Euler characteristic of modifications]
      \label{T:chi-modifications}
      Let $E$ be a vector bundle over a curve $C$.
      \begin{enumeratea}
      \item
        \label{T:chi-modifications-a}
        If $D_1, \dots, D_m$ are effective divisors, and $F_1, \dots, F_m \subset E$ are subbundles, then
        \[
        \chi(E[D_1 \to F_1] \cdots [D_m \to F_m]) =
        \chi(E) - \sum_{i=1}^m \deg(D_i) \rank(E/F_i).
        \]
      \item
        \label{T:chi-modifications-b}
        If $D$ is a any divisor, then
        \[
        \chi(E(D)) = \chi(E) + \rank(E) \deg(D).
        \]
      \end{enumeratea}
    \end{proposition}

    \begin{proof}
      Note that the general statement of part (\ref{T:chi-modifications-a}) follows by applying the $m = 1$ case several times. When $m = 1$, we take Euler characteristics of the sequence \labelcref{E:modification-ses} and note that
      \[
      \chi((E/F_1)|_{D_1}) =
      \deg(D_1) \rank(E/F_1).
      \]
      
      Similarly to the proof of \cref{T:commuting-modifications-twists}, we can reduce (\ref{T:chi-modifications-b}) to the case of a negative effective divisor which is subsumed by part (\ref{T:chi-modifications-a}).
    \end{proof}

\begin{remark}
  \label{T:modifications-in-families}
  The theory of modifications over general varieties (\cref{S:modifications-arbitrary}) is certainly more complicated than the statements we presented for curves. Dimensions greater than one become very useful when we deal with families of curves and vector bundles. The fact that constructing modifications preserve pullbacks allows us to treat a modification over the total space of a family of curves as a family of modifications over the individual curves.

  We will demonstrate this point through a simple example. Let $C$ be a smooth curve, $E$ a vector bundle over $C$, and $F \subset E$ a subbundle. We consider the family of curves
  \[
  \pr_2 \cn \Cc = C \x B \rarr B
  \]
  where $B = C$. Given $b \in B$, we will use $i_b \cn C \rarr \Cc$ to denote the inclusion of the fiber over the point $b$. Choose a point $p_0 \in C$, and construct the divisors
  \[
  D_0 = \{ p_0 \} \x B, \qquad
  D_1 = \Delta_C, \qquad
  D = D_0 + D_1.
  \]
  If
  \[
  E' = (\pr_1^\ast E)[D \to \pr_1^\ast F]
  \]
  is the global modification, then restricting to a fiber over $b$ gives
  \[
  i_b^\ast E' =
  E[i_b^\ast D \to F] =
  E[(b + p_0) \to F].
  \]
  This shows that varying the modification divisor in a family produces modifications which also fit in a family. Furthermore, we know that $E[2 p_0 \to F]$ is the ``limit modification'' as $b$ approaches $p_0$. This is a very simple example to illustrate the power of modifications over higher dimensional varieties. In general, understanding limits of multiple modifications can be very tricky and being tree-like is the right condition to back our intuitive notion of limits.  
\end{remark}


\section{Interpolation and short exact sequences}
\label{S:sequences}

The goal of this section is to define interpolation for vector bundles and develop some of its properties, in particular its behavior in short exact sequences. For a more detailed explanation of this property, see \cite{Atanasov-interpolation}.

\begin{definition}
  \label{T:interpolation}
  Let $E$ be a rank $n$ vector bundle over a curve $C$.
  We say that a subspace of sections
  $V \subseteq \H^0(E)$ satisfies interpolation if $E$ is nonspecial, and for every $d \geq 1$, there exists a collection of $d$ points $p_1, \dots, p_d \in C_\sm$ such that
  \[
  \dim \left( V \cap \H^0\left( E\left(-\sum p_i\right) \right)\right) = \max\{0, \dim V - d n \}.
  \]
  We say that $E$ satisfies interpolation if
  the full space of sections $V = \H^0(E) \subseteq \H^0(E)$ satisfies
  interpolation.
\end{definition}

There are a number of observations which allow us to verify interpolation more easily.

\begin{remark}
  \label{T:interpolation-general}
  By the upper semi-continuity of $\h^0$, the existence of $d$ points satisfying the equality above implies that a general collection of $d$ points (in one component of $C_\sm^d$) satisfies this condition.
\end{remark}

    \begin{remark}
      \label{T:interpolation-two-d}
      In fact, we do not need to check the interpolation condition for every positive integer $d$. It suffices to verify that the statement holds for $\lfloor \h^0(E)/n \rfloor$ and $\lceil \h^0(E)/n \rceil$. The first value implies the statement holds for all $d \leq \lfloor \h^0(E)/n \rfloor$ and the second for all values $d \geq \lceil \h^0(E)/n \rceil$.

      We have arrived at a convenient rephrasing of \cref{T:interpolation}. Let $\h^0(E) = n \cdot d + r$ where $0 \leq r < n$. Consider the following two statements.
      \begin{enumeratea}
      \item
        \label{T:interpolation-two-d-a}
        There exist points $p_1, \dots, p_d \in C_\sm$ such that
        \[
        \h^0(E(-\sum p_i)) = r.
        \]
      \item
        \label{T:interpolation-two-d-b}
        There exist points $p_1, \dots, p_{d+1} \in C_\sm$ such that
        \[
        \h^0(E(-\sum p_i)) = 0.
        \]
      \end{enumeratea}
      Assume $E$ has no higher cohomology. If $r = 0$, then interpolation for $E$ is equivalent to (\ref{T:interpolation-two-d-a}). In the cases when $r > 0$, interpolation is equivalent to (\ref{T:interpolation-two-d-a}) and (\ref{T:interpolation-two-d-b}) together.
    \end{remark}

\begin{remark}
  \label{T:interpolation-h0-h1-remark}
  It is also possible to use the language of divisors to characterize interpolation. Consider a vector bundle $E \to C$ satisfying interpolation. Given an integer $d \geq 1$,
  there is a component of $\Sym^d C$ so that a general effective divisor $D$ in that component satisfies either $\h^1(E(-D)) = 0$ (when $\deg D \leq \h^0(E)/\rank(E)$) or $\h^0(E(-D)) = 0$ (when $\deg D \geq \h^0(E)/\rank(E)$). Conversely, if for all $d$ there is some component of $\Sym^d C$ for which this disjunction holds, then we can deduce interpolation. We have arrived at the following restatement of \cref{T:interpolation}.

  \begin{proposition}
    \label{T:interpolation-h0-h1}
    A nonspecial vector bundle $E \to C$ satisfies interpolation if and only if for every $d \geq 1$, there is a component of $\Sym^d C$ so that a general effective Cartier divisor $D$ of degree $d$ in that component satisfies
    \[
    \h^0(E(-D)) = 0
    \qquad\textrm{or}\qquad
    \h^1(E(-D)) = 0.
    \]
  \end{proposition}

  There is a further simplification worth mentioning. Note that we do not need to verify the vector bundle is nonspecial before applying this result.

  \begin{proposition}
    \label{T:interpolation-h0-h1-two}
    A vector bundle $E$ of rank $n$ satisfies interpolation if and only if
    \begin{enumeratea}
    \item
      a general (in some component) effective divisor $D$ of degree $\lceil \h^0(E) / n \rceil$ satisfies $\h^0(E(-D)) = 0$, and
    \item
      a general (in some component) effective divisor $D$ of degree $\lfloor \h^0(E) / n \rfloor$ satisfies $\h^1(E(-D)) = 0$.
    \end{enumeratea}
    Furthermore, if $\chi(E) \geq 0$, we can replace $\h^0(E)$ with $\chi(E)$ in $\lceil \h^0(E) / n \rceil$ and $\lfloor \h^0(E) / n \rfloor$.
  \end{proposition}

  \begin{proof}
        To conclude that $E$ is nonspecial, we note that $\h^1(E(-D)) = 0$ for some effective divisor of non-negative degree $\lfloor \h^0(E) / n \rfloor$. The first part is a direct consequence of \cref{T:interpolation-h0-h1} and \cref{T:interpolation-two-d}. For the second part, it suffices to note the same argument implies that $\h^1(E) = 0$ as long as $\chi(E) \geq 0$.
  \end{proof}
\end{remark}

Characterizing line bundles which satisfy interpolation is particularly simple and worth elaborating on.

\begin{proposition}
  \label{T:interpolation-line-bundle}
  A line bundle satisfies interpolation if and only if it is nonspecial.
\end{proposition}

\begin{proof}
  One direction is implied by the definition of interpolation. For the converse, consider a nonspecial line bundle $L$. We proceed to choose $m = \h^0(L)$ points $p_i \in C_\sm$ as follows. First, pick $p_1$ such that $\h^0(L(-p_1)) = \h^0(L) - 1$. If $m \geq 2$, we choose a second point $p_2$ such that $\h^0(L(-p_1-p_2)) = \h^0(L) - 2$, and so on. This demonstrates that $L$ satisfies interpolation.
\end{proof}

Next we show interpolation is preserved
by modifications along appropriately general subbundles,
and by positive twists.
To provide the precise statement, we need to introduce the following notion.

\begin{definition}
  \label{T:linearly-general}
  Let $V$ be a vector space, and $\{ W_b \subset V \;|\; b \in B \}$ be a collection of subspaces indexed by a set $B$. We will call $\{ W_b \}$ \emph{linearly general} if for each subspace $W \subset V$, there exists $b \in B$ such that $W_b$ and $W$ intersect transversely.
\end{definition}

    \begin{remark}
      \label{T:linearly-general-complimentary}
      Suppose the ambient vector space has dimension $n$, and all members of the collection $\{ W_b \}$ have dimension $m$. Then to conclude that $\{ W_b \}$ is linearly general, it suffices to know that for all subspaces $W \subset V$ of complimentary dimension $n - m$ there exists $b \in B$ such that $W \cap W_b = 0$.
    \end{remark}

\begin{proposition}
  \label{T:linearly-general-interpolation}
  Let $E$ be a vector bundle over a curve $C$ and $p \in C_\sm$ a smooth point. Suppose we have a collection of vector bundles $\{ G_b \subset E \;|\; b \in B \}$ indexed by a set $B$ and $F \subset E$ is a subbundle, such that
  \begin{enumeratea}
  \item
    $F|_p \subset G_b|_p$ for all $b \in B$, and
  \item
    $\{ G_b/F|_p \;|\; b \in B \}$ is linearly general in $E/F|_p$.
  \end{enumeratea}
  If $E$ and $E[p \to F]$ both satisfy interpolation, then $E[p \to G_b]$ satisfies interpolation for at least one element $b \in B$.
\end{proposition}

\begin{proof}
  We can assemble two copies of sequence \labelcref{E:modification-ses} into the following diagram with exact rows and columns.
      \[\xymatrix{
        &&& 0 \ar[d] \\
        & 0 \ar[d] && G_b/F|_p \ar[d] \\
        0 \ar[r] & E[p \to F] \ar[r] \ar[d] & E \ar[r] \ar@{=}[d] & E/F|_p \ar[r] \ar[d] & 0 \\
        0 \ar[r] & E[p \to G_b] \ar[r] \ar[d] & E \ar[r] & E/G_b|_p \ar[r] \ar[d] & 0 \\
        & G_b/F|_p \ar[d] && 0 \\
        & 0
      }\]
  Given a divisor $D$, we twist the entire sequence by $-D$ and take cohomology. For $A = E/F$, $E/G_b$, and $G_b/F$, there are induced isomorphisms $A(-D)|_p \cong A|_p$, so we will use the latter. We will also avoid the $\H^0$-functor in front of skyscraper sheaves supported on a point. The operation we described leads to the following diagram.
      \[\xymatrix@C=0.13in{
        &&& G_b/F|_p \ar[d] \\
        0 \ar[r] & \H^0(E[p \to F](-D)) \ar[r] \ar[d] & \H^0(E(-D)) \ar[r] \ar@{=}[d] & E/F|_p \ar[r] \ar[d] & \H^1(E[p \to F](-D)) \ar[r] \ar@{->>}[d] & \H^1(E(-D)) \ar[r] \ar@{=}[d] & 0 \\
        0 \ar[r] & \H^0(E[p \to G_b](-D)) \ar[r] & \H^0(E(-D)) \ar[r] & E/G_b|_p \ar[r] & \H^1(E[p \to G_b](-D)) \ar[r] & \H^1(E(-D)) \ar[r] & 0    
      }\]

  Now that we have described the basic tools we need, we can proceed with the proof. For each $d \geq 1$, choose a divisor $D_d$ of degree $d$ such that
  \[
  \h^0(A(-D_d)) = 0
  \qquad\textrm{or}\qquad
  \h^1(A(-D_d)) = 0
  \]
  for $A = E$ and $A = E[p \to F]$ (\cref{T:interpolation-h0-h1}). Note that each value $d$ falls in one of three cases:
  \begin{enumerate}
  \item
    \label{T:linearly-general-interpolation-1}
    $\h^0(E(-D_d)) = 0$,
  \item
    \label{T:linearly-general-interpolation-2}
    $\h^1(E(-D_d)) = 0$ and $\h^1(E[p \to F](-D_d)) = 0$, or
  \item
    \label{T:linearly-general-interpolation-3}
    $\h^1(E(-D_d)) = 0$ and $\h^0(E[p \to F](-D_d)) = 0$.
  \end{enumerate}
  With the aid of the diagram above, case \ref{T:linearly-general-interpolation-1} implies $\h^0(E[p \to G_b](-D_d)) = 0$, and case \ref{T:linearly-general-interpolation-2} implies $\h^0(E[p \to G_b](-D_d)) = 0$.

  Note that our argument so far works for all $b \in B$. The handling of case \ref{T:linearly-general-interpolation-3} requires a choice of $b$. Fortunately, there can be at most one value of $d$ which satisfies this case. First observe that $\H^0(E(-D_d)) \rarr E/F|_p$ is an inclusion, and we choose $b \in B$ so that $G_b/F|_p$ is transverse to $\H^0(E(-D))$. It follows that the composition $\H^0(E(-D)) \rarr E/F|_p \rarr E/G_b|_p$ has maximal rank. Injectivity and surjectivity respectively imply $\h^0(E[p \to G_b](-D_d)) = 0$ and $\h^1(E[p \to G_b](-D_d)) = 0$.
\end{proof}

\begin{proposition}
  \label{T:interpolation-twist-up}
  If $E$ satisfies interpolation, and $D$ is any effective Cartier divisor, then $E(D)$ satisfies interpolation.
\end{proposition}

\begin{proof}
  We need to show that for every degree $d$, there exists a divisor $D'$ of degree $d$ such that either $\h^0(E(D - D')) = 0$ or $\h^1(E(D - D')) = 0$. If $d > \deg D$, take $D' = D + D''$ such that $\h^0(E(-D'') = 0$ or $\h^1(E(-D'')) = 0$ from the interpolation of $E$. If $d = \deg D$, take $D' = D$ and note that $E$ is nonspecial. Since $\h^1(E(D - D')) = 0$ is an open condition in $D'$, it follows that there exists some $D' = D_0$ supported on $C_\sm$ such that $\h^1(E(D - D_0)) = 0$. The interesting case is $d < \deg D$. If we choose an effective divisor $D' \leq D_0$, then $\h^1(E(D - D')) = 0$ follows from $\h^1(E(D - D_0)) = 0$.
\end{proof}

We now study several strengthenings and partial converses to the above results,
subject to additional hypotheses --- \emph{including the irreducibility of $C$,
  which we suppose for the remainder of this section}.

We have already investigated interpolation and twisting up (see \cref{T:interpolation-twist-up}). The following result provides a partial converse; note that the base curve $C$ needs to be irreducible and $\chi(E)$ is relatively large.

\begin{proposition}
  \label{T:interpolation-twist-down}
  Let $E$ be a vector bundle on an irreducible curve $C$, and $D$ an effective divisor on $C$. If
  \begin{enumeratea}
  \item
    $E(D)$ satisfies interpolation, and
  \item
    $\chi(E) \geq \genus(C) \rank(E)$,
  \end{enumeratea}
  then $E$ also satisfies interpolation.
\end{proposition}

\begin{proof}
  Since interpolation is an open condition, we may replace $D$ by a divisor supported on the smooth locus of $C$.

  By \cref{T:interpolation-h0-h1-two}, we only need to show that $\h^1(E(-D')) = 0$ for a general divisor $D'$ of degree $\lfloor \chi(E)/\rank(E) \rfloor$, and $\h^0(E(-D')) = 0$ for general $D'$ of degree $\lceil \chi(E)/\rank(E) \rceil$. Since the arguments are analogous, we will focus on the first case.

  For convenience, set $d = \lfloor \chi(E)/\rank(E) \rfloor$ and $g = \genus(C)$. Since $d \geq g$, the Riemann-Roch theorem implies that the natural map $\Sym^d C \rarr \Pic^d C$ is dominant; hence, it suffices to show that $\h^1(E \otimes L^{\vee}) = 0$ for a line bundle $L$ of degree $d$. Since $E(D)$ satisfies interpolation, we know that there exists a divisor $D''$ of degree $d+ \deg(D)$ such that $\h^1(E(D - D'')) = 0$. Taking $L = \Oc_C(D'' - D)$ completes the argument.
\end{proof}

    \begin{remark}
      \label{T:interpolation-twist-down-smooth}
      Suppose we have a family of curves $\pi \cn \Cc \rarr B$ whose central fiber $\Cc_0 = \pi^{-1}(0)$ is reducible but the general fiber is irreducible. If $\Ec$ is a vector bundle on $\Cc$ whose restriction $\Ec_0$ to $\Cc_0$ satisfies the hypotheses of \cref{T:interpolation-twist-down}, then the general fiber $\Cc_b = \pi^{-1}(b)$ together with $\Ec_b = \Ec|_{\Cc_b}$ also satisfy these conditions. While we cannot conclude that $\Ec_0$ satisfies interpolation, the general bundle $\Ec_b$ does satisfy interpolation. This will be sufficient for our needs in this paper.
    \end{remark}

To study the behavior of interpolation under modifications without any assumption of linear generality,
we will need to investigate
the behavior of interpolation in exact sequences\versionedtext{, and introduce the notion of positive modifications}{}.
We begin by considering a short exact sequence
\begin{equation}
  \label{E:ses-FGH}
  \vcenter{\xymatrix{
      0 \ar[r] &
      F \ar[r] &
      G \ar[r] &
      H \ar[r] &
      0
    }}
\end{equation}
of vector bundles over an irreducible curve $C$. Given a Cartier divisor $D$, we twist back by $D$ and consider associated the long exact sequence in cohomology.
\begin{equation}
  \label{E:les-delta}
  \vcenter{\xymatrix@R=0.1in{
      0 \ar[r] &
      \H^0(F(-D)) \ar[r] &
      \H^0(G(-D)) \ar[r] &
      \H^0(H(-D)) \ar[r]^-{\delta_D} & \\
      \ar[r]^-{\delta_D} &
      \H^1(F(-D)) \ar[r] &
      \H^1(G(-D)) \ar[r] &
      \H^1(H(-D)) \ar[r] &
      0
    }}
\end{equation}
We will use $\delta_D \cn \H^0(H(-D)) \rarr \H^1(F(-D))$ to denote the only non-trivial connecting homomorphism. Our first result allows us to transfer interpolation from the edges $F$ and $H$ to the middle term $G$.

\begin{proposition}
  \label{T:interpolation-ses}
  Let $F$, $G$, and $H$ be as above. If $F$ and $H$ satisfy interpolation, then $G$ satisfies interpolation if and only if
  \begin{enumeratea}
  \item
    \label{T:interpolation-ses-a}
    $\h^0(F)/\rank(F) \leq \lfloor \h^0(H)/\rank(H) \rfloor + 1$, and
  \item
    \label{T:interpolation-ses-b}
    for every $d \geq 1$ and a general effective divisor $D$ of degree $d$, the boundary map $\delta_D$ has maximal rank (i.e., it is either injective or surjective).
  \end{enumeratea}
\end{proposition}

\begin{proof}
  First, assume that $G$ satisfies interpolation. By \cref{T:interpolation-h0-h1}, this means that for a general effective $D$, either $\h^0(G(-D)) = 0$ or $\h^1(G(-D)) = 0$. Using sequence \labelcref{E:les-delta}, the first case implies that $\h^0(F(-D)) = 0$ and $\delta_D$ is injective, and the second that $\h^1(H(-D)) = 0$ and $\delta_D$ is surjective. In particular, we have demonstrated condition (\ref{T:interpolation-ses-b}) stating that $\delta_D$ has maximal rank.

  Condition (\ref{T:interpolation-ses-a}) is a little more interesting. Its negative asserts there exists an integer $d$ such that
  \begin{equation} \label{no-integer-inbetween}
    \frac{\h^0(H)}{\rank(H)} < d < \frac{\h^0(F)}{\rank(F)}.
  \end{equation}
  For contradiction, assume such an integer $d$ exists. Let $D$ be a general effective divisor of degree $d$ and consider the associated sequence \labelcref{E:les-delta}. The first side of the inequality implies
  \[
  \h^0(H(-D)) = 0
  \qquad\textrm{and}\qquad
  \h^1(H(-D)) > 0,
  \]
  while the second side implies
  \[
  \h^0(F(-D)) > 0
  \qquad\textrm{and}\qquad
  \h^1(F(-D)) = 0.
  \]
  Then
  \[
  \h^0(G(-D)) = \h^0(F(-D)) > 0
  \qquad\textrm{and}\qquad
  \h^1(G(-D)) = \h^1(H(-D)) > 0,
  \]
  which contradicts our hypothesis that $G$ satisfies interpolation by \cref{T:interpolation-h0-h1}.
  
  The reverse direction also follows by inspecting sequence \labelcref{E:les-delta}. Let us pick an integer $d \geq 1$ and consider a general effective divisor $D$ of degree $d$. We are given that $F$ and $H$ satisfy interpolation, so the argument can be split in four cases.
  \begin{description}
  \item[Case 1, $\h^0(F(-D)) = 0$ and $\h^0(H(-D)) = 0$.] ${}$ \\
    Since $\H^0(G(-D))$ sits between two zeros, it must also be zero.

  \item[Case 2, $\h^1(F(-D)) = 0$ and $\h^1(H(-D)) = 0$.] ${}$ \\
    Since $\H^1(G(-D))$ sits between two zeros, it must also be zero.

  \item[Case 3, $\h^0(F(-D)) = 0$ and $\h^1(H(-D)) = 0$.] ${}$ \\
    If $\delta_D$ is injective, then $\h^0(G(-D)) = 0$. Otherwise, if $\delta_D$ is surjective, then $\h^1(G(-D)) = 0$.

  \item[Case 4, $\h^1(F(-D)) = 0$ and $\h^0(H(-D)) = 0$.] ${}$ \\
    If $\h^0(F(-D)) > 0$ and $\h^1(H(-D)) > 0$, then $d$ satisfies \cref{no-integer-inbetween}, so condition (\ref{T:interpolation-ses-a}) is violated. If either of these is zero, we fall back to one of the first three cases.
  \end{description}

  Finally, $G$ is nonspecial since $\H^1(G)$ sits between $\H^1(F) = 0$ and $\H^1(H) = 0$ in sequence \labelcref{E:les-delta} for $D = 0$. This proves that $G$ satisfies interpolation, so the converse implication is complete.
\end{proof}

\begin{remark}
  \label{T:interpolation-ses-irreducibility}
  The forward direction of \cref{T:interpolation-ses} holds without the irreducibility hypothesis on $C$, that is, if $G$ satisfies interpolation, then statements (\labelcref{T:interpolation-ses-a}) and (\labelcref{T:interpolation-ses-b}) are true.

  To construct a counterexample for the converse, consider the curve $C$ obtained by gluing two rational components $C_1$ and $C_2$ in a single point. Let $F$ be the line bundle obtained by gluing $\Oc_{C_1}$ and $\Oc_{C_2}(2)$, and let $H$ be the line bundle obtained by gluing $\Oc_{C_1}(2)$ and $\Oc_{C_1}$. Next, we will take $G = F \oplus H$. Both $F$ and $H$ are nonspecial line bundles, so they satisfy interpolation. Condition (\labelcref{T:interpolation-ses-a}) is automatically satisfied since $\h^0(F) = \h^0(H)$ by the symmetry between $F$ and $H$. Since $G$ is the direct sum of $F$ and $H$, it follows that all boundary maps $\delta_D$ are zero. To show they are of maximal rank, we need to know that either the source $\h^0(H(-D))$ or the target $\h^1(F(-D))$ is zero. Again, this is true by the symmetry between $F$ and $H$ and the fact both of them satisfy interpolation. Finally, to see that $G$ does not satisfy interpolation note that there exists no degree $3$ divisor $D$ such that $\h^0(G(-D)) = 0$ or $\h^1(G(-D)) = 0$.
\end{remark}

Specializing $F$ to a line bundle yields the following useful result.

\begin{corollary}
  \label{T:interpolation-ses-line-bundle}
  Let $F$, $G$, and $H$ be as above, and $F$ is a nonspecial line bundle. If $H$ satisfies interpolation, then $G$ satisfies interpolation if and only if
  \begin{enumeratea}
  \item
    \label{T:interpolation-ses-line-bundle-a}
    $\rank(H) (\h^0(F) - 1) \leq \h^0(H)$, and
  \item
    \label{T:interpolation-ses-line-bundle-b}
    for every $d \geq 1$ and a general divisor $D$ of degree $d$, the boundary map $\delta_D$ has maximal rank.
  \end{enumeratea}
\end{corollary}

\begin{proof}
  Other than simplifying the inequality in condition (\ref{T:interpolation-ses-line-bundle-a}), this result follows by noting that a line bundle satisfies interpolation if and only if it is nonspecial (\cref{T:interpolation-line-bundle}).
\end{proof}

Let us return to the short exact sequence of vector bundles \labelcref{E:ses-FGH}. \versionedtext{Defining a positive modification of $G$ takes more effort than the (negative) elementary modifications we have been working with since \cref{S:modifications-arbitrary}. Without introducing any new notation, we will construct positive modification at $p \in C_\sm$ first by twisting up to arrive at $G(np)$, and then applying an elementary modification to get
    \[
    G(np)[ n p \to F ] = G(np)[ n p \to F(np) ].
    \]
    We could have also started with the elementary modification $G[n p \to F]$ and then twisted up to obtain $G[n p \to F](np)$. The two results are naturally isomorphic by \cref{T:commuting-modifications-twists}, so we will avoid stressing the distinction for the sake of convenience.

    The reason we call $G(np)[n p \to F]$ a positive modification is the existence of a natural morphism $G \rarr G(np)[n p \to F]$. To construct this map, start by observing that both $G$ and $G(np)[n p \to F]$ admit injective maps into $G(np)$. The cokernel of the latter $G(np) \rarr H(np)_{np}$ factors through the cokernel of the former $G(np) \rarr G(np)_{np}$, so we have an inclusion $G \rarr G(np)[n p \to F]$.
    \[\xymatrix{
      & G \ar[d] \ar@{-->}[dl] \\
      G(np)[n p \to F] \ar[r] & G(np) \ar[r] \ar[d] & H(np)_{np} \\
      & G(np)_{np} \ar[ur]
    }\]
    A very similar argument shows that there is a natural inclusion $F(np) \rarr G(np)[n p \to F]$. The Snake Lemma provides an isomorphism between the cokernel of this morphism and
    \[
    H = \Ker(H(np) \rarr H(np)_{np}).
    \]
    The following diagram with exact rows summarizes our observations.}{
  Note that we have the following diagram with exact rows:}
\begin{equation}
  \label{E:ses-modification}
  \vcenter{\xymatrix{
      0 \ar[r] &
      F \ar[r] \ar[d] &
      G \ar[r] \ar[d] &
      H \ar[r] \ar@{=}[d] &
      0 \\
      0 \ar[r] &
      F(np) \ar[r] &
      G(np)[n p \to F] \ar[r] &
      H \ar[r] &
      0
    }}
\end{equation}

\versionedtext{
      If the inclusion $F \rarr G$ splits (e.g., if we work in an affine neighborhood of $p$), then the positive modification is
      \[
      G(np)[n p \to F] =
      F(np) \oplus H.
      \]
      The existence and exactness of diagram \labelcref{E:ses-modification} become immediate.
    }{}

The following proposition (when combined with Proposition~\ref{T:interpolation-twist-down} to remove the positive twist)
gives the promised result on modifications along line subbundles without any linear generality assumption.

\begin{proposition}
  \label{T:interpolation-positive}
  Consider diagram \labelcref{E:ses-modification}. If
  \begin{enumeratea}
  \item
    $F$, $G$, and $H$ satisfy interpolation,
  \item
    $F$ is a line bundle,
  \item
    the point $p \in C_{\sm}$ is general, and
  \item
    $\rank(H) (\h^0(F) + n - 1) \leq \h^0(H)$,
  \end{enumeratea}
  then $G(np)[n p \to F]$ satisfies interpolation.
\end{proposition}

\begin{proof}
  Both $F(np)$ and $H$ satisfy interpolation (for the first, we apply \cref{T:interpolation-twist-up}), and
  \[
  \rank(H) (\h^0(F(np)) - 1) =
  \rank(H) (\h^0(F) + n - 1) \leq
  \h^0(H).
  \]
  To apply \cref{T:interpolation-ses-line-bundle} and conclude that $G(np)[n p \to F]$ satisfies interpolation, we need to verify that the connecting homomorphism $\delta_D' \cn \H^0(H(-D)) \rarr \H^1(F(-D + np))$ has maximal rank for $D$ a general divisor of degree $d$ and every $d \geq 1$. On the other hand, we can present $\delta_D'$ as a composition using the connecting homomorphism $\delta_D \cn \H^0(H(-D)) \rarr \H^1(F(-D))$ which has maximal rank (\cref{T:interpolation-ses-line-bundle}).
  \[\xymatrix{
    \H^0(H(-D)) \ar[r]^-{\delta_D} \ar@{=}[d] & H^1(F(-D)) \ar[d]^-{\pi} \\
    \H^0(H(-D)) \ar[r]^-{\delta_D'} & \H^1(F(-D + np))
  }\]
  Since the cokernel of $F(-D) \rarr F(-D + np)$ is supported at $p$, it has no higher cohomology, so the morphism $\pi \cn \H^1(F(-D)) \rarr \H^1(F(-D + np))$ is surjective. If $\delta_D$ is surjective, then $\delta_D'$ is automatically surjective.

  The case when $\delta_D$ is injective requires a little more work. Note that the image $V$ of $\delta_D$ is independent of the point $p$. Therefore, it suffices to show that the restriction of $\pi$ to an arbitrary fixed subspace $V$ has maximal rank.

  Set $L = K_C \otimes F(-D)^\vee$, where $K_C$ is the dualizing line bundle (which exists since $C$ is lci). The dual problem asks whether the image of the natural inclusion $\H^0(L(-np)) \rarr \H^0(L)$ intersects an arbitrary fixed space $V \subset \H^0(L)$ transversely. Since the inclusion has codimension $n$, by shrinking or enlarging $V$, it suffices to answer this question when $\dim V = n$. In turn, this is equivalent to the non-vanishing of the Wronskian associated to $V$ \cite{Wronskian}.
  (Note that this final step requires our assumption that $K$ has characteristic~$0$.)
\end{proof}

    \begin{remark}
      \label{T:interpolation-positive-higher-rank}
      It is natural to ask whether \cref{T:interpolation-positive} holds if the rank of $F$ is greater than $1$. As presented, the proof does not go through if $\rank F \geq 2$. One of the major obstacles is that the image of $\delta_D \cn \H^0(H(-D)) \rarr \H^1(F(-D))$ may be contained in $\H^1(F'(-D))$ for some proper subbundle $F' \subset F$.
    \end{remark}


\section{Elementary modifications of normal bundles}
\label{S:normal-bundles}

We defined modifications for varieties of arbitrary dimension in \cref{S:modifications-arbitrary}, and later provided curve-specific results in \cref{S:modifications-curves}. We plan to apply these ideas by using modifications of normal bundles of curves along some very specific subbundles. The aim of the present section is to introduce these subbundles and explain their properties.

Let us start with a specific example. Consider a curve $C \subset \Pbb^r$ and a point $p \in \Pbb^r$. We would like to construct a line subbundle $N_{C \to p} \subset N_C$ whose fibers consist of normal directions ``pointing to $p$''. To be more specific, choose a smooth point $q \in C_\sm$ whose tangent line $[T_q C] \subset \Pbb^r$ does not pass through $p$. The fiber $N_{C \to p}|_q \subset N_C|_q$ corresponds to the place spanned by $p$ and $[T_q C]$. Note that we made several assumptions about $C$, $p$, and $q$, so the bundle $N_{C \to p}$ may not be defined over the entire curve $C$. In what follows, we attempt to relax some of these hypotheses while carrying the construction more generally for families of curves. It is also possible to replace the point $p$ by a linear space $\Lambda \subset \Pbb^r$ of arbitrary dimension.

After providing a simple example of what the goal of this section is, we are ready to explain our constructions in full generality. Fix an ambient projective space $\Pbb^r$, and let $\Cc \subset \Pbb^r \x B$ and $\Lambda \subset \Pbb^r \x B$ respectively be flat families (over $B$) of curves and linear spaces. The projections on the second factor $B$ will be denoted by $\pi_\Cc \cn \Cc \rarr B$ and $\pi_\Lambda \cn \Lambda \rarr B$. Given a point $b \in B$, we will use $\Cc_b = \pi_\Cc^{-1}(b)$ and $\Lambda_b = \pi_\Lambda^{-1}(b)$ to denote the curve and linear space over $b$. For simplicity, we will also assume the base $B$ is reduced and connected.

We define the open set $U_{\Cc, \Lambda} \subset \Cc$ consisting of all points $p \in \Cc$ such that $[T_p \Cc_{\pi_\Cc(p)}] \subset \Pbb^r$, the projective realization of the tangent space to the curve containing $p$ at the same point, does not meet the corresponding linear space $\Lambda_{\pi_\Cc(p)}$. Note that if $p \in \Cc$ is a smooth point of the fiber $\Cc_{\pi_\Cc(p)}$, then $[T_p \Cc_{\pi_\Cc(p)}] \subset \Pbb^r$ is a line. A node point yields a 2-plane $[T_p \Cc_{\pi_\Cc(p)}]$, while other singularities may lead to even higher dimensional linear spaces. In the cases of interest, all curve singularities will be nodes.

Consider $\epsilon \cn X = \Bl_{\Lambda} (\Pbb^r \x B) \rarr \Pbb^r \x B$ and the projection $\pi_X \cn X \rarr B$. Let $E = \epsilon^{-1}(\Lambda)$ denote the exceptional divisor. Since blowing up along a flat subscheme commutes with base change, the fiber $X_b = \pi_X^{-1}(b)$ can be identified with an individual blowup $\Bl_{\Lambda_b} \Pbb^r$. Note that $p \in [T_p \Cc_{\pi_\Cc(p)}]$ for all $p \in \Cc$, so $U_{\Cc, \Lambda} \subset \Cc \setminus \Lambda$. In particular, the embedding $\Cc \subset \Pbb^r \x B$ lifts to an embedding $U_{\Cc, \Lambda} \rarr X$. (More generally, there is a lift to $\Cc \setminus \Sing \pi_\Cc$ but we will not need this fact.)

After constructing the blowup $X$, it is natural to consider the fiber-by-fiber quotient of $\Pbb^r$ by $\Lambda$. If the fibers of $\Lambda$ have dimension $k < r$, we define
\[
Y = \{ (P, b) \;|\; \Lambda_b \subset P \} \subset \Gbb(k, r) \x B
\]
with projection to the second factor $\pi_Y \cn Y \rarr B$. Similarly to the typical quotient construction, there is a morphism $f \cn X \rarr Y$ whose fibers are projective spaces of dimension $k + 1$. Furthermore $f$ descends to a rational morphism $\Pbb^r \x B \drarr Y$. We have constructed the following commutative diagram over $B$.
\[\xymatrix{
  & X \ar[dr]^-{f} \ar[d]^-{\epsilon} \\
  \Cc \ar[r] \ar[ur] & \Pbb^r \x B \ar@{-->}[r] & Y
}\]

There is a morphism of normal sheaves $N_{\Cc/X} \to N_{\Cc/Y}$. If we restrict it to the open $U_{\Cc, \Lambda}$, this becomes a morphism of vector bundles, and the kernel is locally free because the tangent spaces to $\Cc$ do not meet $\Lambda$. We will denote this kernel by $N_{\Cc \rarr \Lambda}$. It is important to stress this is a vector bundle only over the open $U_{\Cc, \Lambda}$. Note that $N_{\Cc \rarr \Lambda} \subset N_{\Cc/X}|_{U_{\Cc, \Lambda}}$ by definition. On the other hand, there is a natural isomorphism $N_{\Cc/X}|_{U_{\Cc, \Lambda}} \cong N_{\Cc/\Pbb^r \x B}|_{U_{\Cc, \Lambda}}$, so we can treat $N_{\Cc \rarr \Lambda}$ as a subbundle of the original normal bundle $N_{\Cc/\Pbb^r \x B}$.

\begin{remark}
  \label{T:N-C-Lambda-over-point}
  If the base $B$ is a point, we can extend $N_{\Cc \rarr \Lambda}$ to the entire curve $\Cc$ as long as $U_{\Cc, \Lambda} \neq \emptyset$ contains the singular locus of $\Cc$. After observing that $N_{C \rarr \Lambda} \subset N_{\Cc / \Pbb^r}|_{U_{\Cc, \Lambda}}$, this follows immediately from the curve-to-projective extension theorem.
\end{remark}

The widespread use of the bundles $N_{\Cc \rarr \Lambda}$ makes it economical to update our modification notation. First, when dealing with a family of normal bundles we will write $[D \to \Lambda]$ instead of $[D \to N_{\Cc \rarr \Lambda}]$. Second, if $Z \subset \Pbb^r \x B$ is a family of subvarieties whose fiber-by-fiber spans form a flat family $\Lambda_Z \subset \Pbb^r \x B$, we will also write $[D \to Z]$ instead of $[D \to \Lambda_Z]$.


\section{\boldmath Examples of the bundles $N_{C \to \Lambda}$}
\label{S:normal-bundles-examples}

In this section, we calculate two important examples of the bundles $N_{C \to \Lambda}$
appearing in the previous section. For this, it will be helpful to recall the notion
of an Euler vector field.

\begin{definition}
Let $p \in \Pbb V$ and $H \subset \Pbb V$ be a point and a hyperplane, which are complementary.
This gives rise to a direct sum decomposition $V \simeq \langle p \rangle \oplus \langle H \rangle$.
There is a natural $\Cbb^\times$-action on $\Pbb V$, which is induced by the action of
$\lambda \in \Cbb^\times$ on $V$ given by
\[\lambda(x + y) = \lambda x + y \quad \text{for} \quad x \in \langle p \rangle, \ y \in \langle H \rangle.\]
We then define the \emph{Euler vector field} corresponding to $p$ and $H$ as the differential
(at $\lambda = 1 \in \Cbb^\times$) of the above action.
By inspection, the Euler vector field vanishes at $p$ and along $H$, but nowhere else.
\end{definition}

\begin{proposition} \label{to-Lambda}
Let $\Lambda \simeq \Pbb^{k - 1} \subset \Pbb^r$ be a linear subspace with $k \leq r - 1$,
such that no tangent line of $C$ meets $\Lambda$. (So in particular, $C$ does not meet $\Lambda$.)
Then
\[N_{C \to \Lambda} \simeq \Oc_C(1)^k.\]
\end{proposition}
\begin{proof}
Let $p_1, p_2, \ldots, p_k \in \Lambda$ be $k$ points in linear general position.
This gives us a direct sum decomposition
\[N_{C \to \Lambda} \simeq \bigoplus_{j = 1}^k N_{C \to p_j}.\]
Consequently, we are reduced to the case when $\Lambda = p_1 = p$ is a single point.

In this case, write $H$ for a hyperplane complementary to $p$.
Then the Euler vector field corresponding to $p$ and $H$ gives rise to a section
of $N_{C \to p}$ which vanishes precisely along the intersection of $C$ with $H$.
\end{proof}

\begin{proposition} \label{to-kp}
Let $p \in C$ be a point on $C$,
and $k \leq r - 1$ be an integer. Suppose that $p$ is not a generalized flex point,
and that no other tangent line to $C$ besides $T_p C$ meets $\overline{kp}$.
Then there is an isomorphism
\[\frac{N_{C \to kp}}{N_{C \to (k - 1)p}} \simeq \Oc_C(1)(2p).\]
\end{proposition}
\begin{proof}
Pick some point $q \in \overline{kp} \smallsetminus \overline{(k - 1)p}$.
Note that we have a natural map
\[N_{C \to q} \rarr \frac{N_{C \to kp}}{N_{C \to (k - 1)p}},\]
which is an isomorphism away from $p$. Since $N_{C \to q} \simeq \Oc_C(1)$,
we thus have
\[\frac{N_{C \to kp}}{N_{C \to (k - 1)p}} \simeq \Oc_C(1)(np)\]
for some $n$; it remains to show $n = 2$. For this, we note that the above implies
\[\chi(N_{C \to kp}) - \chi(N_{C \to (k - 1)p}) = \chi\left(\frac{N_{C \to kp}}{N_{C \to (k - 1)p}}\right) = d - g + 1 + n.\]
In particular, we see that it is sufficient to prove
\[\chi(N_{C \to kp}) = k(d - g + 3).\]

For this, we choose a local coordinate $t$ on $C$, and an affine patch
$p = (0, 0, \ldots) \in \Abb^r \subset \Pbb^r$, so that in an analytic neighborhood
of the origin, $C$ is the locus of points of the form
\[C(t) = (t, t^2 + f_2, t^3 + f_3, \ldots, t^{k + 1} + f_{k + 1}, f_{k + 2}, \ldots),\]
where the $f_i = f_i(t) = O(t^{k + 2})$ for all $i$ are holomorphic.
(Such a presentation exists because $p$ is not a generalized flex point of $C$.)
Define $q_1, \ldots, q_{k - 1}$ via
\[q_i = (\underbrace{0, \ldots, 0}_{\text{$i$ times}}, 1, 0, 0, \ldots).\]
The Euler vector fields at $p, q_1, \ldots, q_{k - 1}$ then define an
injective map of sheaves
\[\Oc_C(1)^k \rarr N_{C \to kp},\]
with cokernel supported at $p$.
Since $\chi(\Oc_C(1)^k) = k(d - g + 1)$, it remains to show
that the cokernel has Euler characteristic $2k$.
Equivalently, since the cokernel is supported at $p$,
we want to show that the vectors
\[C(t) = C(t) - p, \quad C(t) - q_1, \quad C(t) - q_2, \quad \ldots, \quad C(t) - q_{k - 1}, \quad C'(t)\]
are linearly independent to order exactly $2k$ in $t$. Or
more explictly, that the $(k + 1) \times r$ matrix
\[\begin{array}{ccccccccc}
t & t^2 + f_2 & \ldots & t^{k - 1} + f_{k - 1} & t^k + f_k & t^{k + 1} + f_{k + 1} & f_{k + 2} & \ldots \\
-1 + t & t^2 + f_2 & \ldots & t^{k - 1} + f_{k - 1} & t^k + f_k & t^{k + 1} + f_{k + 1} & f_{k + 2} & \ldots \\
t & -1 + t^2 + f_2 & \ldots & t^{k - 1} + f_{k - 1} & t^k + f_k & t^{k + 1} + f_{k + 1} & f_{k + 2} & \ldots \\
\vdots & \vdots & \ddots & \vdots & \vdots & \vdots & \vdots & {} \\
t & t^2 + f_2 & \ldots & -1 + t^{k - 1} + f_{k - 1} & t^k + f_k & t^{k + 1} + f_{k + 1} & f_{k + 2} & \ldots \\
1 & 2t + f_2' & \ldots & (k - 1) t^{k - 2} + f_{k - 1}' & k t^{k - 1} + f_k' & (k + 1) t^k + f_{k + 1}' & f_{k + 2}' & \ldots \\
\end{array}\]
has the property that the minimum power of $t$ dividing the
determinants of all $(k + 1) \times (k + 1)$ minors is exactly $t^{2k}$.

To show this, we may first subtract the first row from rows $2, 3, \ldots, (k - 1)$, to replace the above matrix with:
\[\begin{array}{ccccccccc}
t & t^2 + f_2 & \ldots & t^{k - 1} + f_{k - 1} & t^k + f_k & t^{k + 1} + f_{k + 1} & f_{k + 2} & \ldots \\
-1 & 0 & \ldots & 0 & 0 & 0 & 0 & \ldots \\
0 & -1 & \ldots & 0 & 0 & 0 & 0 & \ldots \\
\vdots & \vdots & \ddots & \vdots & \vdots & \vdots & \vdots & {} \\
0 & 0 & \ldots & -1 & 0 & 0 & 0 & \ldots \\
1 & 2t + f_2' & \ldots & (k - 1) t^{k - 2} + f_{k - 1}' & k t^{k - 1} + f_k' & (k + 1) t^k + f_{k + 1}' & f_{k + 2}' & \ldots \\
\end{array}\]

Next, we expand along rows $2, 3, \ldots, (k - 1)$, which reduces
our problem to showing that the $2 \times (r - k + 1)$ matrix
\[\begin{array}{ccccc}
t^k + f_k & t^{k + 1} + f_{k + 1} & f_{k + 2} & f_{k + 3} & \ldots \\
k t^{k - 1} + f_k' & (k + 1) t^k + f_{k + 1}' & f_{k + 2}' & f_{k + 3}' & \ldots \\
\end{array}\]
has the property that the minimum power of $t$ dividing the
determinants of all $2 \times 2$ minors is exactly $t^{2k}$.
Since this matrix in particular has the form
\[\begin{array}{ccccc}
O(t^k) & O(t^{k + 1}) & O(t^{k + 1}) & O(t^{k + 1}) & \ldots \\
O(t^{k - 1}) & O(t^k) & O(t^k) & O(t^k) & \ldots \\
\end{array}\]
we conclude that the determinant of every $2 \times 2$ minor is divisible by $t^{2k}$.
In particular, it suffices to find a particular $2 \times 2$ minor
whose determinant is not divisible by $t^{2k + 1}$. But we can easily calculate
\[\left|\begin{array}{cc}
t^k + f_k & t^{k + 1} + f_{k + 1} \\
k t^{k - 1} + f_k' & (k + 1) t^k + f_{k + 1}'
\end{array}\right| \equiv t^{2k} \mod t^{2k + 1}. \qedhere\]
\end{proof}

\begin{corollary} \label{to-kp-rat}
If in addition $C$ is rational, then
\[N_{C \to kp} \simeq \Oc_C(1)(2p)^k.\]
\end{corollary}
\begin{proof}
We argue by induction on $k$. The base case follows from
Corollary~\ref{to-kp}. For the inductive step,
Proposition~\ref{to-kp} gives an exact sequence of vector bundles
\[\xymatrix@C=0.18in{
  0 \ar[r] &
  N_{C \to (k-1)p} \simeq \Oc_C(1)(2p)^{k-1} \simeq \Oc_{\Pbb^1}(d + 2)^{k - 1} \ar[r] &
  N_{C \to kp} \ar[r] &
  \Oc_C(1)(2p) \simeq \Oc_{\Pbb^1}(d + 2) \ar[r] &
  0.
}\]
But every exact sequence of vector bundles of the form
\[\xymatrix{
  0 \ar[r] &
  \Oc_{\Pbb^1}(n)^a \ar[r] &
  \bullet \ar[r] &
  \Oc_{\Pbb^1}(n)^b \ar[r] &
  0
}\]
on $\Pbb^1$ splits.
\end{proof}


\section{Interpolation and specialization}
\label{S:specialization}

In order to prove our main theorem, we will argue via degeneration. In this section,
we set up the basic results necessary for such an argument.

\begin{proposition} \label{prop:interpolation-is-open} Given a flat family of curves $\mathcal{C} \to B$,
and a vector bundle $\mathcal{E}$ on $\mathcal{C}$, the locus of $b \in B$ for which the
pullback $\mathcal{E}_b$ of $\mathcal{E}$ to $\mathcal{C}_b$ satisfies interpolation is open.
\end{proposition}
\begin{proof}
The bundle $\mathcal{E}_b$ satisfies interpolation if and only if
$\H^0(\mathcal{E}_b(-D)) = 0$ or $\H^1(\mathcal{E}_b(-D)) = 0$, for $D$ a general effective divisor.
But the vanishing of a cohomology group is an open condition, hence the result.

For a more careful proof, see Theorem~5.8 of \cite{Atanasov-interpolation}.
\end{proof}

We next show that certain constructions for producing reducible curves yield
curves which correspond to a point in the component of the Hilbert scheme that we care about.

\begin{lemma} \label{nonspecial-smooth} If $C \subseteq \Pbb^r$ is a curve with $\H^1(\Oc_C(1)) = 0$,
then $\H^1(N_C) = 0$.
\end{lemma}
\begin{proof}
We begin by considering the Euler sequence
\[\xymatrix{
  0 \ar[r] &
  \Oc_C \ar[r] &
  \Oc_C(1)^{r + 1} \ar[r] &
  T_{\Pbb^r}|_C \ar[r] &
  0,
}\]
which gives a long exact sequence
\[\xymatrix{
  \cdots \ar[r] &
  \H^1(\Oc_C(1))^{r + 1} = 0 \ar[r] &
  \H^1(T_{\Pbb^r}|_C) \ar[r] &
  \H^2(\Oc_C) = 0 \ar[r] &
  \cdots.
}\]
In particular, we conclude that $\H^1(T_{\Pbb^r}|_C) = 0$.
But now from the exact sequence
\[\xymatrix{
  0 \ar[r] &
  T_C \ar[r] &
  T_{\Pbb^r}|_C \ar[r] &
  N_C \ar[r] &
  0,
}\]
we obtain a long exact sequence
\[\xymatrix{
  \cdots \ar[r] &
  \H^1(T_{\Pbb^r}|_C) = 0 \ar[r] &
  \H^1(N_C) \ar[r] &
  \H^2(T_C) = 0 \ar[r] &
  \cdots.
}\]
Consequently, $\H^1(N_C) = 0$ as desired.
\end{proof}

\begin{definition}
Write $\mathcal{H}_{d, g, r}$ for the Hilbert scheme
of subschemes of $\Pbb^r$ with Hilbert polynomial $P(x) = dx + 1 - g$.
\end{definition}

\begin{definition}
We say $C \subseteq \Pbb^r$ is \emph{nonspecial} if it is smooth, irreducible, and
$\H^1(\Oc_C(1)) = 0$.
We say $C$ is \emph{limit nonspecial} if $C$ lies in the closure (in the Hilbert scheme
$\mathcal{H}_{d, g, r}$) of the locus of nonspecial curves, and satisfies $\H^1(N_C) = 0$.
\end{definition}

\begin{remark}
From Lemma~\ref{nonspecial-smooth}, we see that every nonspecial curve
satisfies $\H^1(N_C) = 0$, and is therefore limit nonspecial.
\end{remark}

The set of nondegenerate limit nonspecial curves is parameterized by an open subset of a component of
$\mathcal{H}_{d, g, r}$.
If $\mathcal{L}$ is a nonspecial line bundle, then $\dim H^0(\mathcal{L}) = d + 1 - g$.
In particular, if $d < g + r$, there are no nonspecial curves (and thus no limit nonspecial curves).

For $d \geq g + r$, the moduli space of nonspecial curves naturally maps to the open subset $U$ of the Picard bundle
parameterizing
$\{(C, \mathcal{L}) : \mathcal{L} \ \text{is nonspecial}\}$.
Write $\mathcal{E}$ for the natural vector bundle on $U$ whose fiber
over $(C, \mathcal{L})$ gives $H^0(\mathcal{L})$.
Then the moduli space
of nonspecial curves can be described as $\mathbb{V}_{r + 1}(\mathcal{E})$,
where $\mathbb{V}_{r + 1}$ denotes the variety
parameterizing $(r + 1)$-frames up to scaling.
In particular, since the Picard bundle (and therefore $U$) is irreducible,
the moduli space of nonspecial curves (and thus of limit nonspecial curves) is also irreducible.

\begin{definition} We say subschemes $Y$ and $Z$ of $X$ meet
\emph{quasi-transversely} at $x \in Y \cap Z$ if
\[T_x Y \oplus T_x Z \to T_x X\]
is of maximal rank (i.e.\ either injective or surjective).
\end{definition}

\begin{proposition}[Hartshorne-Hirschowitz] \label{prop:hh}
Let $C \subset \Pbb^r$ be a curve with $\H^1(N_C) = 0$, and $L$ a line meeting $C$ quasi-transversely
at one or two points. Then $C \cup L$ lies in the closure of the locus of smooth irreducible curves
and satisfies $\H^1(N_{C \cup L}) = 0$.
\end{proposition}
\begin{proof}
See Theorem~4.1 and Corollary~4.2 of~\cite{Hartshorne-Hirschowitz}
(which is stated for $r = 3$;
but the proof given there generalizes trivially to $r$ arbitrary).
\end{proof}

\begin{corollary} \label{right-component}
Let $C \subset \Pbb^r$ be a limit nonspecial curve, and $L$ a line meeting $C$ quasi-transversely
at one or two points in $C_\sm$. Then $C \cup L$ is limit nonspecial.
\end{corollary}
\begin{proof}
By Lemma~\ref{nonspecial-smooth} and Proposition~\ref{prop:hh},
we know that $C \cup L$ lies in the closure of the locus of smooth irreducible
curves and satisfies $\H^1(N_{C \cup L}) = 0$. It thus remains to show that $C \cup L$
lies in the closure of the locus of nonspecial curves.

Since $C$ is a limit nonspecial curve, and we may deform $L$ along with $C$,
we may suppose by semicontinuity that $C$ is itself nonspecial; it thus remains to show that for
$C$ nonspecial, $\H^1(\Oc_{C\cup L}(1)) = 0$.
For this, consider the exact sequence of sheaves
\[\xymatrix{
  0 \ar[r] &
  \Oc_L(1)(-C \cap L) \ar[r] &
  \Oc_{C \cup L}(1) \ar[r] &
  \Oc_C(1) \ar[r] &
  0,
}\]
which gives rise to the long exact sequence
\[\xymatrix{
  \cdots \ar[r] &
  \H^1(\Oc_L(1)(-C \cap L)) \ar[r] &
  \H^1(\Oc_{C \cup L}(1)) \ar[r] &
  \H^1(\Oc_C(1)) = 0 \ar[r] &
  \cdots.
}\]
To complete the proof, it suffices to show $H^1(\Oc_L(1)(-C \cap L)) = 0$.
But this is clear because $\Oc_L(1)(-C \cap L) = \Oc_{\Pbb^1}(1 - \# (C \cap L))$
is either $\Oc_{\Pbb^1}$ or $\Oc_{\Pbb^1}(-1)$.
\end{proof}


\section{Reducible curves and their normal bundles}
\label{S:reducible}

In this section, we give several tools for relating interpolation of the normal bundle of
a reducible curve to interpolation for the normal bundles of its components. More specifically,
we focus on reducible curves of the form $C \cup L$, where $L$ is a line, which we assume
for the rest of this section meets $C$ quasi-transversely.
Then we can reduce interpolation for the normal bundle $N_{C \cup L}$ to interpolation
for a modification $N_C'$ of the normal bundle $N_C$ as follows.
\begin{enumerate}
\item First we find the space of sections $\H^0(N_{C \cup L}|_L(-D))$ which vanish along some
divisor $D$ supported on $L$. Since $L$ is a line, this can be done via explicit computation.
\item Then we find which sections of $\H^0(N_{C \cup L}|_C)$ match up with the above
space of sections along the intersection $C \cap L$. In many cases, this turns out
to be expressible in terms of a modification $N_C'$.
\end{enumerate}
We begin by proving a general proposition that makes precise under what
hypotheses the above method is applicable.

\begin{proposition} \label{trivglue} Let $\mathcal{E}$ be a vector bundle on a reducible curve
$X \cup Y$, and $D$ be an effective divisor on $X$ disjoint from $X \cap Y$.
Assume that
\[\H^0(\mathcal{E}|_X(-D - X \cap Y)) = 0.\]
Let
\begin{gather*}
\ev_X \colon \H^0(\mathcal{E}|_X) \longrightarrow \H^0(\mathcal{E}|_{X \cap Y}) \\
\ev_Y \colon \H^0(\mathcal{E}|_Y) \longrightarrow \H^0(\mathcal{E}|_{X \cap Y})
\end{gather*}
denote the natural evaluation morphisms.
Then $\mathcal{E}$ satisfies interpolation provided that
\[V = \ev_Y^{-1}(\ev_X(\H^0(\mathcal{E}|_X(-D)))) \subseteq \H^0(\mathcal{E}|_Y)\]
satisfies interpolation and has dimension
\[\chi(\mathcal{E}|_Y) + \chi(\mathcal{E}|_X(-D - X \cap Y)).\]
\end{proposition}
\begin{proof}
We first note that since $\H^0(\mathcal{E}|_X(-D - X \cap Y)) = 0$, the map
$\ev_X$ is injective when restricted to $\H^0(\mathcal{E}|_X(-D))$.
Thus, restriction to $Y$ defines an isomorphism
$\H^0(\mathcal{E}(-D)) \simeq V$.
Consequently,
\[\h^0(\mathcal{E}(-D)) = \chi(\mathcal{E}|_Y) + \chi(\mathcal{E}|_X(-D - X \cap Y)) = \chi(\mathcal{E}(-D)) \quad \Rightarrow \quad \h^1(\mathcal{E}(-D)) = 0.\]
Since $\H^0(\mathcal{E}(-D)) \simeq V$ satisfies interpolation
by assumption, we conclude that $\mathcal{E}(-D)$ --- and hence $\mathcal{E}$ --- satisfies interpolation.
\end{proof}

We now specialize to the case of $C \cup L$ as in the previous section;
we seek to relate interpolation for $N_{C \cup L}$
to interpolation for a modification of $N_C$, along a divisor supported at $C \cap L$.
Since this will depend only on the local behavior of the normal bundles
at the nodes $C \cap L$, we can in fact work in slightly greater generality:
Suppose that $N'_{C \cup L}$ is a bundle on $C \cup L$, equipped with an isomorphism
to $N_{C \cup L}$ over an open set of $C \cup L$ containing the entire line $L$;
in particular, containing a neighborhood $U$ of $C \cap L$ in $C$.

\begin{definition}
Write $N_C'$ for the bundle obtained from $N'_{C \cup L}|_{C \smallsetminus (C \cap L)}$,
glued along $U \smallsetminus (C \cap L)$ via our given isomorphism to the bundle $N_C|_{U \cap C}$.
For example, when $N'_{C \cup L} = N_{C \cup L}$, then $N_C' = N_C$.
\end{definition}

Now suppose that $u \in C \cap L$ is a point of intersection.
Let $v \in L$ and $w \in T_u C$ be points distinct from $u$.

\begin{proposition}[Hartshorne-Hirschowitz] \label{hh}
Let $V$ be a neighborhood of $u$ in $C \cap L$, disjoint from the other
points of intersection (if any). Then, we have isomorphisms
\[N'_{C\cup L}|_{C \cap V} \cong N'_C(u)[u\rightarrow v]|_{C \cap V} \quad \text{and} \quad N'_{C\cup L}|_{L \cap V} \cong N_L(u)[u\rightarrow w]|_{L \cap V},\]
extending the obvious isomorphisms on $V \smallsetminus \{u\}$.
\end{proposition}
\begin{proof}
The above in the special case of $N'_{C \cup L} = N_{C \cup L}$ follows from Corollary~3.2 of \cite{Hartshorne-Hirschowitz}
(which is stated for $r = 3$, but the proof given applies for $r$ arbitrary).
The general case follows from the special case by passing to the neighborhood $U \cap V$ of $u$.
\end{proof}

From the above, the subbundles $N_{C \to v} \subseteq N_C$ and $N_{L \to w} \subseteq N_L$
give rise to subbundles $N'_{C \to v}(u)|_V \subseteq N'_{C\cup L}|_{C \cap V}$ and 
$N_{L \to w}(u)|_V \subseteq N'_{C\cup L}|_{L \cap V}$. The key to analyzing these normal bundles
is the following result.

\begin{lemma} \label{cone-glue}
The fibers of $N'_{C \to v}(u)|_V$ and $N_{L \to w}(u)|_V$ at $u$ coincide.
\end{lemma}
\begin{proof}
Consider the cone $S$ parameterizing pairs of points $(x, y)$ with $x \in C$
and $y$ on the line joining $x$ to $v$.
The inclusion $C \hookrightarrow \Pbb^r$ then factors as $C \hookrightarrow S \to \Pbb^r$.
(The map $C \hookrightarrow S$ is defined via $x \mapsto (x, x)$; the map
$S \to \Pbb^r$ is defined via $(x, y) \mapsto y$.)

Then $S$ is a smooth surface
in the neighborhood $C \smallsetminus \{v\}$ of $u$; so shrinking $V$ if necessary, we may suppose
$S$ is smooth along $V$. Further shrinking $V$, we may suppose
in addition that $U \subseteq V$, so $N'_{C \to v} = N_{C \to v}$.

We now show these fibers coincide by identifying
these two bundles with $N_{(C \cup L)/S}|_C$ and $N_{(C \cup L)/S}|_L$.
Since two subbundles of a vector bundle coincide
if and only if they coincide on a dense open, it suffices to identify
these pairs of bundles on $V^\circ = V \smallsetminus \{u\}$.
But this is clear because
\begin{align*}
N_{C \to v}(u)|_{V^\circ} &= N_{C \to v}|_{V^\circ} = N_{C/S}|_{V^\circ} = N_{(C \cup L)/S}|_{C \cap V^\circ} \\
N_{L \to w}(u)|_{V^\circ} &= N_{L \to w}|_{V^\circ} = N_{L/S}|_{V^\circ} = N_{(C \cup L)/S}|_{L \cap V^\circ}. \qedhere
\end{align*}
\end{proof}

\begin{lemma} \label{addone} Suppose $L$ is a $1$-secant line to $C$, meeting $C$ at $u$;
and $p_1, p_2 \in L$ are points distinct from~$u$. Let
$\Lambda_1$ and $\Lambda_2$ be independent linear subspaces of dimensions $k_1$ and $k_2$,
such that $k_1 + k_2 \leq r - 4$,
and $\overline{\Lambda_1 \cdot \Lambda_2}$ is disjoint from $T_u(C \cap L)$.
Then
\[N'_{C \cup L}[p_1 \to \Lambda_1][p_2 \to \Lambda_2]\]
satisfies interpolation provided that
\[N_C'(u)[u \to v][u \to v \cup \Lambda_1 \cup \Lambda_2]\]
satisfies interpolation, where $v \in L$ is any point distinct from $u$.
\end{lemma}

\begin{proof}
From Proposition~\ref{hh}, we have 
\[N'_{C \cup L}|_C \simeq N'_C(u)[u \to v] \quad \Rightarrow \quad N'_{C \cup L}[p_1 \to \Lambda_1][p_2 \to \Lambda_2]|_C \simeq N'_C(u)[u \to v].\]

For $T$ a general linear space of dimension $r - 5 - k_1 - k_2$ (where $\dim \emptyset = -1$
by convention), 
Proposition~\ref{hh} implies that:
\begin{align*}
N'_{C \cup L}[p_1 \to \Lambda_1][p_2 \to \Lambda_2]|_L &\simeq N_L(u)[u \to w][p_1 \to \Lambda_1][p_2 \to \Lambda_2] \\
&\simeq \left(N_{L \to w} \oplus N_{L \to \Lambda_1} \oplus N_{L \to \Lambda_2} \oplus N_{L \to T} \right)(u)[u \to w][p_1 \to \Lambda_1][p_2 \to \Lambda_2] \\
&\simeq N_{L \to w}(u - p_1 - p_2) \oplus N_{L \to \Lambda_1}(-p_2) \oplus N_{L \to \Lambda_2}(-p_1) \oplus N_{L \to T}(-p_1-p_2)\\
&\simeq \Oc_{\Pbb^1} \oplus \Oc_{\Pbb^1}^{k_1 + 1} \oplus \Oc_{\Pbb^1}^{k_2 + 1} \oplus \Oc_{\Pbb^1}(-1)^{r-4-k_1 - k_2} \\
&\simeq \Oc_{\Pbb^1}^{k_1 + k_2 + 3} \oplus \Oc_{\Pbb^1}(-1)^{r-4-k_1 - k_2}.
\end{align*}
The positive subbundle $\Oc_{\Pbb^1}^{k_1 + k_2 + 3}$ here is:
\[N_{L \to w}(u - p_1 - p_2) \oplus N_{L \to \Lambda_1}(-p_2) \oplus N_{L \to \Lambda_2}(-p_1).\]
The above isomorphism also implies:
\begin{align*}
\H^0(N_{C \cup L}'[p_1 \to \Lambda_1][p_2 \to \Lambda_2]|_L(-u)) &\simeq \H^0(\Oc_{\Pbb^1}(-1)^{k_1+k_2+3} \oplus \Oc_{\Pbb^1}(-2)^{r-4-k_1-k_2}) = 0 \\
\chi(N_{C \cup L}'[p_1 \to \Lambda_1][p_2 \to \Lambda_2]|_L(-u)) &= \chi(\Oc_{\Pbb^1}(-1)^{k_1+k_2+3} \oplus \Oc_{\Pbb^1}(-2)^{r-4-k_1-k_2}) = k_1 + k_2 + 4 - r.
\end{align*}

So applying Proposition~\ref{trivglue}, it suffices to show that the subspace $V \subseteq \H^0(N_C'(u)[u \to v])$
whose fiber at $u$ lies in $\H^0(N_{L \to w}(u - p_1 - p_2)|_u \oplus N_{L \to \Lambda_1}(-p_2)|_u \oplus N_{L \to \Lambda_2}(-p_1)|_u)$
satisfies interpolation and has dimension
\[\chi(N_C'(u)[u \to v]) + k_1 + k_2 + 4 - r = \chi(N_C'(u)[u \to v][u \to v \cup \Lambda_1 \cup \Lambda_2]).\]
But from Lemma~\ref{cone-glue},
\[N_{L \to w}(u - p_1 - p_2)|_u \oplus N_{L \to \Lambda_1}(-p_2)|_u \oplus N_{L \to \Lambda_2}(-p_1)|_u = N'_{C \to v}(u)|_u \oplus N'_{C \to \Lambda_1}|_u \oplus N'_{C \to \Lambda_2}|_u,\]
which implies the above space $V$ is precisely
$\H^0(N_C'(u)[u \to v][u \to v \cup \Lambda_1 \cup \Lambda_2])$.
\end{proof}

Now we suppose $L$ and $C$ are as above, with $\#(C \cap L) = 2$.
Write $C \cap L = \{u, v\}$, and 
pick general points $x \in T_u C$ and
$y \in T_v C$. Let $T$ be a general $(r - 4)$-plane in $\Pbb^r$.
(We take $T = \emptyset$ if $r = 3$.)
We suppose that $T_u C$ does not meet $T_v C$,
so $\{x, y, L\}$ are linearly independent.

\begin{definition}
Let $\mathcal{E} \simeq \Oc_{\Pbb^1}(1)^k$, and $u, v, z \in \Pbb^1$ be distinct points.
Then we obtain an isomorphism
\[\xymatrix{
  \phi_z \colon \H^0(\mathcal{E}|_u) \ar[r]^-{\sim} &
  \H^0(\mathcal{E}|_v),
}\]
defined by the composition
$\H^0(\mathcal{E}|_u) \simeq \H^0(\mathcal{E}(-z)) \simeq \H^0(\mathcal{E}|_v)$.
\end{definition}

\begin{lemma} \label{cross-ratio}
For $z, z' \in \Pbb^1 \smallsetminus \{u, v\}$, we have
\[\phi_z = \frac{(z - v)(z' - u)}{(z - u)(z' - v)} \cdot \phi_{z'}.\]
\end{lemma}
\begin{proof}
Decomposing $\mathcal{E}$ as a direct sum, we reduce to the case of $k = 1$, in which case
the lemma holds by direct computation.
\end{proof}

\begin{lemma} \label{addtwo} Suppose $L$ is a $2$-secant line to $C$, meeting $C$ at $\{u, v\}$.
Then $N_{C \cup L}'$ satisfies interpolation provided that
\[N_C'(u + v)[u \to v][v \to u][v \to 2u]\]
satisfies interpolation.
\end{lemma}
\begin{proof}
From Proposition~\ref{hh}, we have 
$N'_{C \cup L}|_C \simeq N'_C(u+v)[u \to v][v \to u]$.
Additionally, Proposition~\ref{hh} implies that for $z \in L$ general,
\begin{align*}
N_{C \cup L}'|_L(-z) &\simeq N_L(u + v)[u \to x][v \to y](-z) \\
&\simeq \big(N_{L \to x} \oplus N_{L \to y} \oplus N_{L \to T}\big) (u + v)[u \to x][v \to y](-z) \\
&\simeq N_{L \to x}(u - z) \oplus N_{L \to y}(v - z) \oplus N_{L \to T}(-z) \\
&\simeq \Oc_{\Pbb^1}(2)(-z) \oplus \Oc_{\Pbb^1}(2)(-z) \oplus \Oc_{\Pbb^1}(1)(-z)^{r - 3} \\
&\simeq \Oc_{\Pbb^1}(2)(-z)^2 \oplus \Oc_{\Pbb^1}(1)(-z)^{r - 3} \\
&\simeq \Oc_{\Pbb^1}(1)^2 \oplus \Oc_{\Pbb^1}^{r - 3}.
\end{align*}
Note that all isomorphisms except the last one are independent of $z$.
The positive subbundle $\Oc_{\Pbb^1}(1)^2$ here is canonically:
\[N_{L \to x}(u - z) \oplus N_{L \to y}(v - z);\]
while a (choice of) negative complement is (up to isomorphism independent of $z$):
\[N_{L \to T}(-z) \simeq \Oc_{\Pbb^1}(1)(-z)^{r - 3}.\]
The above isomorphism also implies:
\begin{align*}
\H^0(N_{C \cup L}'|_L(-z-u-v)) &\simeq \H^0(\Oc_{\Pbb^1}(-1)^2 \oplus \Oc_{\Pbb^1}(-2)^{r - 3}) = 0 \\
\chi(N_{C \cup L}'|_L(-z-u-v)) &= \chi(\Oc_{\Pbb^1}(-1)^2 \oplus \Oc_{\Pbb^1}(-2)^{r - 3}) = 3 - r.
\end{align*}

We next examine the map $\H^0(N_{C \cup L}'|_L(-z)) \to \H^0(N_{C \cup L}'|_L(-z)|_{\{u, v\}})$.
The above decomposition of $N_{C \cup L}'|_L(-z)$ reduces this problem to understanding the maps
\begin{align*}
  \H^0(N_{L \to x}(u - z) \oplus N_{L \to y}(v - z)) &\longrightarrow \H^0((N_{L \to x}(u - z) \oplus N_{L \to y}(v - z))|_{\{u, v\}}) \\
  \H^0(N_{L \to T}(-z)) &\longrightarrow \H^0(N_{L \to T}(-z)|_{\{u, v\}}).
\end{align*}

Since $N_{L \to x}(u - z) \oplus N_{L \to y}(v - z) \simeq \Oc_{\Pbb^1}(1)^2$, the first map
is surjective. The second map is, by construction, the graph of $\phi_z\colon \H^0(N_{L \to T}|_u) \to \H^0(N_{L \to T}|_v)$.
In particular, limiting $z \to v$, Lemma~\ref{cross-ratio} implies the image of the second map
limits to $\H^0(N_{L \to T}|_u) \times \{0\}$.
So applying Proposition~\ref{trivglue}, it suffices to show that the subspace $V \subseteq \H^0(N_C'(u+v)[u \to v][v\to u])$
of sections whose restriction to $\{u, v\}$ lies in
\[\H^0((N_{L \to x}(u - z) \oplus N_{L \to y}(v - z))|_{\{u, v\}}) \oplus (\H^0(N_{L \to T}|_u) \times \{0\}),\]
or equivalently whose restriction to $v$ lies in $\H^0((N_{L \to x}(u - z) \oplus N_{L \to y}(v - z))|_v)$,
satisfies interpolation and has dimension
\[\chi(N_C'(u+v)[u \to v][v\to u]) + 3 - r = \chi(N_C'(u+v)[u \to v][v\to u][v \to 2u]).\]
But from Lemma~\ref{cone-glue} and the machinery of \cref{S:normal-bundles}, we have
\[(N_{L \to x}(u - z) \oplus N_{L \to y}(v - z))|_v = (N'_{C \to x} \oplus N'_{C \to u}(v))|_v,\]
which gives
\[V = \H^0(N_C'(u+v)[u \to v][v\to u][v \to (u \cup x)]) = \H^0(N_C'(u+v)[u \to v][v\to u][v \to 2u]). \qedhere\]
\end{proof}

We now give an alternative condition for $N'_{C \cup L}$ to satisfy interpolation;
this requires first introducing some new notation:

\begin{definition} \label{corr}
Let $x \neq y$ be points on $C$; and $X$ and $Y$ be points of $\Pbb^r$,
or sub-line-bundles of $N_C$.
We say a subspace $V \subseteq \H^0 (\mathcal{E})$
is an
\[\H^0 (\mathcal{E}) \langle x \to X : y \to Y \rangle \]
to indicate that 
\[\H^0 (\mathcal{E}(-x-y)) \subsetneq V \subsetneq \H^0 (\mathcal{E}[x \to X][y \to Y]),\]
but that $V$ is neither $\H^0 (\mathcal{E}(-x)[y \to Y])$, nor $\H^0 (\mathcal{E}(-y)[x \to X])$.

More generally, given points $x_1, y_1, \ldots, x_k, y_k \in C$, and points or sub-line-bundles
$X_1, Y_1$, \ldots, $X_k, Y_k \in \Pbb^r$,
we say $V$ is an
\[\H^0 (\mathcal{E}) \langle x_1 \to X_1 : y_1 \to Y_1 \rangle \cdots \langle x_k \to X_k : y_k \to Y_k \rangle\]
if it is a sum (not direct sum) of spaces $V_i$ which are
\[\H^0 (\mathcal{E}(-x_1 - y_1 - \cdots - \widehat{x_i} - \widehat{y_i} - \cdots - x_k - y_k)) \langle x_i \to X_i : y_i \to Y_i \rangle.\]

In all applications, the
$\H^0(E(-x_1-y_1 - \cdots - \widehat{x_i}-\widehat{y_i}-\cdots - x_k-y_k)[x_i \to X_i][y_i \to Y_i])$
will be linearly independent relative to $\H^0(E(-x_1-x_2 - \cdots - x_k - y_k))$.
\end{definition}

\begin{lemma} \label{addtwo-two} Suppose $L$ is a $2$-secant line to $C$, meeting $C$ at $\{u, v\}$.
Write $x \in T_u C$ and $y \in T_v C$ for points in their respective tangent lines, distinct from $u$ and $v$.
Suppose that $T_u C$ does not meet $T_v C$.
Then $N_{C \cup L}'$ satisfies interpolation provided that every
\[\H^0 \big(N'_C(2u + 2v)[u \to v][v \to u] \big)\langle u \to v : v \to x \rangle \langle v \to u : u \to y \rangle\]
satisfies interpolation, and
\[\H^1 (N'_C[u \to v][v \to u]) = 0.\]
\end{lemma}
\begin{proof}
From Proposition~\ref{hh}, we have 
$N'_{C \cup L}|_C \simeq N'_C(u+v)[u \to v][v \to u]$.
Additionally, Proposition~\ref{hh} implies that for $z, w \in L$ general,
\begin{align*}
N_{C \cup L}'|_L(-z-w) &\simeq N_L(u + v)[u \to x][v \to y](-z-w) \\
&\simeq N_{L \to x}(u - z-w) \oplus N_{L \to y}(v - z-w) \oplus N_{L \to T}(-z-w) \\
&\simeq \Oc_{\Pbb^1}^2 \oplus \Oc_{\Pbb^1}(-1)^{r - 3}.
\end{align*}
The above isomorphism also implies:
\begin{align*}
\H^0(N_{C \cup L}'|_L(-z-w-u-v)) &\simeq \H^0(\Oc_{\Pbb^1}(-2)^2 \oplus \Oc_{\Pbb^1}(-3)^{r - 3}) = 0 \\
\chi(N_{C \cup L}'|_L(-z-w-u-v)) &= \chi(\Oc_{\Pbb^1}(-2)^2 \oplus \Oc_{\Pbb^1}(-3)^{r - 3}) = 4 - 2r.
\end{align*}

We next examine the map $\H^0(N_{C \cup L}'|_L(-z-w)) \to \H^0(N_{C \cup L}'|_L(-z-w)|_{\{u, v\}})$.
The above decomposition of $N_{C \cup L}'|_L(-z-w)$ reduces this problem to understanding the maps
\begin{align*}
\H^0(N_{L \to x}(u-z-w)) &\longrightarrow \H^0(N_{L \to x}(u-z-w)|_{\{u, v\}}) \\
\H^0(N_{L \to y}(v-z-w)) &\longrightarrow \H^0(N_{L \to y}(v-z-w)|_{\{u, v\}}) \\
\H^0(N_{L \to T}(-z-w)) &\longrightarrow \H^0(N_{L \to T}(-z-w)|_{\{u, v\}}).
\end{align*}

Since $\H^0(N_{L \to T}(-z-w)) = \H^0(\Oc_{\Pbb^1}(-1)^{r - 3}) = 0$,
the last map is zero.
The first two maps are, by construction,
the graphs of the isomorphisms $\phi_w$:
\begin{align*}
\H^0(N_{L \to x}(u-z-w)|_u) &\longrightarrow \H^0(N_{L \to x}(u-z-w)|_v) \\
\H^0(N_{L \to y}(v-z-w)|_u) &\longrightarrow \H^0(N_{L \to y}(v-z-w)|_v).
\end{align*}
In particular, they are subsets $W_u \subset \H^0(N'_{L \to x}(u-z-w)|_u) \times \H^0(N'_{L \to x}(u-z-w)|_v)$,
respectively $W_v \subset \H^0(N'_{L \to y}(v-z-w)|_u) \times \H^0(N'_{L \to y}(v-z-w)|_v)$, of dimension $1$,
which are not contained in either factor.

By Proposition~\ref{trivglue},
it suffices to show that the subspace $V \subseteq \H^0(N_C'(u+v)[u \to v][v\to u])$
of sections whose restriction to $\{u, v\}$ lies in $W_u \oplus W_v$
satisfies interpolation and has dimension
\[\chi(N_C'(u+v)[u \to v][v\to u]) + 4 - 2r = \chi(N_C'[u \to v][v\to u]) + 2.\]
We note that by Lemma~\ref{cone-glue}, together with the machinery of
\cref{S:normal-bundles}, the subsets $W_u$ and $W_v$ can equally
well be described as subsets:
\begin{align*}
W_u &\subset \H^0(N'_{C \to v}(v)|_u) \times \H^0(N'_{C \to x}|_v) \\
W_v &\subset \H^0(N'_{C \to y}|_u) \times \H^0(N'_{C \to u}(v)|_v),
\end{align*}
which are of dimension~$1$ and not contained in either factor.

For the dimension statement, we note that since $\H^1(N_C'[u \to v][v\to u]) = 0$,
the map $\underline{\res}$ of restriction to $\{u, v\}$ is surjective. Consequently,
\begin{align*}
  \dim V
  &= \dim \Ker (\underline{\res}) + 2 \\
  &= \dim \H^0(N_C'[u \to v][v\to u]) + 2 \\
  &= \chi(N_C'[u \to v][v\to u]) + 2,
\end{align*}
as desired. For the interpolation statement, it suffices to show that $V$ is an
\[\H^0 \big(N'_C(2u + 2v)[u \to v][v \to u] \big)\langle u \to v : v \to x \rangle \langle v \to u : u \to y \rangle.\]

Write $V_u$ for the subspace of $\H^0(N_C'(u+v)[u \to x][v\to y])$
consisting of sections whose restriction to $\{u, v\}$ lies in $W_u$, and define similarly $V_v$.
Again, since restriction to $\{u, v\}$ is surjective, $V = V_u + V_v$.
It is therefore sufficient (by symmetry) to note that $V_u$ is an
\[\H^0 \big(N'_C(u + v)[u \to v][v \to u]\big) \langle u \to v : v \to x \rangle. \qedhere\]
\end{proof}

\begin{lemma} \label{line-limit} Let $L \subset \Pbb^r$ be a line, and $w, s, t \in \Pbb^r$ be three distinct collinear points,
lying on a line which does not meet $L$.
Then for $u, q, p \in L$, the bundle
\[N_{L \to w \cup s} (u-q)[u \to w][q \to s][p \to t] \simeq \Oc_{\Pbb^1} \oplus \Oc_{\Pbb^1}(-1).\]
Moreover, writing $P \simeq \Oc_{\Pbb^1}$ for the positive subbundle,
the fiber $P|_u$ limits, as $q \to u$, to the fiber
\[N_{L \to s}(-q-p)|_u \subseteq N_L(u-q)[u \to w][q \to s][p \to t].\]
\end{lemma}
\begin{proof}
As $N_{L \to s}$ and $N_{L \to t}$, viewed as subbundles of $N_{L \to w \cup s}$,
have distinct fibers at $p$, this holds after modification at $u$ and $q$.
As $N_{L \to s}(-q)$ is a subbundle of
$N_{L \to w \cup s}(u-q)[u \to w][q \to s]$,
it follows that $N_{L \to s}(-q-p)$ is a subbundle of $N_{L \to w \cup s}(u)[u \to w][q \to s][p \to t]$.
But from Proposition~\ref{T:chi-modifications}, we have
\begin{align*}
\chi(N_{L \to s}(-q-p)) &= 0 \\
\chi(N_{L \to w \cup s}(u)[u \to w][q \to s][p \to t]) &= 4 + 2 - 1 - 1 - 1 = 1.
\end{align*}

Now by the classification of vector bundles on $\Pbb^1$, 
is is clear that any rank~$2$ vector bundle on $\Pbb^1$, which has Euler characteristic~$1$
and a subbundle of Euler characteristic~$0$, is necessarily $\Oc_{\Pbb^1} \oplus \Oc_{\Pbb^1}(-1)$.

Our next problem is to calculate the behavior of the fiber $P|_u$ as $q \to u$. To do this,
we choose isomorphisms $N_{L \to w} \simeq N_{L \to s} \simeq \Oc_{\Pbb^1}(1)$
such that $N_{L \to t}$ is the diagonal. Here, we identify $\Oc_{\Pbb^1}(1)$ with the bundle
of functions with a pole allowed at $\infty$. We then act by an automorphism of $\Pbb^1$
preserving $\infty$ to send $u$ and $p$ to $0$ and $1$ respectively.
We write
\[N_{L \to w \cup s}(u-q)[u \to w][q \to s][p \to t] = \left(N_{L \to w}(u - 2q) \oplus N_{L \to s}(-q)\right)[p \to t].\]
In terms of a local coordinate $z$ on $\Pbb^1$, sections of $N_{L \to w}(u) \oplus N_{L \to s}$
are then expressions of the form
\[\left(\frac{a}{z} + c + dz, b + ez\right).\]
Here, $P|_u$ is identified with the lowest-order terms $[a : b] \in \Pbb(N_{L \to w}(u)|_u \oplus N_{L \to s}|_u)$.
To be a section of $\left(N_{L \to w}(u - 2q) \oplus N_{L \to s}(-q)\right)[p \to t]$, we require:
\begin{gather*}
\left.\left(\frac{a}{z} + c + dz\right)\right|_{z = q} = \left.\frac{d}{dz} \left(\frac{a}{z} + c + dz\right)\right|_{z = q} = 0 \\
(b + ez)|_{z = q} = 0 \\
\left.\left(\frac{a}{z} + c + dz\right)\right|_{z = 1} = (b + ez)|_{z = 1}.
\end{gather*}
This is a system of linear equations in $a, b, c, d, e$; eliminating $c, d, e$ via elementary linear algebra gives
\[a(1 - q) + bq = 0.\]
In particular, as $q \to u = 0$, the subspace $P|_u$ limits to $[a : b] = [0 : 1]$.
Or in other words, the fiber $P|_u$ limits to the fiber $N_{L \to s}|_u \simeq N_{L \to s}(-q-p)|_u \subseteq N_{L \to w \cup s}(u-q)[u \to w][q \to s][p \to t]$.
\end{proof}

\begin{lemma} \label{fiddle-over}
Suppose $L$ is a $1$-secant line to $C$, meeting $C$ at $u$;
and $p, q \in L$ and $w \in T_u C$ are points distinct from~$u$. Let $\Lambda$ be a linear space
of dimension $r - 3$, and $s$ a point. Suppose the
subspaces $\Lambda$, $s$, and $w$ are disjoint from $L$, and that their projections
from $L$ are in linear general position.
Then
\[N'_{C \cup L}[q \to s][p \to \Lambda]\]
satisfies interpolation, in the limit $q \to u$, provided that
\[N_C'[u \to p \cup s]\]
satisfies interpolation.
\end{lemma}
\begin{proof}
First we note that the bundle we wish to prove satisfies interpolation and the bundle
we assume satisfies interpolation, both depend only upon the projection of $\Lambda$ from $L$.
Consequently, since the projection of $\Lambda$ from $L$ meets the projection of $\overline{ws}$ from $L$ at a point,
we may suppose $\Lambda$ meets $\overline{ws}$ at a point $t$. 
Let $\Lambda' \subset \Lambda$ be a codimension~$1$ subspace, disjoint from $t$
(which we take to be the empty set if $r = 3$).

Now, from Proposition~\ref{hh}, we have 
\[N'_{C \cup L}|_C \simeq N'_C(u)[u \to p] \quad \Rightarrow \quad N'_{C \cup L}[q \to s][p \to \Lambda]|_C \simeq N'_C(u)[u \to p].\]
Additionally (using Lemma~\ref{line-limit}),
\begin{align*}
N'_{C \cup L}[q \to s][p \to \Lambda]|_L(-q) &\simeq N_L(u)[u \to w][q \to s][p \to \Lambda](-q) \\
&\simeq N_L(u)[u \to w][q \to s][p \to t \cup \Lambda'](-q) \\
&\simeq \left(N_{L \to w \cup s} \oplus N_{L \to \Lambda'}\right)(u)[u \to w][q \to s][p \to t \cup \Lambda'](-q) \\
&\simeq N_{L \to w \cup s}(u-q)[u \to w][q \to s][p \to t] \oplus N_{L \to \Lambda'}(-2q) \\
&\simeq \Oc_{\Pbb^1} \oplus \Oc_{\Pbb^1}(-1) \oplus \Oc_{\Pbb^1}(-1)^{r - 3} \\
&\simeq \Oc_{\Pbb^1} \oplus \Oc_{\Pbb^1}(-1)^{r - 2}.
\end{align*}
The positive subbundle $\Oc_{\Pbb^1}$ here is:
\[\text{the positive subbundle $P$ of} \ N_{L \to w \cup s}(u-q)[u \to w][q \to s][p \to t].\]
The above isomorphism also implies:
\begin{align*}
\H^0(N'_{C \cup L}[q \to s][p \to \Lambda]|_L(-q-u)) &\simeq \H^0(\Oc_{\Pbb^1}(-1) \oplus \Oc_{\Pbb^1}(-2)^{r-3}) = 0 \\
\chi(N'_{C \cup L}[q \to s][p \to \Lambda]|_L(-q-u)) &= \chi(\Oc_{\Pbb^1}(-1) \oplus \Oc_{\Pbb^1}(-2)^{r-2}) = 2 - r.
\end{align*}

So applying Proposition~\ref{trivglue}, it suffices to show that the subspace $V \subseteq \H^0(N_C'(u)[u \to p])$
whose fiber at $u$ lies in $\H^0(P|_u)$
satisfies interpolation and has dimension
\[\chi(N_C'(u)[u \to p]) + 2 - r = \chi(N_C'[u \to p \cup s]).\]

Again from Lemma~\ref{line-limit}, the subspace $V \subseteq \H^0(N_C'(u)[u \to p])$ thus limits to the space
of sections whose fiber at $u$ lies in $\H^0(N_{L \to s}(-p-q)|_u)$,
or equivalently (by the machinery of \cref{S:normal-bundles}) whose fiber at $u$ lies in
$\H^0(N'_{C \to s}|_u)$,
which we recognize as the space of sections
$\H^0(N_C'(u)[u \to p][u \to s])$.
\end{proof}


\section{A stronger inductive hypothesis}
\label{S:hypothesis}

In this section, we explain a generalization of our interpolation problem; this
generalization will allow us to make an inductive argument in the following sections.

\begin{definition}
Consider a curve $C$, equipped with a collection of general points in $C_\sm$:
\begin{itemize}
\item One marked point $p$;
\item For any triple $(i, j; k)$ of integers in the set
\[(1, 1; 1) \quad (2, 0; 1) \quad (1, 0; 2) \quad (1, 1; 0) \quad (1, 0; 1) \quad (2, 0; 0) \quad (0, 0; 2) \quad (1, 0; 0) \quad (0, 0; 1),\]
a set of $n_{ij}^k$ points
$q_{ij}^k(1), q_{ij}^k(2), \ldots, q_{ij}^k(n_{ij}^k)$. We call these points
\emph{$(i, j; k)$-points}, and we write $q_{ij}^k$ for the divisor
$q_{ij}^k(1) + \cdots + q_{ij}^k(n_{ij}^k)$.
\end{itemize}
Here and throughout, we require
\[\sum_{i,j,k} kn_{ij}^k < r - 1.\]
In addition, we require $r \geq 2$; and if $r = 2$,
we require
\[\sum_{i,j,k} jn_{ij}^k = 0.\]
Then we define the \emph{modified normal bundle}
\begin{align*}
N_C' &= N_C\big( (i + j - 1) q_{ij}^k\big)[p \to kq_{ij}^k][iq_{ij}^k \to p][jq_{ij}^k \to 2p] \\
&= N_C\left(\sum_{i, j, k} (i + j - 1) q_{ij}^k\right)\left[p \to \sum_{i, j, k} kq_{ij}^k\right]\left[\sum_{i,j,k} iq_{ij}^k \to p\right]\left[\sum_{i,j,k} jq_{ij}^k \to 2p\right].
\end{align*}
In general, when we write an expression with indices $i, j, k$ in a twist or modification of a 
vector bundle, the reader should
sum over $i, j, k$, as in the above example.
\end{definition}

\begin{remark}
Note that for every allowed $(i,j;k)$, we have $i \geq j \geq 0$ and $k \geq 0$.
Moreover, when $(i, j; k)$ is allowed, then both $(i, j; 0)$ and $(j, 0; k)$ are either allowed
or equal to $(0, 0; 0)$.
(And in particular, combining these two statements, so is $(j, 0; 0)$.)
\end{remark}

Note that when every $n_{ij}^k = 0$,
we have $N_C' = N_C(-p)$;
hence, interpolation for $N_C'$ will imply interpolation for $N_C$.

\begin{definition}
We say that a quadruple $\big(d,\, g,\, r,\, n \cn\! (i, j; k) \mapsto n_{ij}^k\big)$ is
\emph{good} if the modified normal bundle of
a general curve of degree $d$ and genus $g$ in $\Pbb^r$, with $n_{ij}^k$ general marked points
of type $(i, j; k)$, satisfies interpolation.
\end{definition}

\begin{lemma} \label{chimod} We have
\[\chi(N_C') = (r + 1)d - (r - 3)g - 2 - \sum_{i,j,k} (r - 1 - i - 2j - k) n_{ij}^k.\]
\end{lemma}
\begin{proof} Since $\chi(N_C) = (r + 1)d - (r - 3)(g - 1)$, the result follows from counting the changes
in the normal bundle at respectively $p$, and at all $(i, j; k)$-points (c.f.\ Proposition~\ref{T:chi-modifications}):
\[\chi(N_C') = \chi(N_C) - \left(r - 1 - \sum_{i,j,k} k n_{ij}^k\right) - \sum_{i,j,k} (r - 1 - i - 2j) n_{ij}^k. \qedhere \]
\end{proof}

\begin{lemma} \label{top}
The sub-line-bundle $N_{C \to p}'$ of $N_C'$ consisting of sections which point towards $p$
is nonspecial and has
Euler characteristic given by
\[d - g + 2 + \sum_{i,j,k} (i + j - 1) n_{ij}^k.\]
\end{lemma}
\begin{proof}
The bundle $N_{C \to p}$ (without modification) is isomorphic to $\Oc_C(1)(2p)$ by \cref{to-kp}.
Consequently, $N_{C \to p}'$ is isomorphic to
\[N_{C \to p}' \simeq N_{C \to p}(-p) \big((i + j - 1) q_{ij}^k\big) \simeq \Oc_C(1)(p)\big((i + j - 1) q_{ij}^k\big).\]
By inspection, $N_{C \to p}'$ is a general line bundle
of the given Euler characteristic. It thus remains to show
that the given Euler characteristic is positive. But
\[d - g + 2 + \sum_{i,j,k} (i + j - 1) n_{ij}^k \geq (g + r) - g + 2 - \sum_{i,j,k} k n_{ij}^k \geq r + 2 - (r - 2) \geq 0. \qedhere\]
\end{proof}

\noindent
In particular, since by \cref{T:interpolation-ses} part~\ref{T:interpolation-ses-a},
the bundle $N_C'$ can only satisfy interpolation when
\begin{align*}
(r - 1) \cdot \chi(N_{C \to p}') - (r - 2) &\leq \chi(N_C'),
\end{align*}
the previous two lemmas imply that a \emph{necessary}
condition for $N_C'$ to satisfy interpolation is
\begin{equation}
\sum_{i,j,k} ((r - 2)i + (r - 3)j - k) \cdot n_{ij}^k \leq 2d + 2g - r - 2. \label{regime}
\end{equation}

Our goal for the rest of the paper will be to establish a partial converse: That
subject to certain conditions (the most important of which is \eqref{regime}) --- which
are satisfied in particular when every $n_{ij}^k = 0$ --- the
general modified normal bundle $N_C'$ satisfies interpolation.
To do this, we must first know that the property of $N_C'$ satisfying interpolation
is open; this is made precise by the following crucial proposition.

\begin{proposition} \label{modified-normal-glue}
Let $\mathcal{C} \to B$ be a family of curves, $p \colon B \to \mathcal{C}$
a section, and $q_{ij}^k$ be effective divisors on $\mathcal{C}$ which are flat over $B$
of relative degree $n_{ij}^k$.
Suppose that, for every $b \in B$:
\begin{enumerate}
\item The divisor $2p(b) + \sum_{i,j,k} k q_{ij}^k(b)$ is nondegenerate;
\item For any $x \in q_{ij}^k(b)$, the tangent lines to $\mathcal{C}(b)$
at $x$ and $p(b)$ are disjoint.
\end{enumerate}
Then the locus of $b \in B$ so that the modified
normal bundle for $(\mathcal{C}(b), p(b), q_{ij}^k(b))$
satisfies interpolation is open.
\end{proposition}
\begin{proof}
Define
\[\Lambda = \overline{\sum_{i,j,k} k q_{ij}^k} \quad \text{and} \quad P = \overline{2p}.\]
Our first assumption implies that $p = p(B)$ is contained in $U_{\mathcal{C}, \Lambda}$.
Similarly, our second assumption implies that the $q_{ij}^k$ are
contained in $U_{\mathcal{C}, P}$ (and consequently
in~$U_{\mathcal{C}, p}$).

We now wish to construct a vector bundle $N_{\mathcal{C}}'$
on the total space
$\mathcal{C}$ whose restriction to each fiber $\mathcal{C}(b)$ is $N_{\mathcal{C}(b)}'$.
For this, we can appeal to the results of Section~\ref{S:modifications-arbitrary}:
It suffices to check that the modification datum
\[(p, N_{C \to \Lambda}), \quad (q', N_{C \to p}), \quad (q'', N_{C \to P})\]
is tree-like, where
\[q' = \sum_{i,j,k} i n_{ij}^k \quad \text{and} \quad q'' = \sum_{i,j,k} j q_{ij}^k.\]
Since our second condition implies that $p$ does not meet either $q'$ or $q''$,
it suffices to check $\{N_{C \to p}, N_{C \to P}\}$ is tree-like along $q' \cap q''$.
But this is clear, since $N_{C \to p} \subset N_{C \to P}$.

The desired result now follows from applying \cref{prop:interpolation-is-open}
to our bundle $N_{\mathcal{C}}'$.
\end{proof}


\section{Inductive arguments}
\label{S:fixed}

In this section, we give a number of inductive arguments
to reduce interpolation for certain modified normal bundles to
interpolation for other ``simpler'' modified normal bundles.
We begin by two ways of adding a $2$-secant line, which result from
respectively limiting $u$ and $v$ to $p$ in Lemma~\ref{addtwo}.

Throughout this section, and in the following section, we will make use of several such
``limiting arguments'', all but one of
which are straight-forward applications
of the machinery developed in Section~\ref{S:modifications-arbitrary}.
In \cref{two-secant}
below, we will spell this out this limiting argument explicitly;
subsequently, starting with \cref{stick}, it will be left to the reader
to check that the limiting argument given in \cref{two-secant}
applies, mutatis mutandis.

We will also spell out the limiting argument explicitly
in \cref{two-secant-backwards},
as this is the only case where the argument given 
in \cref{two-secant} (mutatis mutandis) does not apply.

\begin{proposition} \label{negtwist}
If a modified normal bundle $N_C'$
satisfies interpolation, then so does a general negative twist
\[N_C'(-p_1-p_2-\cdots-p_n) \quad \text{for} \quad n \leq r + 1 - \sum_{i,j,k} k n_{ij}^k.\]
\end{proposition}
\begin{proof}
Since $N_C'$ satisfies interpolation, it satisfies \cref{regime}.
By casework, we see that for each allowed $(i, j; k)$,
\[r - 1 - i - 2j - k \leq (r - 2)i + (r - 3)j  + (r - 2)k.\]
Combining these facts and applying \cref{chimod},
\begin{align*}
\chi(N_C') &= (r + 1)d - (r - 3)g - 2 - \sum_{i,j,k} (r - 1 - i - 2j - k) n_{ij}^k \\
&\geq (r + 1)d - (r - 3)g - 2 - \sum_{i,j,k} ((r - 2)i + (r - 3)j  + (r - 2)k) n_{ij}^k \\
&= (r + 1)d - (r - 3)g - 2 - \sum_{i,j,k} ((r - 2)i + (r - 3)j - k) n_{ij}^k - (r - 1) \cdot \sum_{i,j,k} k n_{ij}^k\\
&\geq (r + 1)d - (r - 3)g - 2 - (2d + 2g - r - 2) - (r - 1) \cdot \sum_{i,j,k} k n_{ij}^k \\ 
&= (r - 1)(d - g) + r - (r - 1) \cdot \sum_{i,j,k} k n_{ij}^k \\ 
&\geq (r - 1) \cdot r + (r - 1) - (r - 1) \cdot \sum_{i,j,k} k n_{ij}^k \\
&= (r - 1) \cdot \left(r + 1 - \sum_{i,j,k} k n_{ij}^k\right).
\end{align*}
Consequently, for $n \leq r + 1 - \sum_{i,j,k} k n_{ij}^k$,
the twist $N_C'(-p_1-p_2-\cdots-p_n)$ has nonnegative Euler characteristic,
which immediately implies the desired conclusion.
\end{proof}

\begin{lemma} \label{two-secant}
Let $g > 0$. Suppose that $(d, g, r; n)$ satisfies \eqref{regime} and
\[\sum_{i,j,k} k n_{ij}^k < r - 2.\]
Then $(d, g, r; n)$ is good provided that
$(d - 1, g - 1, r; n')$ is good, where
\[(n')_{ij}^k = \begin{cases}
n_{ij}^k & \text{if $(i, j; k) \neq (1, 1; 1)$;} \\
n_{11}^1 + 1 & \text{if $(i, j; k) = (1, 1; 1)$.}
\end{cases}\]
If instead
\[\sum_{i,j,k} k n_{ij}^k = r - 2,\]
then $(d, g, r; n)$ is good provided that
$(d - 1, g - 1, r; n')$ is good, where
\[(n')_{ij}^k = \begin{cases}
\sum_\ell n_{ij}^\ell & \text{if $k = 0$ and $(i, j; k) \notin \{(0, 0; 0), (1, 1; 0)\}$;} \\
1 + \sum_\ell n_{ij}^\ell & \text{if $(i, j; k) = (1, 1; 0)$;} \\
0 & \text{else.}
\end{cases}\]
\end{lemma}
\begin{proof}
Degenerate $C$ to $D \cup L$, where $L$ is a $2$-secant line, and all marked points
specialize to points on $D$. Applying Lemma~\ref{addtwo},
interpolation for $N_C'$ is reduced to interpolation for
\[N_D'(u + v)[u \to v][v \to u][v \to 2u].\]
The next step is to \emph{limit $u \to p$}, which reduces our problem to interpolation for
\[N_D'(p + v)[p \to v][v \to p][v \to 2p].\]

More precisely, write $\mathcal{D} \to B$ for the constant family $D \times B \to B$ over $B$,
where $B \subset D$ is some open set containing $p$.
What we mean by the above is that there is a vector bundle on $\mathcal{D}$,
whose restriction to the fiber $D \times \{u\}$ is, for $u \in B$,
\[N_D'(u + v)[u \to v][v \to u][v \to 2u].\]

For this, we let $\Lambda$, $P$, $q'$, and $q''$ be as in \cref{modified-normal-glue};
and write $T = \sum_{i,j,k} (i + j - 1) q_{ij}^k$.
By minor abuse of notation, we also write $v$ for the horizontal divisor $v \times B$,
and $u$ for the diagonal in $D \times B$. We then appeal to the
machinery of Section~\ref{S:modifications-arbitrary}, which constructs our desired bundle
\[N_{\mathcal{D}}(T)[p \to \Lambda][q' \to p][q'' \to P](u + v)[u \to v][v \to u][v \to 2u],\]
provided that the modification datum
\[(p, N_{\mathcal{D} \to \Lambda}), \quad (q', N_{\mathcal{D} \to p}), \quad (q'', N_{\mathcal{D} \to P}), \quad (u, N_{\mathcal{D} \to v}), \quad (v, N_{\mathcal{D} \to u}), \quad (v, N_{\mathcal{D} \to \overline{2u}})\]
is tree-like. The divisor $p$ does not cross either $q'$ or $q''$;
additionally, since $v$ is general, $v$ does not cross $p$, $q'$, or $q''$.
Moreover, by shrinking $B$, we may suppose $u$ does not cross $v$, $q'$, or $q''$.
It therefore suffices to see that the collections of bundles
\[\{N_{\mathcal{D} \to p}, N_{\mathcal{D} \to P}\}, \quad \{N_{\mathcal{D} \to u}, N_{\mathcal{D} \to \overline{2u}}\}, \quad \{N_{\mathcal{D} \to \Lambda}, N_{\mathcal{D} \to v}\}\]
are tree-like along $q' \cap q''$, $v$, and $(p, p) \in \mathcal{D} = D \times B$ respectively.
But $N_{\mathcal{D} \to p} \subset N_{\mathcal{D} \to P}$ and $N_{\mathcal{D} \to u} \subset N_{\mathcal{D} \to \overline{2u}}$, which takes care of the first two cases.
For the last case, the generality of $v$ implies $v$ is not contained in
the span of $\Lambda$ with the tangent line to $D$ at $p$. Consequently, the fibers
$N_{\mathcal{D} \to u}|_{(p,p)}$ and $N_{\mathcal{D} \to \overline{2u}}|_{(p,p)}$
are linearly independent.

Moving on, we collect together the transformations $[p \to \sum k q_{ij}^k(\ell)]$ and $[p \to u]$
into a single transformation $[p \to u + \sum k q_{ij}^k(\ell)](-p)$
via Proposition~\ref{T:combining-modifications}.

When $\sum_{i,j,k} k n_{ij}^k < r - 2$, we recognize this as another
modified normal bundle, with a new point of type $(1, 1; 1)$ introduced at $u$,
as desired.

Similarly, when $\sum_{i,j,k} k n_{ij}^k < r - 2$, we recognize its twist by $-p$
as another modified normal bundle,
where all points of type $(i, j; k)$ are changed
to type $(i, j; 0)$, and a new point of type $(1, 1; 0)$ is introduced at $u$.
Eliminating the $(0, 0; 0)$-points (which are just general negative twists)
via \cref{negtwist}, we reduce to interpolation for a bundle assumed to satisfy
interpolation.
\end{proof}

\begin{lemma} \label{two-secant-backwards} Let $g > 0$ and $r > 3$,
and suppose that $(d, g, r; n)$ satisfies \eqref{regime} and
\[\sum_{i,j,k} k n_{ij}^k \in \{r - 3, r - 2\}.\]
Then $(d, g, r; n)$ is good provided that $(d - 1, g - 1, r; n')$ is good,
where for
\[\sum_{i,j,k} k n_{ij}^k = r - 3,\]
we have
\[(n')_{ij}^k = \begin{cases}
\sum_\ell n_{ij}^\ell & \text{if $k = 0$ and $(i, j; k) \neq (0, 0; 0)$,} \\
1 & \text{if $(i, j; k) = (1, 0; 1)$,} \\
0 & \text{else;}
\end{cases}\]
and for
\[\sum_{i,j,k} k n_{ij}^k = r - 2,\]
we have
\[(n')_{ij}^k = \begin{cases}
\sum_\ell n_{ij}^\ell & \text{if $k = 0$ and $(i, j; k) \neq (0, 0; 0)$,} \\
1 & \text{if $(i, j; k) = (1, 0; 2)$,} \\
0 & \text{else.}
\end{cases}\]
\end{lemma}
\begin{proof}
Again, we degenerate $C$ to $D \cup L$, where $L$ is a $2$-secant line, and all marked points
specialize to points on $D$. Twisting the formula in Lemma~\ref{addtwo} by $-u$,
interpolation for $N_C'$ is reduced to interpolation for
\[N_D'(v)[u \to v][v \to u][v \to 2u] \simeq N_D'(v)[u \to v][v \to 2u][v \to u].\]
We now limit $v \to p$, to reduce our problem to interpolation for
\[N_D'(p)[u \to p][p \to u][p \to 2u] \simeq N_D'(p)[u \to p][p \to 2u][p \to u].\]

Since this is the only case in the paper where the limiting argument
explained in \cref{two-secant} does not apply (the modification datum in question
is not necessarily tree-like), we elaborate further.
As in \cref{two-secant}, write $\mathcal{D} \to B$ for the constant family $D \times B \to B$ over $B$,
where $B \subset D$ is some open set containing $p$.
Then we want a vector bundle on $\mathcal{D}$,
whose restriction to the fiber $D \times \{v\}$ is, for $v \in B$,
\[N_D'(u + v)[u \to v][v \to u][v \to 2u].\]

By minor abuse of notation, we write $u$ for the horizontal divisor $u \times B$,
and $v$ for the diagonal in $D \times B$.
As $u$ is general, $u$ does not meet $p$ or any $q_{ij}^k$.
Moreover, by shrinking $B$, we may suppose $v$ does not intersect the tangent
line to $D$ at $u$; by the machinery of \cref{S:normal-bundles}, we obtain
a subbundle $N_{\mathcal{D} \to v} \subset N_{\mathcal{D}}$.
Applying the theory of \cref{S:modifications-arbitrary}, this subbundle
corresponds to a subbundle $N_{\mathcal{D} \to v}' \subset N_{\mathcal{D}}'$
in a neighborhood of $u$
(where it is tree-like with respect to our modification datum).
However, the subbundles $N_{\mathcal{D} \to u} \subset N_{\mathcal{D}}$ and $N_{\mathcal{D} \to \overline{2u}} \subset N_{\mathcal{D}}$
need not be tree-like with respect to our modification datum.
To get around this, we invoke the theory of \cref{S:modifications-curves}:
Since $D$ is a curve,
the subbundles $N_{D \to u} \subset N_D$ and $N_{D \to \overline{2u}} \subset N_D$
correspond to subbundles $N_{D \to u}' \subset N_D'$ and $N_{D \to \overline{2u}}' \subset N_D'$.
We then let
\[N_{\mathcal{D} \to u}' = \pi^*(N_{D \to u}') \subset \pi^*(N_D') \simeq N_{\mathcal{D}}' \quad \text{and} \quad N_{\mathcal{D} \to \overline{2u}}' = \pi^*(N_{D \to \overline{2u}}')  \subset \pi^*(N_D') \simeq N_{\mathcal{D}}',\]
where $\pi \colon \mathcal{D} = D \times B \to D$ is the projection.
The machinery of Section~\ref{S:modifications-arbitrary} then constructs our desired bundle
\[N_{\mathcal{D}}'(u + v)[u \to N_{\mathcal{D} \to v}'][v \to N_{\mathcal{D} \to u}'][v \to N_{\mathcal{D} \to \overline{2u}}'].\]
(This modification datum is tree-like, since $u$ does not cross $v$,
and $N_{\mathcal{D} \to u}' \subset  N_{\mathcal{D} \to \overline{2u}}'$.)

Moving on, if $\sum_{i,j,k} k n_{ij}^k = r - 3$, we then collect together
the transformations $[p \to k q_{ij}^k]$ and $[p \to 2u]$
(occurring in the right expression)
into a negative twist
via Proposition~\ref{T:combining-modifications}.
This yields another
modified normal bundle, where all points of type $(i, j; k)$ are changed
to type $(i, j; 0)$, and a new point of type $(1, 0; 1)$ is introduced at $u$.
Eliminating all $(0, 0; 0)$-points,
we arrive at the desired conclusion.

Similarly, if $\sum_{i,j,k} k n_{ij}^k = r - 2$, we collect together
the transformations $[p \to k q_{ij}^k]$ and $[p \to u]$
(occurring in the left expression)
into a negative twist
via Proposition~\ref{T:combining-modifications}.
This yields another
modified normal bundle, where all points of type $(i, j; k)$ are changed
to type $(i, j; 0)$, and a new point of type $(1, 0; 2)$ is introduced at $u$.
Eliminating all $(0, 0; 0)$-points,
we arrive at the desired conclusion.
\end{proof}

\begin{lemma} \label{special-5} Let $r = 5$ and $g \geq 2$.
Write $C$ for a general curve of degree $d - 2$ and genus $g - 2$
in $\Pbb^5$, with markings given by $n$, and fix general points $q, x, y \in C$.
Then $(d, g, 5; n)$ is good provided that
\[N_C'[q \to x + y][x + y \to q]\]
satisfies interpolation.

In particular, if $n = \mathbf{0}$, then $(d, g, 5; \mathbf{0})$ is good provided that
$(d - 2, g - 2, 5; n')$ is good, where
\[(n')_{ij}^k = \begin{cases}
2 & \text{if $(i, j; k) = (1, 0; 1)$;} \\
0 & \text{otherwise.}
\end{cases}\]
\end{lemma}
\begin{proof}
We degenerate a general curve of degree $d$ and genus $g$
in $\Pbb^5$ to a union $C \cup L \cup M$, where $L$ and $M$
are $2$-secant lines to $C$, and all marked points specialize to points on $C$.
Write $C \cap L = \{x, z\}$ and $C \cap M = \{y, w\}$.
By \cref{addtwo}, it suffices to show interpolation for
\[N_C'(x + y + z + w)[x \to z][z \to x][z \to 2x][y \to w][w \to y][w \to 2y].\]
Limiting $z$ and $w$ to a common point $q$
reduces the above to interpolation for
\[N_C'(x + y + 2q)[x \to q][q \to x][q \to 2x][y \to q][q \to y][q \to 2y] \simeq N_C'(x + y)[q \to x + y][x + y \to q],\]
which follows from our assumption that $N_C'[q \to x + y][x + y \to q]$
satisfies interpolation.
\end{proof}

We now give several techniques
to reduce from interpolation of modified normal bundles of curves in a give projective space,
to interpolation for curves in a projective space of smaller dimension.
The basic construction here is to add a line transverse to a hyperplane
to a curve contained in that hyperplane. We also explore variants
with adding two lines.

\begin{lemma} \label{stick}
Suppose that
\[2d + 2g - 3r + 2 \leq \sum_{i, j, k} ((r - 2)i + (r - 3)j - k) \cdot n_{ij}^k \leq 2d + 2g - r - 2.\]
If in addition
\[\sum_{i,j,k} k n_{ij}^k < r - 2,\]
then $(d, g, r; n)$ is good provided that
$(d - 1, g, r - 1; n')$ is good, where
\[(n')_{ij}^k = \begin{cases}
\sum_\ell n_{\ell i}^k & \text{if $j = 0$ and $(i, j; k) \neq (0, 0; 0)$;} \\
0 & \text{else.}
\end{cases}\]
If instead
\[\sum_{i,j,k} k n_{ij}^k = r - 2,\]
then $(d, g, r; n)$ is good provided that
$(d - 1, g, r - 1; n')$ is good, where
\[(n')_{ij}^k = \begin{cases}
\sum_{\ell, m} n_{\ell i}^m & \text{if $j = k = 0$ but $i \neq 0$;} \\
0 & \text{else.}
\end{cases}\]
\end{lemma}
\begin{proof}
We degenerate $C$ to $D \cup L$ where $D \subset H$ is a curve contained in a hyperplane,
and $L$ is a $1$-secant line to $D$ transverse to $H$.
We specialize so $p \in L$ and all other marked points lie on $D$.
Clearly, it suffices to show
\[\mathcal{E} = N_{D \cup L}'(-z)\]
satisfies interpolation, where $z \in L$. From Lemma~\ref{addone} (with $\Lambda_1 = \overline{kq_{ij}^k}$
and $\Lambda_2 = \emptyset$), we conclude it is sufficient to prove interpolation for the bundle
\[\mathcal{E} = N_D\big((i + j - 1) q_{ij}^k \big)[i q_{ij}^k \to p][j q_{ij}^k \to 2p](x)[x \to p]\left[x \to p + \sum k q_{ij}^k\right]\]
Identifying $\Oc_D(1)$ with the normal bundle of $D$ in the cone $\overline{p D}$, we obtain
a splitting:
\[N_D \simeq N_{D/H} \oplus \Oc_D(1).\]
This induces a splitting $\mathcal{E} \simeq \mathcal{F} \oplus \mathcal{L}$ with
\[\mathcal{F} = N_{D/H} \big((j - 1) q_{ij}^k\big) [jq_{ij}^k \to x] [x \to kq_{ij}^k] \quad \text{and} \quad \mathcal{L} = \Oc_D(1) \big(x + (i + j - 1) q_{ij}^k\big).\]

Now we claim $\mathcal{F}$ satisfies interpolation.
Indeed, when $\sum k n_{ij}^k < r - 2$, then $\mathcal{F}$ is a modified normal bundle
of the type assumed to satisfy interpolation.
Otherwise, when $\sum k n_{ij}^k = r - 2$, then $\mathcal{F}(-x)$
is a modified normal bundle of the type assumed to satisfy interpolation.

Next, $\mathcal{L}$
satisfies interpolation since $\Oc_D(1)$ satisfies interpolation. So to check $\mathcal{F} \oplus \mathcal{L}$
satisfies interpolation, we just need to check
\[(r - 2) \cdot (\chi(\mathcal{L}) - 1) \leq \chi(\mathcal{F}) \leq (r - 2) \cdot (\chi(\mathcal{L}) + 1).\]
For this, we first calculate
\begin{align*}
\chi(\mathcal{F}) &= r(d - 1) - (r - 4)g - 2 - \sum_{i,j,k} (r - 2 - j - k) \cdot n_{ij}^k, \\
\chi(\mathcal{L}) &= (d - 1) - g + 1 + 1 + \sum_{i, j, k} (i + j - 1) \cdot n_{ij}^k \\
&= d - g + 1 + \sum_{i, j, k} (i + j - 1) \cdot n_{ij}^k.
\end{align*}
The condition for $\mathcal{F} \oplus \mathcal{L}$ to satisfy interpolation is then
\[2d + 2g - 3r + 2 \leq \sum_{i, j, k} ((r - 2)i + (r - 3)j - k) \cdot n_{ij}^k \leq 2d + 2g - r - 2. \qedhere\]
\end{proof}

\begin{lemma} \label{two-sticks} Suppose that $r > 3$ and
\[2d + 2g - 4r + 3 \leq \sum_{i,j,k} ((r - 2)i + (r - 3)j - k) \cdot n_{ij}^k \leq 2d + 2g - 2r - 1.\]
If in addition
\[\sum_{i,j,k} k n_{ij}^k < r - 3,\]
then $(d, g, r; n)$ is good provided that $(d - 2, g - 1, r - 1; n')$ is good, where
\[(n')_{ij}^k = \begin{cases}
\sum_\ell n_{\ell i}^k & \text{if $j = 0$ and $(i, j; k) \notin \{(0, 0; 0), (2, 0, 1)\}$;} \\
1 + \sum_\ell n_{\ell i}^k & \text{if $(i, j; k) = (2, 0, 1)$;} \\
0 & \text{else.} \\
\end{cases}\]
If instead
\[\sum_{i,j,k} k n_{ij}^k = r - 3,\]
then $(d, g, r; n)$ is good provided that $(d - 2, g - 1, r - 1; n')$ is good, where
\[(n')_{ij}^k = \begin{cases}
\sum_{\ell, m} n_{\ell i}^m & \text{if $j = k = 0$ and $i \notin \{0, 2\}$;} \\
1 + \sum_{\ell,m} n_{\ell i}^m & \text{if $j = k = 0$ and $i = 2$;} \\
0 & \text{else.} \\
\end{cases}\]
If instead
\[\sum_{i,j,k} k n_{ij}^k = r - 2,\]
then $(d, g, r; n)$ is good provided that $(d - 2, g - 1, r - 1; n')$ is good, where
\[(n')_{ij}^k = \begin{cases}
\sum_{\ell,m} n_{\ell i}^m & \text{if $j = k = 0$ and $i \neq 0$;} \\
1 & \text{if $(i, j; k) = (2, 0, 1)$;} \\
0 & \text{else.} \\
\end{cases}\]
\end{lemma}
\begin{proof}
We degenerate $C$ to $D \cup L \cup M$ where $D \subset H$ is a curve contained in a hyperplane, and
$L$ and $M$ are $1$-secant lines to $D$ transverse to $H$, which meet at some point $q \notin H$.
Write $x = L \cap D$ and $s = M \cap D$.

We specialize so $p \in L$ and all other marked points lie on $D$.
Applying Lemma~\ref{addtwo}, and twisting by $-q$, it suffices to show
\[N_{D \cup L}'(s)[q \to s][s \to q][s \to 2q] = N_{D \cup L}'(s)[q \to s][s \to q][s \to p \cup x]\]
satisfies interpolation.

First suppose $\sum k n_{ij}^k < r - 2$. Then limiting $q \to x$, we conclude from
Lemma~\ref{addone} (with $\Lambda_1 = s$ and $\Lambda_2 = \overline{kq_{ij}^k}$) that it is
sufficient to prove interpolation for the bundle
\[\mathcal{E} = N_D\big((i + j - 1) q_{ij}^k \big)[iq_{ij}^k \to p][j q_{ij}^k \to 2p](s)[s \to x][s \to p \cup x](x)[x \to p]\left[x \to p + s + \sum kq_{ij}^k\right].\]
Similarly, for $\sum k n_{ij}^k = r - 2$, we conclude by limiting $q \to x$ and
applying Lemma~\ref{fiddle-over} that it is sufficient to prove interpolation
for the bundle
\[\mathcal{E}' = N_D\big( (i + j - 1) q_{ij}^k \big)[iq_{ij}^k \to p][j q_{ij}^k \to 2p](s)[s \to x][s \to p \cup x][x \to p \cup s].\]
Identifying $\Oc_D(1)$ with the normal bundle of $D$ in the cone $\overline{pD}$, we obtain the splitting
\[N_D \simeq N_{D/H} \oplus \Oc_D(1).\]
This induces splittings $\mathcal{E} \simeq \mathcal{F} \oplus \mathcal{L}$ and $\mathcal{E}' \simeq \mathcal{F}' \oplus \mathcal{L}'$, where:
\begin{align*}
\mathcal{F} &= N_{D/H}\big(s + (j - 1) q_{ij}^k \big)[j q_{ij}^k \to x][2s \to x]\left[x \to s + \sum kq_{ij}^k\right], \\
\mathcal{L} &= \Oc_D(1)\big(x + (i + j - 1) q_{ij}^k \big), \\
\mathcal{F}' &= N_{D/H}\big(s + (j - 1) q_{ij}^k \big)[j q_{ij}^k \to x][2s \to x][x \to s], \\
\mathcal{L}' &= \Oc_D(1)\big((i + j - 1) q_{ij}^k \big).
\end{align*}

Now we claim $\mathcal{F}$, respectively $\mathcal{F}'$, satisfies interpolation.
Indeed, when $\sum k n_{ij}^k < r - 3$, then $\mathcal{F}$ is a modified normal bundle
of the type assumed to satisfy interpolation.
Otherwise, when $\sum n_{ij}^k = r - 3$, then $\mathcal{F}(-x)$
is a modified normal bundle of the type assumed to satisfy interpolation.
Finally, when $\sum n_{ij}^k = r - 2$, then $\mathcal{F}'$ is a modified normal bundle
of the type assumed to satisfy interpolation.

Next, $\mathcal{L}$, respectively $\mathcal{L}'$,
satisfies interpolation since $\Oc_D(1)$ satisfies interpolation. So to check $\mathcal{F} \oplus \mathcal{L}$,
respectively $\mathcal{F}' \oplus \mathcal{L}'$,
satisfies interpolation, we just need to check
\begin{align*}
(r - 2) \cdot (\chi(\mathcal{L}) - 1) &\leq \chi(\mathcal{F}) \leq (r - 2) \cdot (\chi(\mathcal{L}) + 1), \\
(r - 2) \cdot (\chi(\mathcal{L}') - 1) &\leq \chi(\mathcal{F}') \leq (r - 2) \cdot (\chi(\mathcal{L}') + 1).
\end{align*}

For this, we first calculate
\begin{align*}
\chi(\mathcal{F}) &= r(d - 2) - (r - 4)(g - 1) - 2 - (r - 5) - \sum_{i,j,k} (r - 2 - j - k) \cdot n_{ij}^k \\
\displaybreak[0]
&= r(d - 2) - (r - 4)g - 1 - \sum_{i,j,k} (r - 2 - j - k) \cdot n_{ij}^k, \\
\chi(\mathcal{L}) &= (d - 2) - (g - 1) + 1 + 1 + \sum_{i, j, k} (i + j - 1) \cdot n_{ij}^k \\
\displaybreak[0]
&= d - g + 1 + \sum_{i, j, k} (i + j - 1) \cdot n_{ij}^k, \\
\chi(\mathcal{F}') &= r(d - 2) - (r - 4)(g - 1) - 2 - (r - 5) - \sum_{i,j,k} (r - 2 - j) \cdot n_{ij}^k \\
\displaybreak[0]
&= r(d - 2) - (r - 4)g - 1 - (r - 2) - \sum_{i,j,k} (r - 2 - j - k) \cdot n_{ij}^k, \\
\chi(\mathcal{L}') &= (d - 2) - (g - 1) + 1 + \sum_{i, j, k} (i + j - 1) \cdot n_{ij}^k \\
&= d - g + \sum_{i, j, k} (i + j - 1) \cdot n_{ij}^k.
\end{align*}

Substituting this into the above,
the condition for either $\mathcal{F} \oplus \mathcal{L}$, or $\mathcal{F}' \oplus \mathcal{L}'$, to satisfy interpolation is then
\[2d + 2g - 4r + 3 \leq \sum_{i,j,k} ((r - 2)i + (r - 3)j - k) \cdot n_{ij}^k \leq 2d + 2g - 2r - 1. \qedhere\]
\end{proof}

\noindent
Next, we give an inductive construction based around adding a $1$-secant line to $C$.

\begin{lemma} \label{lower-d}
Let $d > g + r$. If $r = 3$, then assume in addition
that $\sum_{i,j,k} jn_{ij}^k = 0$. Suppose that
\[\sum_{i,j,k} (r - 1 - i - 2j - k) n_{ij}^k \leq (r + 1) d - (2r - 4)g - 2.\]
Then $(d, g, r; n)$ is good provided that both
$(d - 1, g, r; n)$ and $(d - 1, g, r - 1; n')$ are good, where if
\[\sum_{i,j,k} k n_{ij}^k < r - 2,\]
then
\[n' = n;\]
and if instead
\[\sum_{i,j,k} k n_{ij}^k = r - 2,\]
then
\[(n')_{ij}^k = \begin{cases}
\sum_\ell n_{ij}^\ell & \text{if $k = 0$ and $(i, j; k) \neq (0, 0; 0)$,} \\
0 & \text{else.}
\end{cases}\]
\end{lemma}
\begin{proof}
We degenerate $C$ to $D \cup L$, where $D$ is a nondegenerate curve and $L$ is a $1$-secant
line to $D$. We suppose that $s := C \cap L$ is general, and we write $q$ for some other point on $L$.
For $q$ general, projection from $q$ gives a local immersion from $D$ to $\Pbb^{r - 1}$.

The image $D'$ under projection is by construction a general curve of degree $d - 1$ and genus $g$
in $\Pbb^{r - 1}$, and all marked points on $D'$ are general. We now consider
\[\wtilde{N}_{D'}' := N_{D'}\big((i + j - 1)q_{ij}^k\big)[p \to k q_{ij}^k][iq_{ij}^k \to p][j q_{ij}^k \to 2p].\]

For for $\sum_{i,j,k} k n_{ij}^k < r - 2$, this is a modified normal of the type
assumed to satisfy interpolation. Otherwise, for $\sum_{i,j,k} k n_{ij}^k = r - 2$, then 
$\wtilde{N}_{D'}'(-p)$ is a modified normal of the type
assumed to satisfy interpolation. Either way, we conclude that $\wtilde{N}_{D'}'$ satisfies interpolation.

Our assumed inequality implies, via Proposition~\ref{T:interpolation-twist-down}, that
it suffices to prove interpolation for the bundle $N_C'(\Delta_0)$ for any effective divisor $\Delta_0 \subset D \subset C$.
From Lemma~\ref{addone}
(with $\Lambda_1 = \Lambda_2 = \emptyset$), this in turn reduces to interpolation for
\[N_D'(\Delta_0)(s)[2s \to q].\]
Taking $\Delta_0 = \Delta + s$ for $\Delta$ a general effective divisor of large degree,
it suffices to prove interpolation for
\[N_D'(\Delta)(2s)[2s \to q] = N_D'(\Delta)(2s)[2s \to N_{D \to q}(-p - q_{ij}^k)].\]
Because the quotient
\[N_D'(\Delta) / N_{D \to q}(\Delta) \simeq \wtilde{N}_{D'}'(\Delta)\]
satisfies interpolation, we can apply Proposition~\ref{T:interpolation-positive}
to reach the desired conclusion, subject to the inequality
\[(r - 2) \cdot \big(\chi(N_{D \to q}(-p - q_{ij}^k)(\Delta)) + 1\big) \leq \chi(N_D'(\Delta)) - \chi(N_{D\to q}(-p - q_{ij}^k)(\Delta)).\]
(For $\Delta$ of large degree, $N_{D \to q}(-p - q_{ij}^k)(\Delta)$ will be nonspecial.)
This inequality is in turn equivalent to the inequality
\[(r - 2) \cdot \big(\chi(N_{D \to q}(-p - q_{ij}^k)) + 1\big) \leq \chi(N_D') - \chi(N_{D\to q}(-p - q_{ij}^k)).\]
By Proposition~\ref{to-Lambda} and Lemma~\ref{chimod}, we have
\begin{gather*}
\chi(N_{D \to q}(-p - q_{ij}^k)) = (d - 1) - g + 1 - 1 - \sum_{i,j,k} n_{ij}^k \\
\chi(N_D') = (r + 1)(d - 1) - (r - 3)g - 2 - \sum_{i,j,k} (r - 1 - i - 2j - k) n_{ij}^k.
\end{gather*}
Thus, we just need to check the inequality
\[r + 2 \leq 2d + 2g + \sum_{i,j,k} (i + 2j + k) n_{ij}^k.\]
But this inequality holds since
$2d + 2g \geq 2(g + r + 1) + 2g \geq 2r + 2 \geq r + 2$ by assumption.
\end{proof}

We now give several methods for getting rid of marked points
(without changing the degree or genus of our curve). 

\begin{lemma} \label{izz}
Suppose that $(d, g, r; n)$ satisfies \eqref{regime}, and that
\[\sum_{i,j,k} (r - 1 - i - 2j - k) \cdot n_{ij}^k \leq (r + 1) d - (2r - 4)g - 2.\]
If there is some integer $\ell$ with
and $n_{\ell 0}^0 > 0$, then $(d, g, r; n)$ is good provided that both $(d, g, r; n')$
and $(d - 1, g, r - 1; n'')$
are good. Here, we define
\[(n')_{ij}^k = \begin{cases}
n_{ij}^k - 1 & \text{if $(i, j; k) = (\ell, 0; 0)$;} \\
n_{ij}^k & \text{else.}
\end{cases}\]
In addition, if 
\[\sum_{i,j,k} k n_{ij}^k < r - 2,\]
then we define
\[(n'')_{ij}^k = \begin{cases}
\sum_\ell n_{\ell i}^k & \text{if $j = 0$ and $(i, j; k) \neq (0, 0; 0)$;} \\
0 & \text{else.}
\end{cases}\]
If instead
\[\sum_{i,j,k} k n_{ij}^k = r - 2,\]
then we define 
\[(n'')_{ij}^k = \begin{cases}
\sum_{\ell, m} n_{\ell i}^m & \text{if $j = k = 0$ but $i \neq 0$;} \\
0 & \text{else.}
\end{cases}\]
\end{lemma}
\begin{proof}
Let $q$ be some point of type $(\ell, 0; 0)$. From Proposition~\ref{T:interpolation-twist-down},
it is sufficient to prove that $N_C'(q)$ satisfies interpolation, since the
given inequality rearranges to $\chi(N_C') \geq (r - 1) g$ (c.f.\ Lemma~\ref{chimod}).
By assumption, some modification $N_C''$ of type $(d, g, r; n')$ satisfies interpolation.
Then we can write
\[N_C'(q) = N_C''(\ell q)[\ell q \to p],\]
Proposition~\ref{T:interpolation-positive} implies this satisfies interpolation,
as long as $N_C''$, $N_{C \to p}''$, and $N_C'' / N_{C \to p}''$ all satisfy interpolation,
and $(r - 2)(\chi(N_{C \to p}'') + \ell - 1) \leq \chi(N_C'' / N_{C \to p}'')$.

We first note that $N_C''$ satisfies interpolation, since $(d, g, r; n')$
is good by assumption; in addition, $N_{C \to p}''$ satisfies interpolation, since it
is a nonspecial line bundle by \cref{top}. Write $\tilde{C}$ for the proper transform of $C$
in the blowup $\operatorname{Bl}_p \Pbb^r$; let $\bar{C} \subset \Pbb^{r - 1}$ denote the projection of $C$ from $p$.
Then using the exact sequence
\[0 \to N_{C \to p}(-p) \to N_C(-p) \simeq N_{\tilde{C} / \operatorname{Bl}_p \Pbb^r} \to N_{\bar{C} / \Pbb^{r - 1}} \to 0,\]
we recognize $(N_C'' / N_{C \to p}'')(-\alpha p)$ as a modified
normal bundle, for $\bar{C}$, of type $(d - 1, g, r - 1; n'')$, where
\[\alpha = \begin{cases}
2 & \text{if $\sum_{i,j,k} k n_{ij}^k < r - 2$;} \\
1 & \text{if $\sum_{i,j,k} k n_{ij}^k = r - 2$.}
\end{cases}\]
In particular, our assumption that $(d - 1, g, r - 1; n'')$ is good implies
that $(N_C'' / N_{C \to p}'')(-\alpha p)$, and thus $N_C'' / N_{C \to p}''$,
satisfies interpolation.
It thus remains only to check
\[(r - 2)(\chi(N_{C \to p}'') + \ell - 1) \leq \chi(N_C'' / N_{C \to p}'').\]
To do this, we first calculate (using \cref{chimod} and \cref{top}):
\begin{align*}
\chi(N_{C \to p}'') &= d - g + 2 + \sum_{i,j,k} (i + j - 1) (n')_{ij}^k = d - g + 3 - \ell + \sum_{i,j,k} (i + j - 1) n_{ij}^k \\
\chi(N_C'' / N_{C \to p}'') &= (r - 2) + r(d - 1) - (r - 4)g - 2 - \sum_{i,j,k} (r - 2 - j - k)(n')_{ij}^k \\
&= rd - (r - 4)g + r - 6 - \sum_{i,j,k} (r - 2 - j - k)n_{ij}^k.
\end{align*}
Substituting the above expressions into our desired inequality
reduces it to \eqref{regime}, which holds by assumption.
\end{proof}

\begin{lemma} \label{ijz}
Let $r = 3$, and suppose that $(d, g, r; n)$ satisfies \eqref{regime}.
If there are integers $\ell$ and $m \geq 1$ with
$n_{\ell m}^0 > 0$, then $(d, g, r; n)$ is good provided that $(d, g, r; n')$
is good, where
\[(n')_{ij}^k = \begin{cases}
n_{ij}^k - 1 & \text{if $(i, j; k) = (\ell, m; 0)$;} \\
n_{ij}^k & \text{else.}
\end{cases}\]
\end{lemma}
\begin{proof}
Let $q$ be some point of type $(\ell, m; 0)$. Then since $m \geq 1$,
it is sufficient to prove that $N_C'\big(-(m - 1)q\big)$ satisfies interpolation.
By assumption, some modification $N_C''$ of type $(d, g, r; n')$ satisfies interpolation.
Then we can write
\[N_C'\big(-(m - 1)q\big) = N_C''(\ell q)[\ell q \to p].\]
Proposition~\ref{T:interpolation-positive} implies this satisfies interpolation,
as long as $N_C''$, $N_{C \to p}''$, and $N_C'' / N_{C \to p}''$ all satisfy interpolation,
and $\chi(N_{C \to p}'') + \ell - 1 \leq \chi(N_C'' / N_{C \to p}'')$.

We first note that $N_C''$ satisfies interpolation, since $(d, g, r; n')$
is good by assumption. For the remaining conditions, we first
note that by \cref{to-kp},
\begin{align*}
N_{C \to p}'' &\simeq \Oc_C(1)(p)\big((i + j - 1) (q')_{ij}^k\big) \\
\wedge^2 N_C'' &\simeq (\wedge^2 N_C)\big((i + 2j - 2) (q')_{ij}^k\big)\big((k (n')_{ij}^k) \cdot p\big)(-2p) \\
&\simeq K_C(4)\big((i + 2j - 2) (q')_{ij}^k\big)\big((k (n')_{ij}^k) \cdot p\big)(-2p).
\end{align*}
In particular,
\[N_C'' / N_{C \to p}'' \simeq K_C(3)\big((j - 1) (q')_{ij}^k\big)\big(((k (n')_{ij}^k) - 3) \cdot p\big).\]

These expressions imply in particular that $N_{C \to p}''$ and $N_C'' / N_{C \to p}''$
are general line bundles on~$C$.
By \cref{top}, the line bundle $N_{C \to p}''$ is nonspecial;
since $\ell \geq 1$, it therefore remains only to check
\[\chi(N_{C \to p}'') + \ell - 1 \leq \chi(N_C'' / N_{C \to p}'').\]
For this, we calculate
\[\chi(N_C'' / N_{C \to p}'') - \chi(N_{C \to p}'') = 2g + 2g - 6 - \sum_{i,j,k} (i - k) \cdot (n')_{ij}^k = 2g + 2g - 6 + \ell - \sum_{i,j,k} (i - k) \cdot n_{ij}^k.\]
From \eqref{regime},
\[\sum_{i,j,k} (i - k) \cdot n_{ij}^k \leq 2d + 2g - 5,\]
which implies
\[\chi(N_C'' / N_{C \to p}'') - \chi(N_{C \to p}'') \geq 2g + 2g - 6 + \ell - (2d + 2g - 5) = \ell - 1,\]
as desired.
\end{proof}

\begin{lemma} \label{zzk}
Suppose again that
\[\sum_{i,j,k} (r - 1 - i - 2j - k) \cdot n_{ij}^k \leq (r + 1) d - (2r - 4)g - 2.\]
If there is some integer $\ell$ with
and $n_{00}^\ell > 0$, then $(d, g, r; n)$ is good provided that $(d, g, r; n')$ and $(d, g, r; n'')$
are both good, where
\begin{align*}
(n')_{ij}^k &= \begin{cases}
n_{ij}^k - 1 & \text{if $(i, j; k) = (0, 0; \ell)$;} \\
n_{ij}^k & \text{else.}
\end{cases} \\
(n'')_{ij}^k &= \begin{cases}
\sum_m n_{ij}^m & \text{if $(i, j) \neq (0, 0)$ and $k = 0$;} \\
0 & \text{else.}
\end{cases}
\end{align*}
\end{lemma}
\begin{proof}
Let $q$ be some point of type $(0, 0; \ell)$. Again from Proposition~\ref{T:interpolation-twist-down},
it is sufficient to prove that $N_C'(q)$ satisfies interpolation. Write $(q_{ij}^k)^\circ$
for the divisor $q_{ij}^k$, minus $q$ if $(i, j; k) = (0, 0; \ell)$.
Then by assumption, some modification $N_C''$ of type $(d, g, r; n'')$ satisfies interpolation.
Then we can write
\[N_C'(q) = N_C''\big(-(q_{ij}^k)^\circ\big)(p)[p \to kq_{ij}^k].\]
Since $N_C''$ satisfies interpolation, we conclude that $N_C''\big(-(q_{ij}^k)^\circ\big)(p)$ does as well.
So applying \cref{T:linearly-general-interpolation},
we conclude that $N_C'$ satisfies interpolation.
\end{proof}


\section{Base cases}
\label{S:base}

In this section, we prove interpolation in a number of special cases,
which will form the base cases for our inductive argument.

\begin{lemma} \label{p2}
Suppose that $r = 2$ and $\sum_{i,j,k} j n_{ij}^k = 0$. Then $(d, g, r; n)$ is good.
\end{lemma}
\begin{proof}
Our earlier assumption that $\sum_{i,j,k} k n_{ij}^k < r - 1 = 1$ implies, together with our given assumption,
that $n_{ij}^k = 0$ unless $j = k = 0$. The modified normal bundle in this case is then
\[N_C' = N_C\big((i - 1) q_{ij}^k\big).\]
But from Lemma~\ref{nonspecial-smooth}, we have $\H^1(N_C) = 0$; consequently, $\H^1(N_C') = 0$.
Because $N_C$ is a line bundle, this implies $N_C$ satisfies interpolation.
\end{proof}

\begin{lemma} \label{p3:mid} Suppose that $r = 3$ and $g > 0$, and that
\[2d + 2g - 9 \leq \sum_{i,j,k} (i - k) \cdot n_{ij}^k \leq 2d + 2g - 7.\]
Then $(d, g, r; n)$ is good.
\end{lemma}
\begin{proof}
We degenerate $C$ to $D \cup L \cup M$ where $D \subset H$ is a curve contained in a plane, and
$L$ and $M$ are $1$-secant lines to $D$ transverse to $H$, which meet at some point $q \notin H$.
Write $x = L \cap D$ and $s = M \cap D$.

We specialize so $p \in L$ and all other marked points lie on $D$.
Applying Lemma~\ref{addtwo}, and twisting by $-q$, it suffices to show
\[N_{D \cup L}'(s + q)[q \to s][s \to q][s \to 2q](-q) = N_{D \cup L}'(s)[q \to s][s \to q]\]
satisfies interpolation.

First suppose $\sum k n_{ij}^k = 0$. Then limiting $q \to x$, we conclude from Lemma~\ref{addone}
(with $\Lambda_1 = s$ and $\Lambda_2 = \emptyset$) that
it is sufficient to prove interpolation for the bundle
\[\mathcal{E}_0 = N_D\left(\sum (i + j - 1) q_{ij}^k \right)[i q_{ij}^k \to p][j q_{ij}^k \to 2p](s)[s \to x](x)[x \to p]\left[x \to p \cup s\right].\]

Similarly, suppose $\sum k n_{ij}^k = 1$. Then limiting $q \to x$, we conclude from
Lemma~\ref{fiddle-over} that it is sufficient to prove interpolation for the bundle
\[\mathcal{E}_1 = N_D\left(\sum (i + j - 1) q_{ij}^k \right)[i q_{ij}^k \to p][j q_{ij}^k \to 2p](s)[s \to x][x \to p \cup s].\]
Identifying $\Oc_D(1)$ with its normal bundle in the cone $\overline{p D}$, we obtain
a splitting:
\[N_D \simeq N_{D/H} \oplus \Oc_D(1).\]
This induces splittings $\mathcal{E}_\alpha \simeq \mathcal{F} \oplus \mathcal{L}_\alpha$ for $\alpha = 0, 1$ with
\[\mathcal{F} = N_{D/H} \big(s + (j - 1) q_{ij}^k\big) [jq_{ij}^k \to x] \quad \text{and} \quad \mathcal{L}_\alpha = \Oc_D(1) \left((1 - \alpha)x + \sum (i + j - 1) q_{ij}^k\right).\]

Both $\mathcal{F}$ and $\mathcal{L}_\alpha$ satisfy interpolation:
for $\mathcal{F}$ this follows from Lemma~\ref{p2}, and for $\mathcal{L}_\alpha$
this is immediate from $\H^1(\Oc_D(1)) = 0$.
So to check $\mathcal{E}_\alpha = \mathcal{F} \oplus \mathcal{L}_\alpha$
satisfies interpolation, we just need to check
\[|\chi(\mathcal{L}_\alpha) - \chi(\mathcal{F})| \leq 1.\]
For this, we first calculate
\begin{align*}
\chi(\mathcal{F}) &= 3(d - 2) + (g - 2) + 1 + \sum (j - 1) n_{ij}^k \\
&= 3d + g - 7 + \sum (j - 1) n_{ij}^k. \\
\chi(\mathcal{L}_\alpha) &= (d - 2) - (g - 1) + 1 + 1 - \alpha + \sum_{i, j, k} (i + j - 1) \cdot n_{ij}^k \\
&= d - g + 1 - \sum_{i,j,k} k n_{ij}^k + \sum_{i, j, k} (i + j - 1) \cdot n_{ij}^k.
\end{align*}
We conclude that
\begin{align*}
|\chi(\mathcal{L}_\alpha) - \chi(\mathcal{F})| &= \left|d - g + 1 - \sum_{i,j,k} k n_{ij}^k + \sum (i + j - 1) \cdot n_{ij}^k - 3d - g + 7 - \sum (j - 1) n_{ij}^k\right| \\
&= \left|\sum (i - k) \cdot n_{ij}^k - 2d - 2g + 8\right| \\
&\leq 1. && \qedhere
\end{align*}
\end{proof}

\begin{lemma} \label{last-case} For $g \geq 1$, the tuple $(5g + 1, g, 4g + 1; \mathbf{0})$ is good.
\end{lemma}
\begin{proof}
We will construct an explicit curve of degree $5g + 1$ and genus $g$ in $\Pbb^{4g + 1}$
(with no marked points); and check directly that its normal bundle satisfies interpolation.

Let $D \subset \Pbb^{4g + 1}$ be a rational normal curve, and let
$L_1, L_2, \ldots, L_g$ be a collection of $g$ lines which are
$2$-secant to $D$.
Then by construction, $D \cup L_1 \cup \cdots \cup L_g$ is a curve
of degree $5g + 1$ and genus $g$ in $\Pbb^{4g + 1}$.

Our task is now to show $N_D$ satisfies interpolation. Write $L_i \cap D = \{x_i, y_i\}$.
Then from Corollary~\ref{to-kp-rat}, we obtain
\[N_{D \to 2x_i} \simeq N_{D \to 2y_i} \simeq \Oc_{\Pbb^1}(4g + 3)^2.\]
We have a map of vector bundles
\[\bigoplus_i \left(N_{D \to 2x_i} \oplus N_{D \to 2y_i}\right) \to N_D,\]
which is an isomorphism over the generic point. Moreover,
\[\chi\left(\bigoplus_i \left(N_{D \to 2x_i} \oplus N_{D \to 2y_i}\right)\right) = 4g(4g + 4) = \chi(N_D);\]
which implies that in fact the above map yields an isomorphism
\begin{equation} \label{dirsum}
N_D \simeq \bigoplus_i \left(N_{D \to 2x_i} \oplus N_{D \to 2y_i}\right).
\end{equation}

Now from Lemma~\ref{addtwo-two}, writing $z_i \in T_{x_i} D$ and $w_i \in T_{y_i} D$ for points
distinct from $x_i$ and $y_i$ respectively, it suffices to show interpolation for every
\[H^0 N_D(2x_i + 2y_i)[x_i \to y_i][y_i \to x_i]) \langle x_i \to y_i : y_i \to z_i \rangle \langle y_i \to x_i : x_i \to w_i \rangle.\]
Happily, each of these transformations respects the direct sum decomposition \eqref{dirsum}: the above
space of sections is a direct sum of spaces of sections of each bundle on the RHS of \eqref{dirsum}.
By symmetry, it therefore suffices to show interpolation for every
\begin{multline} \label{corr-bundle}
\H^0\left(N_{D \to 2x_1} \oplus N_{D \to 2y_1}\right)(2x_1 + 2y_1 - x_2 - y_2 - \cdots - x_g - y_g)\\
[x_1 \to y_1][y_1 \to x_1] \langle x_1 \to y_1 : y_1 \to z_1 \rangle \langle y_1 \to x_1 : x_1 \to w_1 \rangle.
\end{multline}

Now given a bundle $\mathcal{E} \simeq \mathcal{L} \oplus \mathcal{L}$
with $\mathcal{L}$ a line bundle on a variety $X$,
splittings of $\mathcal{E}$ are in bijection with splittings of any fiber $\mathcal{E}|_x$ (for $x \in X$).
In particular, given such a bundle, an
inclusion $\iota \colon \mathcal{L} \hookrightarrow \mathcal{E}$, and any vector
$v \in \mathcal{E}|_x$ for some $x \in X$ with $v \notin \iota(\mathcal{L}|_x)$,
there is a splitting so that $\iota$ is inclusion into the first factor and $v$
is an element of the fiber of the second factor. Applying this here, we can choose
a splitting
\[N_{D \to 2x_1} \simeq N_{D \to x_1} \oplus N_{D \to x_1}^\perp \quad \text{with} \quad N_{D \to x_1} \simeq N_{D \to x_1}^\perp \simeq \Oc_{\Pbb^1}(4g + 3),\]
so that
\[N_{D \to x_1}^\perp|_{y_1} = N_{D\to z_1}|_{y_1} \subset N_{D \to 2x_1}|_{y_1}.\]

Similarly, we define $N_{D \to y_1}^\perp$. Then
thanks to Remark~\ref{T:only-neighborhood}, the space of sections \eqref{corr-bundle}
splits as a direct sum of two spaces of sections: the space
\begin{multline*}
\H^0\left(N_{D \to x_1} \oplus N_{D \to y_1}^\perp\right)(2x_1 + 2y_1 - x_2 - y_2 - \cdots - x_g - y_g)\\
(-x_1)[y_1 \to x_1](-x_1 - y_1) \langle y_1 \to x_1 : x_1 \to w_1 \rangle,
\end{multline*}
and the space obtained by reversing the roles of $x$ and $y$ above.
It thus, by symmetry, suffices to prove interpolation for the above space of sections; we can 
rewrite this as
\[\H^0\left(N_{D \to x_1} \oplus N_{D \to y_1}^\perp\right)(y_1 - x_2 - y_2 - \cdots - x_g - y_g) [y_1 \to N_{D \to x_1}] \langle y_1 \to N_{D \to x_1} : x_1 \to N_{D \to y_1}^\perp \rangle.\]
Under the isomorphisms $N_{D \to x_1} \simeq N_{D \to x_1}^\perp \simeq \Oc_{\Pbb^1}(4g + 3)$, 
the above space of sections becomes
\[\H^0\left(\Oc_{\Pbb^1}(2g + 6) \oplus \Oc_{\Pbb^1}(2g + 5)\right)\langle y_1 \to \Oc_{\Pbb^1}(2g + 6) : x_1 \to \Oc_{\Pbb^1}(2g + 5)\rangle.\]

This is some subspace of sections which is codimension one in
\[
\H^0\left(\Oc_{\Pbb^1}(2g + 6) \oplus \Oc_{\Pbb^1}(2g + 5)\right),
\]
but which by definition in particular does not contain the first factor $\H^0(\Oc_{\Pbb^1}(2g + 6))$.
But by direct inspection, every such subspace of sections satisfies interpolation.
\end{proof}

\begin{lemma} \label{835}
The tuple $(8, 3, 5; n)$ is good, where
\[n_{ij}^k = \begin{cases}
2 & \text{if $(i, j; k) = (1, 0; 1)$;} \\
0 & \text{otherwise.}
\end{cases}\]
\end{lemma}
\begin{proof}
Write $q_{10}^1 = \{s, t\}$, and write $E$ for a general
curve of degree $6$ and genus $1$ in $\Pbb^5$.
Applying \cref{special-5}, we need to prove interpolation for
\[\mathcal{F} = N_E'[q \to x + y][x + y \to q] = N_E[s + t \to p][p \to s + t][q \to x + y][x + y \to q].\]

Degenerate $E$ to $C \cup L$, where $C$ is a rational normal curve,
and $L$ is a $2$-secant line.
We specialize $s$ and $x$ to $L$, and all other marked
points to $C$.
Write $E \cap L = \{u, v\}$, and let $z \in T_u E$ and $w \in T_v E$
be points distinct from $u$ and $v$ respectively.
Then by \cref{hh},
\[\mathcal{F}|_L = N_L(u + v)[u \to z][v \to w][s \to p][x \to q] \simeq \Oc_L^{\oplus 4}.\]

Applying \cref{trivglue} to $\mathcal{F}$, with $D$ a single general point
on $L$, it suffices to show that the space of sections
\[\H^0(\mathcal{F}|_C(-u - v)) = \ev_C^{-1}(\ev_L(\H^0(\mathcal{F}|_L(-D)))) \subseteq \H^0(\mathcal{F}|_C)\]
satisfies interpolation and has dimension
\[\chi(\mathcal{F}|_C) + \chi(\mathcal{F}|_L(-D - x - y)) = \chi(\mathcal{F}|_C(-x-y)).\]
In other words, it remains to prove interpolation for the bundle
\[\mathcal{F}|_C(-u - v) = N_C[t \to p][p \to s + t][q \to x + y][y \to q][u \to v][v \to u].\]
Limiting $x \to p$ and $s \to q$, this reduces to interpolation for
\[N_C[t \to p][p \to q + t][q \to p + y][y \to q][u \to v][v \to u].\]
Further limiting $y \to u$ and $t \to v$, this reduces to interpolation for
\begin{align*}
N_C & [v \to p][p \to q + v][q \to p + u][u \to q][u \to v][v \to u] \\
&\simeq N_{C \to p}(-2u - v - p) \oplus N_{C \to q}(-2v - u - q) \oplus N_{C \to u}(-2u - v - p) \oplus N_{C \to v}(-2v - u - q) \\
&\simeq \Oc_{\Pbb^1}(3)^{\oplus 4}.
\end{align*}
To see the first isomorphism above, we note that
there is a natural injection of sheaves from
$N_{C \to p} \oplus N_{C \to q} \oplus N_{C \to u} \oplus N_{C \to v}$ to
$N_C$;
since they are both vector bundles of the same Euler
characteristic, the cokernel must be zero.
The final isomorphism to $\Oc_{\Pbb^1}(3)^{\oplus 4}$ is provided by \cref{to-kp}.
This completes the proof, since $\Oc_{\Pbb^1}(3)^{\oplus 4}$
clearly satisfies interpolation.
\end{proof}


\section{Summary of Remainder of Proof of \cref{thm:intro}}
\label{S:combinat-summary}

To finish the proof of \cref{thm:intro}, it remains to show
that our collection of inductive statements and base cases from the preceeding two sections
combine to show that every tuple $(d, g, r; \mathbf{0})$
with $d \geq g + r$ and $(d, g, r) \notin \{(5, 2, 3), (6, 2, 4), (7, 2, 5)\}$ is good.
This is a purely combinatorial problem,
but one which requires a rather involved argument
(as hinted by the presence of exactly three exceptions);
hence we defer the proof to \cref{S:combinat}, providing a brief outline here:

\begin{enumerate}
\item We begin by calculating how our various inductive arguments interact with Equation~\eqref{regime}.

\item Next, we show that
our base cases and inductive arguments imply 
\eqref{regime} is a sufficient condition for the modified normal bundle of a rational curve to satisfy interpolation.

\item We then show our base cases and inductive arguments imply \eqref{regime}
is sufficient for the modified normal bundle of a space
curve to satisfy interpolation --- apart from two infinite families, which contain only finitely many members with $n = \mathbf{0}$.
Except $(d, g, r) = (5, 2, 3)$, these $n = \mathbf{0}$ cases are also good by ad-hoc application of our base cases and inductive arguments.

\item Then we show there are finitely many $(d, g, r; n)$ which are not good for $4 \leq r \leq 11$,
and use a computer program (c.f.\ \cref{A:code}) to search over all possible applications of our base cases and inductive arguments,
thereby greatly reducing the size of the finite list (and showing in particular
that all $(d, g, r; n)$ with $9 \leq r \leq 11$ are good).

\item Finally, we apply our base cases and inductive arguments to show $(d, g, r; n)$ is good for $r \geq 12$,
unless certain inequalities and congruence conditions modulo 5 (which force $n \neq \mathbf{0}$) are satisfied.
\end{enumerate}

\section{The three exceptional cases}
\label{S:exceptional}

In this section, we show conversely that if $C$ is a general
curve of degree $d$ and genus $g$ in $\Pbb^r$, where
\begin{equation} \label{exc-dgr}
(d, g, r) \in \{(5, 2, 3), (6, 2, 4), (7, 2, 5)\},
\end{equation}
then $N_C$ does not satisfy interpolation.
In these cases, we also determine when $C$ passes through
general points $p_1, p_2, \ldots, p_n$.

\begin{lemma} Let $C \subset \Pbb^r$ be a hyperelliptic curve
of degree $d$ and genus $g$;
write $S$ for the surface obtained by taking the union of all
lines joining pairs of points on $C$ conjugate under the hyperelliptic
involution. Then $S$ is a surface of degree
\[\deg S = d - g - 1.\]
\end{lemma}
\begin{proof}
Let $\Lambda \subset \Pbb^r$ be a general subspace of codimension~$2$.
Write $\pi \colon C \to \Pbb^1$ for the map induced by projection from $\Lambda$;
by construction, this is a map of degree $d$. Similarly, write $\theta \colon C \to \Pbb^1$
for the hyperelliptic map.

Then the points of intersection of $\Lambda$ with $S$ correspond
to pairs of distinct points $(x, y) \in C \times C$, with $\theta(x) = \theta(y)$
and $\pi(x) = \pi(y)$. Equivalently, they correspond to the nodes
of the image of $C$ under the map $(\theta, \pi) \colon C \to \Pbb^1 \times \Pbb^1$.
This image curve is, by construction, of bidegree $(d, 2)$;
in particular, its arithmetic genus is $(d - 1)(2 - 1) = d - 1$.
The number of nodes is therefore $(d - 1) - g = d - g - 1 = \deg S$, as desired.
\end{proof}

\begin{corollary}
If $C \subset \Pbb^r$ is of genus $2$ and degree $r + 2$,
then the above surface $S$ has degree
\[\deg S = r - 1.\]
\end{corollary}
\begin{proof}
Note that every curve of genus $2$ is hyperelliptic, so we may apply
apply the previous lemma.
\end{proof}

\begin{lemma} We have
\[\chi(N_C) = r^2 + 2r + 5.\]
\end{lemma}
\begin{proof}
We simply calculate
\begin{align*}
\chi(N_C) &= (r + 1)d - (r - 3)(g - 1) \\
&= (r + 1)(r + 2) - (r - 3)(2 - 1) \\
&= r^2 + 2r + 5. \qedhere
\end{align*}
\end{proof}

To show that interpolation does not hold in the cases of \cref{exc-dgr},
we study the short exact sequence
\[0\rightarrow N_{C/S}\rightarrow N_C\rightarrow N_S|_C\rightarrow 0.\]
By \cref{T:interpolation-ses} part~\ref{T:interpolation-ses-a}, for $N_C$ to satisfy interpolation, it is necessary for
\[\chi(N_{C/S}) \leq \frac{\chi(N_C) + r - 2}{r - 1} = \frac{r^2 + 3r + 3}{r - 1}.\]
For $r \in \{3, 4, 5\}$, the right-hand side is strictly less than $11$.
It is therefore sufficient to observe:

\begin{proposition} \label{p11}
We have
\[\chi(N_{C/S}) = 11.\]
So in particular, the bundles $N_C$ do not satisfy interpolation
for
$(d, g, r) \in \{(5, 2, 3), (6, 2, 4), (7, 2, 5)\}$.
\end{proposition}
\begin{proof}
Our first task is to understand the intersection theory on $S$.
We have two natural divisor classes $F$ and $H$ on $S$, where $H$ is the class of a hyperplane section and $F$ is the class of a line connecting two points on $C$ which are sent to each other under the hyperelliptic involution.
Then $S$ is a $\Pbb^1$-bundle over $\Pbb^1$, with $F$ the class of a fiber. As $H\cdot F=1$, this shows that $H$ and $F$ generate the Picard group of $S$. We know that
\[F\cdot F=0, \quad F\cdot H=1, \quad \text{and} \quad H\cdot H=r-1\]
(as $S$ is a surface of degree $r-1$).

Now assume $C$ has the class $a\cdot F+b\cdot H.$ We know that $C\cdot F=2$ and $C\cdot H=r+2$. This gives us the equations $b=2$ and $a+(r-1)b=r+2$, so $(a,b)=(4-r,2).$ Now this implies that $N_{C/S}=\mathcal{O}(C)|_C$ has degree given by $C\cdot C=4(4-r)+4(r-1)=12,$ so by Riemann-Roch $\chi(N_{C/S})=11.$
\end{proof}

\begin{proposition} \label{to-surface}
Assume $r \in \{3, 4, 5\}$.
There exists a non-degenerate curve of genus $2$ and degree $r+2$ through $n$ general points if and only if
there exists a ruled non-degenerate surface of degree $r-1$ through $n$ general points.
\end{proposition}
\begin{proof}
We first show, using a dimension-counting argument, that every non-degenerate ruled surface of degree $r-1$ in $\Pbb^r$ contains a curve of genus $2$ and degree $r+2$.

By a result of \cite{minimaldegree}, any ruled non-degenerate surface of degree $r-1$ in $\Pbb^r$ must be a rational normal scroll. By Lemma~2.6 of~\cite{coskun}, the space of such surfaces has dimension
\[(r+1)r+r-6=r^2+2r-6.\]
Now note that we have a rational map from the space of non-degenerate curves of genus $2$ and degree $r+2$ to this space of surfaces, given by our earlier construction of a ruled surface associated to a hyperelliptic curve embedded into projective space. 

We previously calculated that $\chi(N_{C/S})=11$. Furthermore, as the degree of $N_{C/S}$ was $12$, we must have $H^1(N_{C/S}) = 0$. As the space of possible $C$ is irreducible, this implies that the dimension of a generic fiber of this rational map is $11$. But we calculated the dimension of the space of possible $C$ to be
\[r^2+2r+5=r^2+2r-6+11,\]
so this map must be dominant.

Now assume that we have a general $S$ through $n$ general points. We have just shown that the space of possible $C$ for a general $S$ is $11$-dimensional, and as $C$ is a divisor in $S$, this space must be the projectivized space of global sections of an appropriate line bundle. Thus, if $n\leq 11$, then there is a curve passing through our $n$ points on the surface.
So take $n > 11$. Then
\[r^2+2r-6\geq n(r-2)\geq 11(r-2) \quad \Rightarrow \quad r^2-9r+16\geq 0,\]
which is false for $r=3,4,5$. So in this case there is neither an $S$ nor a $C$ through $n$ general points.
\end{proof}

\begin{proof}[Proof of \cref{through-n-general}]
Except in the cases $(d, g, r) \in \{(5, 2, 3), (6, 2, 4), (7, 2, 5)\}$, this is immediate from
\cref{thm:intro}. It thus remains to consider the cases
$(d, g, r) \in \{(5, 2, 3), (6, 2, 4), (7, 2, 5)\}$, in which case we want to prove
a nonspecial curve $C$ of degree $d$ and genus $g$
in $\Pbb^r$ passes through $n$ general points if and only if $n \leq 9$.
In these cases, we appeal to \cref{to-surface}, which reduces our problem to
showing that a ruled surface $S$ of degree $r - 1$ passes through $n$ general points
if and only if $n \leq 9$.

For $r=3$, such a surface is a quadric, and it is easy to see that there is a quadric through $n$ general points if and only if $n \leq 9$.
For $r > 3$, it is known by \cite{coskun} (last paragraph of Section~5) that there are $(r-2)(r-3) \neq 0$ scrolls through $r+4$ general points that meet a general $r-4$ plane, so for $r \in \{4,5\}$ we also have that there is a scroll through $n$ points if and only if
$n \leq 9$.
\end{proof}


\appendix

\addtocontents{toc}{\protect\renewcommand\protect\cftsecpresnum{Appendix\ }}
\addtocontents{toc}{\protect\renewcommand\protect\cftsecaftersnum{:}}

\addtocontents{toc}{\protect\renewcommand\protect\cftsecaftersnumb{\protect\phantom{\protect\cftsecpresnum}}}

\titlelabel{Appendix \thetitle:\quad}
\section{Remainder of Proof of \cref{thm:intro}} 
 \label{S:combinat}

In this appendix, we will show by a purely combinatorial argument
that any case of \cref{thm:intro} can be reduced,
using our inductive constructions in \cref{S:fixed}, to one of the base cases considered in \cref{S:base}.

\titlelabel{\thetitle \quad}
\subsection{Compatibility with \eqref{regime}}
\label{S:regime}

To avoid duplicating work, we begin in this subsection by 
determining when each of our inductive constructions preserves the condition of \cref{regime} --- i.e.\ what
additional condition (in terms of $d$, $g$, $r$, and $n$), in addition to
\cref{regime} for $(d, g, r; n)$, implies \cref{regime} for the various
$(d', g', r'; n')$ appearing in the results of \cref{S:fixed}.
We will restrict ourselves only to those lemmas (and cases thereof) which will be used
most commonly; the others will be addressed later, as we invoke them.

While in general we will merely substitute $(d', g', r'; n')$ into \cref{regime}
(a task we leave to the reader), we can obtain a better result in the case
of Lemma~\ref{stick}:

\begin{lemma} \label{stick-always}
In Lemma~\ref{stick}, the tuple $(d - 1, g, r - 1, n')$ always satisfies
\cref{regime}.
\end{lemma}
\begin{proof}
If $\sum_{i,j,k} k n_{ij}^k < r - 2$, then we want to verify
\[\sum_{i,j,k} ((r - 3)j - k) \cdot n_{ij}^k \leq 2(d - 1) + 2g - (r - 1) - 2 = 2d + 2g - r - 3.\]
Similarly, for $\sum_{i,j,k} k n_{ij}^k = r - 2$, we want to verify
\[\sum_{i,j,k} ((r - 3)j) \cdot n_{ij}^k \leq 2(d - 1) + 2g - (r - 1) - 2 = 2d + 2g - r - 3.\]
Since
$\sum_{i,j,k} ((r - 3)j - k) \cdot n_{ij}^k \leq \sum_{i,j,k} ((r - 3)j) \cdot n_{ij}^k$,
it suffices to verify in all cases that
\[\sum_{i,j,k} ((r - 3)j) \cdot n_{ij}^k \leq 2d + 2g - r - 3.\]
But we have
\begin{align*}
\sum_{i,j,k} ((r - 3)j) \cdot n_{ij}^k &= \sum_{i,j,k} ((r-2)i+(r-3)j-k) \cdot n_{ij}^k - (r - 2) \cdot \sum_{i,j,k} i n_{ij}^k + \sum_{i,j,k} k n_{ij}^k \\
&\leq 2d + 2g - r - 2 - (r - 2) \cdot \sum_{i,j,k} i n_{ij}^k + r - 2 \\
&= 2d + 2g - 4 - (r - 2) \cdot \sum_{i,j,k} i n_{ij}^k.
\end{align*}

We are thus done if $\sum_{i,j,k} i n_{ij}^k \geq 2$.
Since $n_{ij}^k = 0$ unless $i \geq j$, we can thus assume
$\sum_{i,j,k} j n_{ij}^k \leq 1$. In this case,
\[\sum_{i,j,k} ((r - 3)j) \cdot n_{ij}^k \leq r - 3 \leq r - 3 + 2(d + g - r) = 2d + 2g - r - 3. \qedhere\]
\end{proof}

For the other commonly used lemmas in \cref{S:fixed}, we simply substitute
$(d', g', r', n')$ into \cref{regime}, and rearrange. Collecting these results
(together with the result of \cref{stick-always}), we obtain the following
table of extra conditions necessary for the $(d', g', r', n')$ to satisfy
\cref{regime}.

\medskip
\begin{center}
\renewcommand{\arraystretch}{1.6}
\begin{tabular}{c@{\hspace{20pt}}c@{\hspace{50pt}}r@{\;}c@{\;}l}
\toprule
\textbf{Result} & {\boldmath $\sum_{i,j,k}kn_{ij}^k$} & \multicolumn{3}{c}{\textbf{Condition}} \\
\midrule
\cref{two-secant} & $\leq r-3$ & $\sum_{i,j,k}((r-2)i+(r-3)j-k)n_{ij}^k$ & $\leq$ & $2d+2g-3r$\\
\cref{two-secant-backwards} & $\geq r - 3$ & $\sum_{i,j,k}((r-2)i+(r-3)j-k)n_{ij}^k$ & $\leq$ & $ 2d+2g-3r$\\
\cref{stick} & Arbitrary & \multicolumn{3}{c}{No Condition} \\
\cref{two-sticks} & $\leq r-4$ & $\sum_{i,j,k}((r-3)j-k)n_{ij}^k $ & $\leq$ & $ 2d+2g-3r$\\
\cref{two-sticks} & $\geq r-3$ & $\sum_{i,j,k}((r-3)j - k)n_{ij}^k $ & $\leq$ & $ 2d+2g- 4r + 2$\\
\cref{lower-d} & $\leq r-3$ & $\sum_{i,j,k}((r-2)i+(r-3)j-k)n_{ij}^k$ & $\leq$ & $ 2d+2g-r-4$\\
\cref{lower-d} & $= r-2$ & $\sum_{i,j,k}((r-2)i+(r-3)j-k)n_{ij}^k$ & $\leq$ & $ 2d+2g-r-4$ \ and \\[-1ex]
 & & $\sum_{i,j,k}((r-3)i+(r-4)j)n_{ij}^k$ & $\leq$ & $ 2d+2g-r-3$\\
\cref{zzk} & Arbitrary & $\sum_{i,j,k}((r-2)i+(r-3)j)n_{ij}^k $ & $\leq$ & $ 2d+2g-r-2$\\
\bottomrule
\end{tabular}
\renewcommand{\arraystretch}{1.0}
\end{center}
\medskip

\noindent
We now further consider the special case where
\[\sum_{i,j,k} ((r - 2) i + (r - 3)j - k) \cdot n_{ij}^k \leq 2d + 2g - 3r + 2.\]
In this case, we consider Lemmas~\ref{two-sticks}, \ref{lower-d}, and~\ref{zzk}.
For these lemmas, we use this relation
to simplify our previous inequalities; these alternate inequalities are collected below
in the following table:

\begin{center}
\renewcommand{\arraystretch}{1.3}
\begin{tabular}{c@{\hspace{20pt}}c@{\hspace{50pt}}c}
\toprule
\textbf{Result} & {\boldmath $\sum_{i,j,k}kn_{ij}^k$} & \textbf{Condition} \\
\midrule
\cref{two-sticks} & $\leq r-4$ & $\sum_{i,j,k} i n_{ij}^k \geq 1$ \\
\cref{two-sticks} & $\geq r-3$ & $\sum_{i,j,k} i n_{ij}^k \geq 2$\\
\cref{lower-d} & Arbitrary & No Condition \\
\cref{zzk} & Arbitrary & No Condition \\
\bottomrule
\end{tabular}
\renewcommand{\arraystretch}{1.0}
\end{center}


\subsection{Interpolation for rational curves}
\label{S:rational-curves}
In this subsection, we prove that for rational curves, \eqref{regime} is in fact a sufficient condition for $N_C'$ to satisfy interpolation.
In the case of no marked points ($N_C' = N_C$), this result was obtained independently by both Sacchiero \cite{sacchiero}
and Ran \cite{ran}; however, our proof here will be independent.
We will do this by induction on the degree of $C$.

\begin{lemma} \label{rat:bigdegree} Assume all tuples $(d',0,r';n')$ satisfying \eqref{regime} and $d'<d$ are good. If
\[\sum_{i,j,k} ((r - 2)i + (r - 3)j - k) \cdot n_{ij}^k \leq 2d - 3r + 1,\]
then $(d,0,r;n)$ is good.
\end{lemma}

\begin{proof}
First we note that the given inequality implies
\[-(r - 2) \leq \sum_{i,j,k} ((r - 2)i + (r - 3)j - k) \cdot n_{ij}^k \leq 2d - 3r + 1 \implies d \geq \frac{2r + 1}{2} > r = g + r.\]
Next, for any $(i, j; k)$ with $i \geq 1$ and $j \geq 0$, we have
\[r - 1 - i - 2j - k \leq (r - 2)i + (r - 3)j - k;\]
which in turn implies
\[\sum_{\substack{i, j, k \\ (i, j) \neq (0, 0)}} (r - 1 - i - 2j - k) n_{ij}^k \leq \sum_{\substack{i, j, k \\ (i, j) \neq (0, 0)}} ((r - 2)i + (r - 3)j - k) \cdot n_{ij}^k.\]
Additionally, since $\sum_{i,j,k} k n_{ij}^k \leq r - 2$, we have
\[\sum_k (r - 1 - k) n_{00}^k \leq \sum_k ((r - 1)k - k) n_{00}^k \leq (r - 1)(r - 2) - \sum_k k n_{00}^k.\]
Adding these together, we obtain
\begin{align*}
\sum_{i,j,k}(r - 1 - i - 2j - k) n_{ij}^k &\leq (r - 1)(r - 2) + \sum_{i,j,k} ((r - 2)i + (r - 3)j - k) \cdot n_{ij}^k \\
&\leq (r - 1)(r - 2) + 2d - 3r + 1 \\
&= (r - 1)(r + 1) + 2d - 2 - (6r - 6) \\
&\leq (r - 1) d + 2d - 2 \\
&= (r + 1) d - 2.
\end{align*}
Interpolation thus follows from Lemma~\ref{lower-d} (c.f.\ \cref{S:regime}).
\end{proof}

\begin{theorem} \label{thm:ratcurves} All tuples $(d,0,r;n)$ satisfying \eqref{regime} are good.
\end{theorem}
\begin{proof}
Assume otherwise. Take a counterexample with minimal $d$. If
\[\sum_{i,j,k} ((r - 2)i + (r - 3)j - k) \cdot n_{ij}^k \leq 2d - 3r + 1\]
then \ref{rat:bigdegree} gives a contradiction. But if
\[\sum_{i,j,k} ((r - 2)i + (r - 3)j - k) \cdot n_{ij}^k \geq 2d - 3r + 2\]
then Lemma~\ref{stick} gives a contradiction.
\end{proof}


\subsection{Space curves}
\label{S:space}

\noindent
In this subsection, we prove that following result.

\begin{theorem}\label{thm:p3}
The tuple $(d, g, 3; n)$ is good provided it satisfies \eqref{regime}, and does not
lie in one of two infinite families:
\begin{itemize}
\item $\sum_{i,j,k} j n_{ij}^k = \sum k n_{ij}^k = 0$ with $g \neq 0$, and
\[\sum i n_{ij} = 2d + 2g - 14.\]
\item $\sum_{i,j,k} k n_{ij}^k = 1$ with $g \neq 0$, and
\[\sum i n_{ij} = 2d + 2g - 9.\]
\end{itemize}
\end{theorem}

As Theorem~\ref{thm:ratcurves} takes care of the $g=0$ case, we will assume that $g\neq 0$ for the rest of this section. Also, we note that \eqref{regime} can be rewritten for $r=3$ as

\begin{equation} \label{regimethree}
\sum_{i,j,k} (i - k) \cdot n_{ij}^k \leq 2d + 2g - 5.
\end{equation}

We prove this theorem by induction on $d$; and for fixed values of $d$ by induction on the number of marked points.

\begin{lemma} \label{p3:0} Suppose that Theorem~\ref{thm:p3} holds for $(d', g', 3; n')$ for all $d' < d$. Then
Theorem~\ref{thm:p3} holds for $(d, g, 3; n)$ provided that
\[g > 0, \quad \sum_{i,j,k} k n_{ij}^k = 1, \quad \text{and} \quad \sum_{i,j,k} i n_{ij}^k \leq 2d + 2g - 9.\]
\end{lemma}
\begin{proof} 
By the assumption of Theorem~\ref{thm:p3}, we have
$\sum_{i,j,k} i n_{ij}^k \neq 2d + 2g - 9$;
in particular, our assumption in fact implies
\[\sum_{i,j,k} i n_{ij}^k \leq 2d + 2g - 10.\]
Applying Lemma~\ref{two-secant}, it suffices to show $(d - 1, g - 1, 3; n')$ is good,
where
\[(n')_{ij}^k = \begin{cases}
\sum_\ell n_{ij}^\ell & \text{if $k = 0$ and $(i, j; k) \notin \{(0, 0; 0), (1, 1; 0)\}$;} \\
1 + \sum_\ell n_{ij}^\ell & \text{if $(i, j; k) = (1, 1; 0)$;} \\
0 & \text{else.}
\end{cases}\]

By our inductive hypothesis, it is sufficient to see that $(d - 1, g - 1, r; n')$
satisfies \eqref{regime} and does not lie in either of the above infinite families.
For \eqref{regime}, we want
\[1 + \sum i n_{ij}^k \leq 2(d - 1) + 2(g - 1) - 5 = 2d + 2g - 9,\]
which is precisely what we observed above. To see it does not lie in either of our infinite families,
we note that $(d - 1, g - 1, r; n')$ satisfies $\sum_{i,j,k} k (n')_{ij}^k = 0$, but $\sum_{i,j,k} j (n')_{ij}^k \neq 0$.
\end{proof}

\begin{lemma} \label{p3:1} Suppose that Theorem~\ref{thm:p3} holds for $(d', g', 3; n')$ for all $d' < d$;
and also when $d = d'$ and $n'$ has fewer marked points than $n$. Then
Theorem~\ref{thm:p3} holds for $(d, g, 3; n)$ provided that
\[g > 0, \quad \sum_{i,j,k} k n_{ij}^k = 0, \quad \text{and} \quad \sum_{i,j,k} i n_{ij}^k \leq 2d + 2g - 9.\]
\end{lemma}

\begin{proof}
Again, Lemma~\ref{two-secant} implies that it suffices to show $(d - 1, g - 1, 3; n')$ is good,
where
\[(n')_{ij}^k = \begin{cases}
n_{ij}^k & \text{if $(i, j; k) \neq (1, 1; 1)$;} \\
n_{11}^1 + 1 & \text{if $(i, j; k) = (1, 1; 1)$.}
\end{cases}\]
If $\sum_{i,j,k} i n_{ij}^k \neq 2d + 2g - 14$, then we have
\[\sum_{i,j,k} i (n')_{ij}^k = 1 + \sum_{i,j,k} i n_{ij}^k \neq 1 + 2d + 2g - 14 = 2(d - 1) + 2(g - 1) - 9;\]
and so because $\sum_{i,j,k} k (n')_{ij}^k = 1$, our inductive hypothesis
implies $(d - 1, g - 1, 3; n')$ is good, subject to the inequality
\[\sum_{i,j,k} in_{ij}^k \leq 2(d - 1) + 2(g - 1) - 5 = 2d + 2g - 9.\]

It thus remains to consider the case when
$\sum_{i,j,k} i n_{ij}^k = 2d + 2g - 14$.
But in this case, our assumptions in Theorem~\ref{thm:p3} imply $\sum_{i,j,k} j n_{ij}^k > 0$.
Moreover, by assumption, $\sum_{i,j,k} kn_{ij}^k = 0$. Consequently,
we must have some point of type $(\ell, m; 0)$ with $m \neq 0$.
Applying Lemma~\ref{ijz}, we conclude that $(d, g, 3; n)$ is good provided that
$(d, g, 3; n'')$ is good, where
\[(n'')_{ij}^k = \begin{cases}
n_{ij}^k - 1 & \text{if $(i, j; k) = (\ell, m; 0)$;} \\
n_{ij}^k & \text{else.}
\end{cases}\]
Since $\sum_{i,j,k} k (n'')_{ij}^k = \sum_{i,j,k} k n_{ij}^k = 0$, it
is sufficient, by our inductive assumption, to note that
\[\sum_{i,j,k} i (n'')_{ij}^k = \sum_{i,j,k} i n_{ij}^k - \ell = 2d + 2g - 14 - \ell < 2d + 2g - 14.\]
(Above, we used that $\ell \geq m > 0$, so $\ell \neq 0$.)
\end{proof}

\begin{lemma} \label{p3:small}
Suppose that Theorem~\ref{thm:p3} holds for $(d', g', 3; n')$ for all $d' < d$;
and also when $d = d'$ and $n'$ has fewer marked points than $n$. Then
Theorem~\ref{thm:p3} holds for $(d, g, 3; n)$ provided that
\[\sum_{i,j,k} (i - k) \cdot n_{ij}^k \leq 2d + 2g - 10.\]
\end{lemma}
\begin{proof}
Since $\sum_{i,j,k} k n_{ij}^k \leq 1$, our given inequality implies
\[\sum_{i,j,k} i n_{ij}^k \leq 2d + 2g - 9.\]
If $g = 0$, then the result follows from Theorem~\ref{thm:ratcurves}.
Otherwise, if $g > 0$, then
the result follows from Lemma~\ref{p3:0} or~\ref{p3:1}, according to whether
$\sum_{i,j,k} k n_{ij}^k = 0$ or $\sum_{i,j,k} k n_{ij}^k = 1$.
\end{proof}

\noindent
Theorem~\ref{thm:p3} then follows from combining Lemmas~\ref{p3:small}, \ref{p3:mid}, and~\ref{stick}.

\begin{corollary}
\cref{thm:intro} holds for $r = 3$.
\end{corollary}
\begin{proof}
If $n_{ij}^k = 0$ for all $(i, j; k)$, then Theorem~\ref{thm:p3} implies $N_C$ satisfies interpolation,
unless $g \neq 0$ and
\[0 = 2d + 2g - 14 \quad \Rightarrow \quad d + g = 7.\]
Since $d \geq g + 3$, this means either $(d, g) = (5, 2)$ or $(d, g) = (6, 1)$.
The case of $(d, g) = (5, 2)$ is excluded by the assumption of \cref{thm:intro};
it thus suffices to show $N_C$ satisfies interpolation for $(d, g, r) = (6, 1, 3)$.

In this case, we apply Lemma~\ref{lower-d}, which implies the desired result
so long as $(5, 1, 3; \mathbf{0})$ and $(5, 1, 2; \mathbf{0})$ are both good.
But these follow from Theorem~\ref{thm:p3} and Lemma~\ref{p2} respectively.
\end{proof}


\subsection{Curves in low dimensional projective spaces}
\label{S:low-dimensions}

In this subsection, we study curves in $\Pbb^r$, where $4 \leq r \leq 11$.
Combined with the results of the previous subsection for curves in $\Pbb^3$,
this establishes \cref{thm:intro} for $r \leq 11$. Note that this
range includes all the counterexamples to interpolation listed in \cref{thm:intro} --- as
well as the counterexample-free dimension $r = 11$, which will serve (along
with \cref{thm:ratcurves}) as the base case of
our inductive argument for higher-dimensional projective spaces.

\begin{definition} We say that $(d, g, r)$ is \emph{excellent}
if $(d, g, r; n)$ is good for every $n$ satisfying \cref{regime}.
\end{definition}

\noindent
In these terms, our basic goal is to demonstrate the following.

\begin{theorem} \label{can-induct-r-small}
Let $r \geq 4$, and suppose that $d + g \geq 2r - 1$ and
\[(d - 1, g - 1, r), \quad (d - 1, g, r - 1), \quad \text{and} \quad (d - 2, g - 1, r - 1)\]
are all excellent. Then $(d, g, r)$ is excellent.
\end{theorem}
\begin{proof}
If $g = 0$, then the result follows from \cref{thm:ratcurves};
we thus suppose $g > 0$.
If
\[\sum_{i,j,k} ((r - 2)i + (r - 3)j - k) \cdot n_{ij}^k \leq 2d + 2g - 3r,\]
then the desired result follows from Lemma~\ref{two-secant} if $\sum_{i,j,k} k n_{ij}^k < r - 2$,
and from Lemma~\ref{two-secant-backwards} if $\sum_{i,j,k} k n_{ij}^k = r - 2$.
On the other hand, if
\[\sum_{i,j,k} ((r - 2)i + (r - 3)j - k) \cdot n_{ij}^k \geq 2d + 2g - 3r + 2,\]
then the desired result follows from Lemma~\ref{stick}. It thus remains to consider the
case where
\[\sum_{i,j,k} ((r - 2)i + (r - 3)j - k) \cdot n_{ij}^k = 2d + 2g - 3r + 1.\]

If $\sum_{i,j,k} i n_{ij}^k \geq 2$, then the desired result follows Lemma~\ref{two-sticks}.
We are left with the case $\sum_{i,j,k} i n_{ij}^k \leq 1$. If $\sum_{i,j,k} i n_{ij}^k = 1$,
we have $\sum_{i,j,k} j n_{ij}^k \leq 1$; and we may assume
$\sum_{i,j,k} k n_{ij}^k \geq r - 3$ (since otherwise we may again
apply Lemma~\ref{two-sticks}). Consequently,
\[2d + 2g - 3r + 1 = \sum_{i,j,k} ((r - 2)i + (r - 3)j - k) \cdot n_{ij}^k \leq (r - 2) + (r - 3) - (r - 3) = r - 2.\]
Similarly, if $\sum_{i,j,k} i n_{ij}^k = 0$, then we have $\sum_{i,j,k} j n_{ij}^k = 0$ as well,
which gives
\[2d + 2g - 3r + 1 = \sum_{i,j,k} ((r - 2)i + (r - 3)j - k) \cdot n_{ij}^k \leq 0 + 0 - 0 = 0.\]
Either way,
\[2d + 2g - 3r + 1 \leq r - 2.\]
But this contradicts our assumption that $d + g \geq 2r - 1$.
\end{proof}

\begin{proposition} \label{prop:p4}
All tuples $(d, g, 4)$ with $d \geq g + 4$ with $d + g \geq 11$ are excellent.
In addition, \cref{thm:intro} holds for $r = 4$.
\end{proposition}
\begin{proof}
We argue by induction on $d + g$.
It is a finite computation to verify the proposition in the range $d + g \leq 16$
(see \cref{A:code}). For the inductive step, we thus suppose $d + g \geq 17$.

In particular, unless $g = 0$ (in which case the result follows
from \cref{thm:ratcurves}), $(d - 1, g - 1, 4)$ is
excellent by our inductive hypothesis.
As in Theorem~\ref{can-induct-r-small}, this implies the desired result when
\[\sum_{i,j,k} ((r - 2)i + (r - 3)j - k) \cdot n_{ij}^k \leq 2d + 2g - 3r.\]

\noindent
We next consider the case when
\[\sum_{i,j,k} ((r - 2)i + (r - 3)j - k) \cdot n_{ij}^k \geq 2d + 2g - 3r + 2.\]
In this case, \cref{stick} implies the desired result provided that $(d - 1, g, 3; n')$
is good, where if
\[\sum_{i,j,k} k n_{ij}^k < r - 2,\]
then
\[(n')_{ij}^k = \begin{cases}
\sum_\ell n_{\ell i}^k & \text{if $j = 0$ and $(i, j; k) \neq (0, 0; 0)$,} \\
0 & \text{else;}
\end{cases}\]
and if
\[\sum_{i,j,k} k n_{ij}^k = r - 2,\]
then
\[(n')_{ij}^k = \begin{cases}
\sum_{\ell, m} n_{\ell i}^m & \text{if $j = k = 0$ but $i \neq 0$,} \\
0 & \text{else.}
\end{cases}\]

In either case, we know (c.f.\ \cref{S:regime}) that \cref{regime}
is satisfied; it thus remains to check that neither case falls into
the exceptional families of Theorem~\ref{thm:p3}.
But in either case, we have
\[\sum_{i,j,k} i (n')_{ij}^k = \sum_{i,j,k} j n_{ij}^k;\]
so it remains to show
\begin{equation} \label{lhs-big}
\sum_{i,j,k} j n_{ij}^k \leq 2(d - 1) + 2g - 15 = 2d + 2g - 17.
\end{equation}
We will return to this after first
considering the case where
\[\sum_{i,j,k} ((r - 2)i + (r - 3)j - k) \cdot n_{ij}^k = 2d + 2g - 3r + 1.\]

As in Theorem~\ref{can-induct-r-small}, our assumption that $d + g \geq 2r - 1 = 7$
implies that either $\sum_{i,j,k} i n_{ij}^k \geq 2$, or
$\sum_{i,j,k} i n_{ij}^k = 1$ and $\sum_{i,j,k} k n_{ij}^k \leq r - 4$; either way,
Lemma~\ref{two-sticks} implies the desired result provided that
$(d - 2, g - 1, r - 1; n')$ is good, where if
\[\sum_{i,j,k} k n_{ij}^k < r - 3,\]
then
\[(n')_{ij}^k = \begin{cases}
\sum_\ell n_{\ell i}^k & \text{if $j = 0$ and $(i, j; k) \notin \{(0, 0; 0), (2, 0, 1)\}$,} \\
1 + \sum_\ell n_{\ell i}^k & \text{if $(i, j; k) = (2, 0, 1)$,} \\
0 & \text{else;} \\
\end{cases}\]
and if
\[\sum_{i,j,k} k n_{ij}^k = r - 3,\]
then
\[(n')_{ij}^k = \begin{cases}
\sum_{\ell, m} n_{\ell i}^m & \text{if $j = k = 0$ and $i \notin \{0, 2\}$,} \\
1 + \sum_{\ell,m} n_{\ell i}^m & \text{if $j = k = 0$ and $i = 2$,} \\
0 & \text{else;} \\
\end{cases}\]
and finally if
\[\sum_{i,j,k} k n_{ij}^k = r - 2,\]
then
\[(n')_{ij}^k = \begin{cases}
\sum_{\ell,m} n_{\ell i}^m & \text{if $j = k = 0$ and $i \neq 0$,} \\
1 & \text{if $(i, j; k) = (2, 0, 1)$,} \\
0 & \text{else.} \\
\end{cases}\]

In either case, we know (c.f.\ \cref{S:regime}) that \cref{regime}
is satisfied; it thus remains to check that neither case falls into
the exceptional families of Theorem~\ref{thm:p3}.
But in either case, we have
\[\sum_{i,j,k} i (n')_{ij}^k = 2 + \sum_{i,j,k} j n_{ij}^k;\]
so it remains to show
\[2 + \sum_{i,j,k} j n_{ij}^k \leq 2(d - 2) + 2(g - 1) - 15 = 2d + 2g - 21,\]
or equivalently, that
\begin{equation} \label{lhs-crit}
\sum_{i,j,k} j n_{ij}^k \leq 2d + 2g - 23.
\end{equation}

Since \cref{lhs-crit} visibly implies \cref{lhs-big}, we conclude that
to verify this proposition, it suffices to prove \cref{lhs-crit}. For this, we calculate
\begin{align*}
\sum_{i,j,k} ((r - 2)i + (r - 3)j - k) \cdot n_{ij}^k &= \sum_{i,j,k} (2i + j - k) \cdot n_{ij}^k \\
&\geq \sum_{i,j,k} (3j - k) \cdot n_{ij}^k \\
&\geq 3 \cdot \sum_{i,j,k} j n_{ij}^k - (r - 2).
\end{align*}
Using \cref{regime}, this implies
\[3 \cdot \sum_{i,j,k} j n_{ij}^k - r + 2 \leq 2d + 2g - r - 2;\]
or upon rearrangement,
\[\sum_{i,j,k} j n_{ij}^k \leq \frac{2d + 2g - 4}{3}.\]
It thus suffices to note that
\[\frac{2d + 2g - 4}{3} \leq 2d + 2g - 23;\]
which follows immediately from our assumption that $d + g \geq 17$.
\end{proof}

\begin{proposition} \label{prop:p5}
All tuples $(d, g, 5)$ with $d \geq g + 5$ and $d + g \geq 14$
are excellent.
In addition, \cref{thm:intro} holds for $r = 5$.
\end{proposition}
\begin{proof}
We argue by induction on $d + g$.
It is a finite computation to verify the proposition in the range $d + g \leq 15$ (see \cref{A:code}).
For the inductive step, we thus suppose $d + g \geq 16$.

In particular, unless $g = 0$ (in which case the result follows
from \cref{thm:ratcurves}), $(d - 1, g - 1, 5)$ is
excellent by our inductive hypothesis.
Moreover, by Proposition~\ref{prop:p4}, both $(d - 1, g, 4)$ and $(d - 2, g - 1, 4)$
are excellent.
Theorem~\ref{can-induct-r-small} thus implies the desired result.
\end{proof}

\begin{proposition} \label{prop:p6}
All tuples $(d, g, 6)$ with $d \geq g + 6$ and $d + g \geq 13$ are excellent.
In addition, \cref{thm:intro} holds for $r = 6$.
\end{proposition}
\begin{proof}
Again, we argue by induction on $d + g$.
It is a finite computation to verify the proposition in the range $d + g \leq 16$
(see \cref{A:code}). For the inductive step, we thus suppose $d + g \geq 17$.

In particular, unless $g = 0$ (in which case the result follows
from \cref{thm:ratcurves}), $(d - 1, g - 1, 6)$ is
excellent by our inductive hypothesis.
Moreover, by Proposition~\ref{prop:p5}, both $(d - 1, g, 5)$ and $(d - 2, g - 1, 5)$
are excellent.
Theorem~\ref{can-induct-r-small} thus implies the desired result.
\end{proof}

\begin{proposition} \label{prop:p7}
All tuples $(d, g, 7)$ with $d \geq g + 7$ and $d + g \geq 14$
are excellent.
In addition, \cref{thm:intro} holds for $r = 7$.
\end{proposition}
\begin{proof}
Again, we argue by induction on $d + g$.
It is a finite computation to verify the proposition in the range $d + g \leq 15$
(see \cref{A:code}). For the inductive step, we thus suppose $d + g \geq 16$.

In particular, unless $g = 0$ (in which case the result follows
from \cref{thm:ratcurves}), $(d - 1, g - 1, 7)$ is
excellent by our inductive hypothesis.
Moreover, by Proposition~\ref{prop:p6}, both $(d - 1, g, 6)$ and $(d - 2, g - 1, 6)$
are excellent.
Theorem~\ref{can-induct-r-small} thus implies the desired result.
\end{proof}

\begin{proposition} \label{prop:p8}
All tuples $(d, g, 8)$ with $d \geq g + 8$
are excellent.
(In particular, \cref{thm:intro} holds for $r = 8$.)
\end{proposition}
\begin{proof}
Again, we argue by induction on $d + g$.
It is a finite computation to verify the proposition in the range $d + g \leq 16$
(see \cref{A:code}). For the inductive step, we thus suppose $d + g \geq 17$.

In particular, unless $g = 0$ (in which case the result follows
from \cref{thm:ratcurves}), $(d - 1, g - 1, 8)$ is
excellent by our inductive hypothesis.
Moreover, by Proposition~\ref{prop:p6}, both $(d - 1, g, 7)$ and $(d - 2, g - 1, 7)$
are excellent.
Theorem~\ref{can-induct-r-small} thus implies the desired result.
\end{proof}

\begin{proposition} \label{prop:plow}
All tuples $(d, g, r)$ with $d \geq g + r$ and $9 \leq r \leq 11$ are excellent.
(In particular, \cref{thm:intro} holds for $9 \leq r \leq 11$.)
\end{proposition}
\begin{proof}
Again, we argue by induction on $d + g$. By \cref{can-induct-r-small}, it is sufficient
to check the range $d + g \leq 2r - 2$. But this a finite computation (c.f.\ \cref{A:code}).
\end{proof}


\subsection{Curves in high dimensional projective spaces}
\label{S:high-dimensions}

In this subsection we study curves in $\Pbb^r$, where $r \geq 12$.
In order to state our main result, we will need the following definition:

\begin{definition}
Suppose that
\[\sum_{i,j,k} (i + j) n_{ij}^k \leq 3.\]
Then we define $\delta(n)$ according to the following table.
\begin{center}
\begin{tabular}{c@{\hspace{30pt}}c@{\hspace{30pt}}c}
  \toprule
  {\boldmath $\sum_{i,j,k} i n_{ij}^k$} & {\boldmath $\sum_{i,j,k} j n_{ij}^k$} & {\boldmath $\delta(n)$} \\
  \midrule
  0 & 0 & 2 \\
  1 & 0 & 3 \\
  1 & 1 & 5 \\
  2 & 0 & 4 \\
  2 & 1 & 5 \\
  3 & 0 & 4 \\
  \bottomrule
\end{tabular}
\end{center}
\end{definition}

\noindent
Our main result will be the following theorem, which we will prove by induction
on $r$.

\begin{theorem} \label{thm:high}
The tuple $(d, g, r; n)$ is good if $r \geq 11$ and $d \geq g + r$, unless
either
\[\sum_{i,j,k} i n_{ij}^k = \sum_{i,j,k} j n_{ij}^k = 1, \quad \sum_{i,j,k} k n_{ij}^k = r - 2, \quad \text{and} \quad d + g = 2r - 2;\]
or
\begin{align*}
  \sum_{i,j,k} (i + j) \cdot n_{ij}^k \leq 3, \qquad&
  \sum_{i,j,k} k n_{ij}^k = 4r - 2d - 2g - \delta(n) > \frac{r}{2}, \quad \textrm{and} \\
  &d + g + r \equiv \delta(n) + 2 \ \text{or} \ \delta(n) + 4 \ \mod 5.
\end{align*}
\end{theorem}

Note that \cref{prop:plow} implies the \cref{thm:high} for $r = 11$;
this will serve as the base case of our induction.
For our inductive step, we will therefore suppose $r \geq 12$.

Before proving \cref{thm:high}, we first deduce two useful corollaries.
These corollaries assert that certain tuples $(d, g, r; n)$ are good,
and only require the truth of \cref{thm:high} for tuples $(d, g, r; n')$
which satisfy $\sum_{i,j,k} (n')_{ij}^k \leq \sum_{i,j,k} n_{ij}^k$.
These corollaries can therefore be used in our inductive argument.
We begin with the following lemma.

\begin{lemma} \label{zzk-dg} The inequalities of \cref{zzk} and \cref{lower-d} are satisfied provided that
\[\sum_{i,j,k} (i + k) \cdot n_{ij}^k \leq \frac{3r^2 - 3r - 4}{2r - 4} - \frac{r - 5}{2r - 4}(d + g).\]
\end{lemma}
\begin{proof}
Subject to the given inequality,
\begin{align*}
\sum_{i,j,k} (r - 1 - i - 2j - k) \cdot n_{ij}^k &\leq (r - 2) \cdot \sum_{i,j,k} (i + k) \cdot n_{ij}^k \\
&\leq (r - 2) \cdot \left(\frac{3r^2 - 3r - 4}{2r - 4} - \frac{r - 5}{2r - 4}(d + g)\right) \\
&= \frac{3r - 3}{2} \cdot r - \frac{r - 5}{2}(d + g) - 2 \\
&\leq \frac{3r - 3}{2}(d - g) - \frac{r - 5}{2}(d + g) - 2 \\
&= (r + 1) d - (2r - 4) g - 2. \qedhere
\end{align*}
\end{proof}

The following corollary gives a slight strengthening of
\cref{thm:high}, which will be useful for induction: Once we
prove it, we may assume the stronger statement given below as
our inductive hypothesis, but need only show the weaker statement
of \cref{thm:high}.

\begin{corollary} \label{high-lengthened}
The tuple $(d, g, r; n)$ is good if $r \geq 11$ and $d \geq g + r$, unless
either
\[\sum_{i,j,k} i n_{ij}^k = \sum_{i,j,k} j n_{ij}^k = 1, \quad \sum_{i,j,k} k n_{ij}^k = r - 2, \quad \text{and} \quad d + g = 2r - 2;\]
or
\begin{align*}
  \sum_{i,j,k} (i + j) \cdot n_{ij}^k \leq 3, \qquad&
  \sum_{i,j,k} k n_{ij}^k = 4r - 2d - 2g - \delta(n) > \frac{r + 3}{2}, \quad \text{and} \\
  &d + g + r \equiv \delta(n) + 2 \ \text{or} \ \delta(n) + 4 \ \mod 5.  
\end{align*}
\end{corollary}
\begin{proof}
For $r = 11$, this follows from \cref{prop:plow}; we thus assume $r \geq 12$.
Applying \cref{thm:high}, it suffices to consider the case where
\begin{equation} \label{hl-remains}
\sum_{i,j,k} (i + j) \cdot n_{ij}^k \leq 3 \quad \text{and} \quad \frac{r + 1}{2} \leq \sum_{i,j,k} k n_{ij}^k = 4r - 2d - 2g - \delta(n) \leq \frac{r + 3}{2}.
\end{equation}
By induction, it is sufficient to show that in such
a case, we can always apply \cref{zzk}.
For this, we first need to know that $n_{00}^k \neq 0$ for some $k$. But
\[\sum_{i,j,k} k n_{ij}^k \geq \frac{r + 1}{2} \geq \frac{13}{2} > 2 \cdot 3 \geq \sum_{i,j,k} 2(i + j) n_{ij}^k \geq \sum_{\substack{i,j,k \\ (i, j) \neq (0, 0)}} k n_{ij}^k.\]

Next we need to check the inequalities of \cref{zzk}.
By \cref{zzk-dg} and the results of \cref{S:regime},
this boils down to showing the two inequalities:
\begin{align*}
\sum_{i,j,k} ((r - 2)i + (r - 3)j - k) \cdot n_{ij}^k &\leq 2d + 2g - r - 2, \\
\sum_{i,j,k} (i + k) \cdot n_{ij}^k &\leq \frac{3r^2 - 3r - 4}{2r - 4} - \frac{r - 5}{2r - 4}(d + g).
\end{align*}
By our assumption that $\sum_{i,j,k} (i + j) n_{ij}^k \leq 3$, these reduce to:
\begin{align*}
3(r - 2) - \sum_{i,j,k} j n_{ij}^k - \sum_{i,j,k} k n_{ij}^k &\leq 2d + 2g - r - 2, \\
3 + \sum_{i,j,k} k n_{ij}^k &\leq \frac{3r^2 - 3r - 4}{2r - 4} - \frac{r - 5}{2r - 4}(d + g).
\end{align*}
Solving for $d + g$ in \cref{hl-remains}, we obtain:
\[d + g = \frac{4r - \delta(n) - \sum_{i,j,k} k n_{ij}^k}{2}.\]
Substituting this into the above, it remains to show:
\begin{align*}
3(r - 2) - \sum_{i,j,k} j n_{ij}^k - \sum_{i,j,k} k n_{ij}^k &\leq 2 \cdot \frac{4r - \delta(n) - \sum_{i,j,k} k n_{ij}^k}{2} - r - 2, \\
3 + \sum_{i,j,k} k n_{ij}^k &\leq \frac{3r^2 - 3r - 4}{2r - 4} - \frac{r - 5}{2r - 4} \cdot \frac{4r - \delta(n) - \sum_{i,j,k} k n_{ij}^k}{2}.
\end{align*}
Or upon rearrangement, that
\begin{align*}
\delta(n) &\leq 4 + \sum_{i,j,k} jn_{ij}^k, \\
(3r - 3) \cdot \sum_{i,j,k} k n_{ij}^k &\leq 2r^2 + 2r + 16 + (r - 5) \cdot \delta(n).
\end{align*}

The first of these inequalities is clear. For the second, we first note that $\delta(n) \geq 2$;
it is thus sufficient to show
\[(3r - 3) \cdot \sum_{i,j,k} k n_{ij}^k \leq 2r^2 + 2r + 16 + 2(r - 5) = 2r^2 + 4r + 6,\]
or upon rearrangement, that
\[\sum_{i,j,k} k n_{ij}^k \leq \frac{2r^2 + 4r + 6}{3r - 3}.\]
Applying \cref{hl-remains}, it thus remains to show
\[\frac{r + 3}{2} \leq \frac{2r^2 + 4r + 6}{3r - 3},\]
which is clear for $r \geq 12$.
\end{proof}

For convenience, we include
the following corollary, giving several more easily-used
special cases of \cref{high-lengthened}, which will appear
in our subsequent inductive argument.

\begin{corollary} \label{cor:high}
The tuple $(d, g, r; n)$ is good if $r \geq 11$ and $d \geq g + r$, provided
that \eqref{regime} is satisfied and at least one of the following holds:
\begin{enumerate}
\item \label{dg-big} If $d + g \geq 2r - 1$;
\item \label{J-zero} If $d + g \geq (7r - 7)/4$ and we do not have both
\[\sum_{i,j,k} in_{ij}^k = \sum_{i,j,k} jn_{ij}^k = 1 \quad \text{and} \quad \sum_{i,j,k} k n_{ij}^k = r - 2;\]
\item \label{K-small} If $\sum_{i,j,k} k n_{ij}^k \leq (r + 3)/2$. In particular, this happens if $n_{ij}^k = 0$ for all $(i, j; k)$.
\end{enumerate}
Condition~\ref{K-small} in particular implies \cref{thm:intro}
holds for $r \geq 12$ --- which, combined with the results of \cref{S:low-dimensions},
completes the proof of \cref{thm:intro}.
\end{corollary}
\begin{proof}
We begin with Conditions~\ref{dg-big} and~\ref{J-zero},
making use of \cref{high-lengthened}: If
\[4r - 2d - 2g - \delta(n) > \frac{r + 3}{2},\]
then in particular we have
\[4r - 2d - 2g - 2 \geq 4r - 2d - 2g - \delta(n) \geq \frac{r + 4}{2};\]
or upon rearrangement,
\[d + g \leq \frac{7r - 8}{4}.\]
We conclude that $(d, g, r; n)$ is good unless
\[d + g \leq \max\left(\frac{7r - 8}{4}, 2r - 2\right) = 2r - 2;\]
and unless $\sum_{i,j,k} i n_{ij}^k = \sum_{i,j,k} j n_{ij}^k = 1$ and $\sum_{i,j,k} k n_{ij}^k = r - 2$,
that $(d, g, r; n)$ is good unless
\[d + g \leq \frac{7r - 8}{4}.\]

\noindent
Finally, we consider Condition~\ref{K-small}: In this case, it sufficient
to note that $r - 2 > (r + 3)/2$.
\end{proof}

\begin{proposition} \label{prop:dg-small} If \cref{thm:high}
holds for $r' = r - 1$, then to prove \cref{thm:high} in $\Pbb^r$,
it is sufficient to consider cases where
\begin{equation} \label{dg-small}
g > 0 \quad \text{and} \quad d + g \leq 2r.
\end{equation}
\end{proposition} 
\begin{proof}
The case of $g = 0$ is covered by \cref{thm:ratcurves}, so it suffices
to consider the cases where $g > 0$.
Next, for $d + g \geq 2r + 1$, we have
\begin{align*}
(d - 1) + (g - 1) &\geq 2r - 1, \\
(d - 1) + g &\geq 2(r - 1) - 1, \\
(d - 2) + (g - 1) &\geq 2(r - 1) - 1.
\end{align*}
\cref{can-induct-r-small} therefore implies the desired result
by induction on $d + g$.
\end{proof}

\noindent
For the remainder of this section, we will thus make the
assumptions given by \cref{dg-small}.

\begin{proposition} \label{K-maximal-lhs-small}
Suppose that \cref{thm:high} holds for all $(d', g', r'; n')$
where either $d' < d$, or $d' = d$ and $\sum_{i,j,k} (n')_{ij}^k < \sum_{i,j,k} n_{ij}^k$.
Then \cref{thm:high} holds for $(d, g, r; n)$ provided that
\[\sum_{i,j,k} ((r - 2)i + (r - 3)j - k) \cdot n_{ij}^k \leq 2d + 2g - 3r \quad \text{and} \quad \sum_{i,j,k} k n_{ij}^k \geq r - 3.\]
\end{proposition}
\begin{proof}
The desired result follows from \cref{two-secant-backwards}: Indeed,
the $n'$ appearing in \cref{two-secant-backwards} satisfies
\[\sum_{i,j,k} k (n')_{ij}^k \leq 2 \leq \frac{r + 3}{2}. \qedhere\]
\end{proof}

\begin{proposition} \label{K-maximal-lhs-big}
Suppose that \cref{thm:high} holds for all $(d', g', r'; n')$
where either $d' < d$, or $d' = d$ and $\sum_{i,j,k} (n')_{ij}^k < \sum_{i,j,k} n_{ij}^k$.
Then \cref{thm:high} holds for $(d, g, r; n)$ provided that
\[\sum_{i,j,k} ((r - 2)i + (r - 3)j - k) \cdot n_{ij}^k \geq 2d + 2g - 3r + 2 \quad \text{and} \quad \sum_{i,j,k} k n_{ij}^k = r - 2.\]
\end{proposition}
\begin{proof}
The desired result follows from \cref{stick}: The $n'$ appearing in \cref{stick}
satisfies
\[\sum_{i,j,k} k (n')_{ij}^k = 0 \leq \frac{r + 3}{2}. \qedhere\]
\end{proof}

\begin{lemma} \label{to-ij}
We have
\[\sum_{i,j,k} ((r - 2)i + (r - 3)j - k) \cdot n_{ij}^k \geq \frac{2r - 5}{2} \cdot \sum_{i,j,k} (i + j) n_{ij}^k - r + 2.\]
\end{lemma}
\begin{proof}
Because $i \geq j$ whenever $n_{ij}^k \neq 0$, we obtain
\begin{align*}
\sum_{i,j,k} ((r - 2)i + (r - 3)j - k) \cdot n_{ij}^k &\geq \frac{(r - 2) + (r - 3)}{2} \cdot \sum_{i,j,k} (i + j) n_{ij}^k - \sum_{i,j,k} k n_{ij}^k \\
&\geq \frac{2r - 5}{2} \cdot \sum_{i,j,k} (i + j) n_{ij}^k - (r - 2). \qedhere
\end{align*}
\end{proof}

\begin{lemma} \label{ij-bound} We have
\[\sum_{i,j,k} (i + j) n_{ij}^k \leq \frac{4(d + g) - 8}{2r - 5}.\]
\end{lemma}
\begin{proof}
By \cref{regime} together with \cref{to-ij},
\[ \frac{2r - 5}{2} \cdot \sum_{i,j,k} (i + j) n_{ij}^k - r + 2 \leq \sum_{i,j,k} ((r - 2)i + (r - 3)j - k) \cdot n_{ij}^k \leq 2(d + g) - r - 2.\]
Rearranging yields the statement of this lemma.
\end{proof}

\begin{lemma} \label{stick-applies-ij-large} We have 
\[\sum_{i,j,k} ((r - 2)i + (r - 3)j - k) \cdot n_{ij}^k \geq 2d + 2g - 3r + 2\]
provided that
\[\sum_{i,j,k} (i + j) \cdot n_{ij}^k \geq \frac{4(d + g) - 4r}{2r - 5}.\]
\end{lemma}
\begin{proof}
We have
\[\sum_{i,j,k} ((r - 2)i + (r - 3)j - k) \cdot n_{ij}^k \geq \frac{2r - 5}{2} \cdot \sum_{i,j,k} (i + j) \cdot n_{ij}^k - r + 2;\]
it is therefore sufficient to show
\[\frac{2r - 5}{2} \cdot \sum_{i,j,k} (i + j) \cdot n_{ij}^k - r + 2 \geq 2d + 2g - 3r + 2,\]
which is a rearrangement of our assumption.
\end{proof}

\begin{lemma} \label{can-apply-stick-applies-ij-large}
We have
\[\frac{4(d + g) - 4r}{2r - 5} \leq 3.\]
\end{lemma}
\begin{proof}
Upon rearrangement, our desired inequality becomes
\[d + g \leq \frac{10r - 15}{4}.\]
Using \cref{prop:dg-small}, it thus remains to check
\[2r \leq \frac{10r - 15}{4},\]
which is immediate for $r \geq 8$.
\end{proof}

\begin{proposition} \label{prop:high-4}
Suppose that \cref{thm:high} holds for all $(d', g', r'; n')$
where either $d' < d$, or $d' = d$ and $\sum_{i,j,k} (n')_{ij}^k < \sum_{i,j,k} n_{ij}^k$.
Then \cref{thm:high} holds for $(d, g, r; n)$ if $\sum_{i,j,k} (i + j) \cdot n_{ij}^k \geq 4$.
\end{proposition}
\begin{proof}
From  \cref{stick-applies-ij-large} and
\cref{can-apply-stick-applies-ij-large}, we obtain
\[\sum_{i,j,k} ((r - 2)i + (r - 3)j - k) \cdot n_{ij}^k \geq 2d + 2g - 3r + 2.\]
If $\sum_{i,j,k} k n_{ij}^k = r - 2$, then the result follows from
\cref{K-maximal-lhs-big}. Otherwise, to conclude by \cref{stick}, it suffices
to show $(d - 1, g, r - 1; n')$
satisfies our inductive hypothesis, where
\[(n')_{ij}^k = \begin{cases}
\sum_\ell n_{\ell i}^k & \text{if $j = 0$ and $(i, j; k) \neq (0, 0; 0)$;} \\
0 & \text{else.}
\end{cases}\]
Because $\sum_{i,j,k} j (n')_{ij}^k = 0$,
it is sufficient (c.f.\ \cref{cor:high}) to check
\[(d - 1) + g \geq \frac{7(r - 1) - 7}{4};\]
or upon rearrangement,
\[d + g \geq \frac{7r - 10}{4}.\]
However, by \cref{ij-bound},
\[4 \leq \sum_{i,j,k} (i + j) \cdot n_{ij}^k \leq \frac{4(d + g) - 8}{2r - 5};\]
which upon rearrangement yields
\[d + g \geq 2r - 3.\]
It is thus sufficient to note that for $r \geq 12$,
\[\frac{7r - 10}{4} \leq 2r - 3. \qedhere\]
\end{proof}

\begin{proposition} \label{prop:high-3}
Suppose that \cref{thm:high} holds for all $(d', g', r'; n')$
where either $d' < d$, or $d' = d$ and $\sum_{i,j,k} (n')_{ij}^k < \sum_{i,j,k} n_{ij}^k$.
Then \cref{thm:high} holds for $(d, g, r; n)$ if $\sum_{i,j,k} (i + j) \cdot n_{ij}^k \in \{2, 3\}$ and
\[\frac{4(d + g) - 4r}{2r - 5} \leq \sum_{i,j,k} (i + j) \cdot n_{ij}^k.\]
In particular, \cref{thm:high} holds for $\sum_{i,j,k} (i + j) \cdot n_{ij}^k = 3$.
\end{proposition}
\begin{proof}
By \cref{K-maximal-lhs-big},
it suffices to consider the case $\sum_{i,j,k} k n_{ij}^k < r - 3$.
Moreover, by \cref{stick-applies-ij-large},
the inequality required for \cref{stick} is satisfied.
Write
\[\epsilon = \sum_{i,j,k} j n_{ij}^k \in \{0, 1\}.\]
As in \cref{prop:high-4}, it suffices to show $(d - 1, g, r - 1; n')$
satisfies our inductive hypothesis, where
\[(n')_{ij}^k = \begin{cases}
\sum_\ell n_{\ell i}^k & \text{if $j = 0$ and $(i, j; k) \neq (0, 0; 0)$;} \\
0 & \text{else.}
\end{cases}\]
Because
\[\sum_{i,j,k} i (n')_{ij}^k = \epsilon \quad \text{and} \quad \sum_{i,j,k} j (n')_{ij}^k = 0,\]
we have $\delta(n') = 2 + \epsilon$. Our problem is thus to show that we cannot simultaneously have
\begin{align*}
\sum_{i,j,k} k n_{ij}^k = \sum_{i,j,k} k (n')_{ij}^k &= 4(r - 1) - 2(d - 1) - 2g - (2 + \epsilon) > \frac{(r - 1) + 3}{2}, \\
(d - 1) + g + (r - 1) &\equiv 4 + \epsilon \ \text{or} \ 1 + \epsilon \ \mod 5.
\end{align*}
Or upon rearrangement, that we cannot simultaneously have
\begin{align*}
\sum_{i,j,k} k n_{ij}^k &= 4r - 2d - 2g - (4 + \epsilon) > \frac{r + 2}{2}, \\
d + g + r &\equiv 1 + \epsilon \ \text{or} \ 3 + \epsilon \ \mod 5.
\end{align*}
But by assumption (and because $\delta(n) = 4 + \epsilon$), we cannot simultaneously have
\begin{align*}
\sum_{i,j,k} k n_{ij}^k &= 4r - 2d - 2g - (4 + \epsilon) > \frac{r}{2}, \\
d + g + r &\equiv 1 + \epsilon \ \text{or} \ 3 + \epsilon \ \mod 5. \qedhere
\end{align*}
\end{proof}

\begin{proposition} \label{prop:high-2}
Suppose that \cref{thm:high} holds for all $(d', g', r'; n')$
where either $d' < d$, or $d' = d$ and $\sum_{i,j,k} (n')_{ij}^k < \sum_{i,j,k} n_{ij}^k$.
Then \cref{thm:high} holds for $(d, g, r; n)$ if $\sum_{i,j,k} (i + j) \cdot n_{ij}^k = 2$.
\end{proposition}
\begin{proof}
By \cref{prop:high-3}, we may reduce to the case where
\[\frac{4(d + g) - 4r}{2r - 5} > \sum_{i,j,k} (i + j) \cdot n_{ij}^k = 2;\]
or upon rearrangement,
\[d + g > 2r - 2 - \frac{1}{2} \quad \Rightarrow \quad d + g \geq 2r - 2.\]
If
\[\sum_{i,j,k} ((r - 2)i + (r - 3)j - k) \cdot n_{ij}^k \leq 2d + 2g - 3r,\]
then by \cref{K-maximal-lhs-small},
it suffices to consider the case $\sum_{i,j,k} k n_{ij}^k < r - 2$.
In this case, \cref{two-secant} implies the desired result:
We are reduced to showing $(d - 1, g - 1, r; n')$
is good, where crucially we do \emph{not} have
\[\sum_{i,j,k} i(n')_{ij}^k = \sum_{i,j,k} j(n')_{ij}^k = 1.\]
In particular, our inductive hypothesis implies the desired result so long as
\[d + g - 2 = (d - 1) + (g - 1) \geq \frac{7r - 7}{4}.\]
This inequality holds since for $r \geq 12$,
\[d + g - 2 \geq 2r - 4 \geq \frac{7r - 7}{4}.\]

\noindent
On the other hand, if
\[\sum_{i,j,k} ((r - 2)i + (r - 3)j - k) \cdot n_{ij}^k \geq 2d + 2g - 3r + 2,\]
then by \cref{K-maximal-lhs-big}, we may also assume
$\sum_{i,j,k} k n_{ij}^k < r - 2$.
In this case,
we claim the desired result follows from \cref{stick}. Indeed, we are reduced to
showing $(d - 1, g, r - 1; n')$ is good, where we again crucially do \emph{not} have
\[\sum_{i,j,k} i(n')_{ij}^k = \sum_{i,j,k} j(n')_{ij}^k = 1.\]
In particular, our inductive hypothesis implies the desired result so long as
\[d + g - 1 = (d - 1) + g \geq \frac{7(r - 1) - 7}{4} = \frac{7r - 14}{4}.\]
This inequality holds since for $r \geq 12$,
\[d + g - 1 \geq 2r - 3 \geq \frac{7r - 14}{4}.\]

\noindent
It thus remains to consider the case where
\[\sum_{i,j,k} ((r - 2)i + (r - 3)j - k) \cdot n_{ij}^k = 2d + 2g - 3r + 1.\]
By assumption, we have either $\sum_{i,j,k} i n_{ij}^k = 2$ and $\sum_{i,j,k} j n_{ij}^k = 0$,
or $\sum_{i,j,k} i n_{ij}^k = \sum_{i,j,k} j n_{ij}^k = 1$.

First we consider the cases where either $\sum_{i,j,k} i n_{ij}^k = 2$ and $\sum_{i,j,k} j n_{ij}^k = 0$,
or $\sum_{i,j,k} k n_{ij}^k < r - 3$.
In either of these cases, we claim the desired result follows from \cref{two-sticks}.
Indeed, we are reduced to showing $(d - 2, g - 1, r - 1; n')$ is good, where we again crucially do \emph{not} have
\[\sum_{i,j,k} i(n')_{ij}^k = \sum_{i,j,k} j(n')_{ij}^k = 1.\]
In particular, our inductive hypothesis implies the desired result so long as
\[d + g - 3 = (d - 2) + (g - 1) \geq \frac{7(r - 1) - 7}{4} = \frac{7r - 14}{4}.\]
This inequality holds since for $r \geq 12$,
\[d + g - 3 \geq 2r - 5 \geq \frac{7r - 14}{4}.\]

Thus, it remains to consider the case $\sum_{i,j,k} i n_{ij}^k = \sum_{i,j,k} j n_{ij}^k = 1$
and $\sum_{i,j,k} k n_{ij}^k \in \{r - 2, r - 3\}$. In this case, we have
\[(r - 2) + (r - 3) - \sum_{i,j,k} k n_{ij}^k = \sum_{i,j,k} ((r - 2)i + (r - 3)j - k) \cdot n_{ij}^k = 2d + 2g - 3r + 1;\]
or upon rearrangement,
\[\sum_{i,j,k} k n_{ij}^k = 5r - 2d - 2g - 6 \equiv r \ \mod 2.\]
It follows that in fact,
\[\sum_{i,j,k} k n_{ij}^k = r - 2,\]
and that
\[r - 2 = 5r - 2d - 2g - 6 \quad \Rightarrow \quad d + g = 2r - 2.\]
But this case is excluded by assumption.
\end{proof}

\begin{proposition} \label{prop:high-1}
Suppose that \cref{thm:high} holds for all $(d', g', r'; n')$
where either $d' < d$, or $d' = d$ and $\sum_{i,j,k} (n')_{ij}^k < \sum_{i,j,k} n_{ij}^k$.
Then \cref{thm:high} holds for $(d, g, r; n)$ if $\sum_{i,j,k} (i + j) \cdot n_{ij}^k = 1$.
\end{proposition}
\begin{proof}
Consider first the case when
\[\sum_{i,j,k} ((r - 2)i + (r - 3)j - k) \cdot n_{ij}^k \neq 2d + 2g - 3r + 1 \quad \text{and} \quad \sum_{i,j,k} k n_{ij}^k \geq r - 3.\]
If in addition
\[\sum_{i,j,k} ((r - 2)i + (r - 3)j - k) \cdot n_{ij}^k < 2d + 2g - 3r + 1,\]
then \cref{K-maximal-lhs-small} implies the desired result. Similarly, if
\[\sum_{i,j,k} ((r - 2)i + (r - 3)j - k) \cdot n_{ij}^k > 2d + 2g - 3r + 1 \quad \text{and} \quad \sum_{i,j,k} k n_{ij}^k = r - 2,\]
then \cref{K-maximal-lhs-big} implies the desired result. In this case we may thus assume
\[\sum_{i,j,k} ((r - 2)i + (r - 3)j - k) \cdot n_{ij}^k > 2d + 2g - 3r + 1 \quad \text{and} \quad \sum_{i,j,k} k n_{ij}^k = r - 3.\]
Applying \cref{stick}, it suffices to show $(d - 1, g, r - 1; n')$
satisfies our inductive hypothesis, where
\[(n')_{ij}^k = \begin{cases}
\sum_\ell n_{\ell i}^k & \text{if $j = 0$ and $(i, j; k) \neq (0, 0; 0)$;} \\
0 & \text{else.}
\end{cases}\]

For this we first note that $\sum_{i,j,k} j (n')_{ij}^k = 0$;
since $\delta(n') = 2$, our problem is thus reduced to showing
that we do not simultaneously have
\begin{align*}
r - 3 = \sum_{i,j,k} k n_{ij}^k = \sum_{i,j,k} k (n')_{ij}^k &= 4(r - 1) - 2(d - 1) - 2g - 2, \\
(d - 1) + g + (r - 1) &\equiv 4 \ \text{or} \ 1 \ \mod 5.
\end{align*}
Reducing the first equation mod $5$ and rearranging, it suffices
to show that we do not simultaneously have
\begin{align*}
d + g + r &\equiv 2 \ \mod 5 \\
d + g + r &\equiv 1 \ \text{or} \ 3 \ \mod 5,
\end{align*}
which is clear.

Next, we consider the case when
\[\sum_{i,j,k} ((r - 2)i + (r - 3)j - k) \cdot n_{ij}^k \leq 2d + 2g - 3r - 1 \quad \text{and} \quad \sum_{i,j,k} k n_{ij}^k < r - 3.\]
By \cref{two-secant} it is thus sufficient
to show $(d - 1, g - 1, r; n')$ satisfies our inductive hypothesis, where
\[(n')_{ij}^k = \begin{cases}
n_{ij}^k & \text{if $(i, j; k) \neq (1, 1; 1)$;} \\
n_{11}^1 + 1 & \text{if $(i, j; k) = (1, 1; 1)$.}
\end{cases}\]

For this, we first note that $\sum_{i,j,k} k(n')_{ij}^k < r - 2$;
since $\delta(n') = 5$, our problem is thus reduced to showing
\[1 + \sum_{i,j,k} k n_{ij}^k = \sum_{i,j,k} k (n')_{ij}^k \neq 4r - 2(d - 1) - 2(g - 1) - 5;\]
or upon rearrangement, that
\[\sum_{i,j,k} k n_{ij}^k \neq 4r - 2d - 2g - 2.\]
But we have
\[(r - 2) - \sum_{i,j,k} k n_{ij}^k = \sum_{i,j,k} ((r - 2)i + (r - 3)j - k) \cdot n_{ij}^k \leq 2d + 2g - 3r - 1;\]
which upon rearrangement gives
\[\sum_{i,j,k} k n_{ij}^k \geq 4r - 2d - 2g - 1,\]
completing the proof in this case.

Next, we consider the case when
\[\sum_{i,j,k} ((r - 2)i + (r - 3)j - k) \cdot n_{ij}^k \geq 2d + 2g - 3r + 3 \quad \text{and} \quad \sum_{i,j,k} k n_{ij}^k < r - 3.\]
Applying \cref{stick}, it suffices to show $(d - 1, g, r - 1; n')$
satisfies our inductive hypothesis, where
\[(n')_{ij}^k = \begin{cases}
\sum_\ell n_{\ell i}^k & \text{if $j = 0$ and $(i, j; k) \neq (0, 0; 0)$;} \\
0 & \text{else.}
\end{cases}\]

For this we first note that $\sum_{i,j,k} j (n')_{ij}^k = 0$;
since $\delta(n') = 2$, our problem is thus reduced to showing
\[\sum_{i,j,k} k n_{ij}^k = \sum_{i,j,k} k (n')_{ij}^k \neq 4(r - 1) - 2(d - 1) - 2g - 2 = 4r - 2d - 2g - 4.\]
But we have
\[(r - 2) - \sum_{i,j,k} k n_{ij}^k = \sum_{i,j,k} ((r - 2)i + (r - 3)j - k) \cdot n_{ij}^k \geq 2d + 2g - 3r + 3;\]
which upon rearrangement gives
\[\sum_{i,j,k} k n_{ij}^k \leq 4r - 2d - 2g - 5,\]
completing the proof in this case.

Next, we consider the case when
\[2d + 2g - 3r \leq \sum_{i,j,k} ((r - 2)i + (r - 3)j - k) \cdot n_{ij}^k \leq 2d + 2g - 3r + 2 \quad \text{and} \quad \sum_{i,j,k} k n_{ij}^k < r - 3.\]
In this case, we seek to apply \cref{two-sticks}. For this, it suffices
to show $(d - 2, g - 1, r - 1; n')$ is good, where
\[(n')_{ij}^k = \begin{cases}
\sum_\ell n_{\ell i}^k & \text{if $j = 0$ and $(i, j; k) \notin \{(0, 0; 0), (2, 0, 1)\}$;} \\
1 + \sum_\ell n_{\ell i}^k & \text{if $(i, j; k) = (2, 0, 1)$;} \\
0 & \text{else.} \\
\end{cases}\]
Since
\[\sum_{i,j,k} i (n')_{ij}^k = 2 + \sum_{i,j,k} j n_{ij}^k = 2 \quad \text{and} \quad \sum_{i,j,k} j (n')_{ij}^k = 0,\]
we have $\delta(n') = 4$, and it suffices to show that we do not simultaneously have
\begin{align*}
1 + \sum_{i,j,k} k n_{ij}^k = \sum_{i,j,k} k (n')_{ij}^k &= 4(r - 1) - 2(d - 2) - 2(g - 1) - 4 > \frac{(r - 1) + 3}{2}, \\
(d - 2) + (g - 1) + (r - 1) &\equiv 1 \ \text{or} \ 3 \ \mod 5.
\end{align*}
Or, upon rearrangement, that we do not simultaneously have
\begin{align*}
\sum_{i,j,k} k n_{ij}^k &= 4r - 2d - 2g - 3 > \frac{r}{2}, \\
d + g + r &\equiv 0 \ \text{or} \ 2 \ \mod 5.
\end{align*}
But this is precisely the assumption of \cref{thm:high} (we have $\delta(n) = 3$).

It thus remains to consider the case when
\[\sum_{i,j,k} ((r - 2)i + (r - 3)j - k) \cdot n_{ij}^k = 2d + 2g - 3r + 1 \quad \text{and} \quad \sum_{i,j,k} k n_{ij}^k \in \{r - 3, r - 2\}.\]
Our first equation gives
\[(r - 2) - \sum_{i,j,k} k n_{ij}^k = 2d + 2g - 3r + 1 \quad \Rightarrow \quad \sum_{i,j,k} k n_{ij}^k = 4r - 2d - 2g - 3.\]
This in addition implies
\[4r - 2d - 2g - 3 = \sum_{i,j,k} k n_{ij}^k = r - 3 \ \text{or} \ r - 2.\]
Reducing mod $5$ and rearranging, we obtain
\[r + d + g \equiv 0 \ \text{or} \ 2 \ \mod 5.\]
But this case is excluded by assumption,
since $\delta(n) = 3$ and $\sum_{i,j,k} k n_{ij}^k \geq r - 3 > r/2$.
\end{proof}

\begin{proposition} \label{prop:high-0}
Suppose that \cref{thm:high} holds for all $(d', g', r'; n')$
where either $d' < d$, or $d' = d$ and $\sum_{i,j,k} (n')_{ij}^k < \sum_{i,j,k} n_{ij}^k$.
Then \cref{thm:high} holds for $(d, g, r; n)$ if $\sum_{i,j,k} (i + j) \cdot n_{ij}^k = 0$.
\end{proposition}
\begin{proof} 
Consider first the case when
\[\sum_{i,j,k} ((r - 2)i + (r - 3)j - k) \cdot n_{ij}^k \leq 2d + 2g - 3r.\]
By \cref{K-maximal-lhs-small}, it suffices to consider the case
$\sum_{i,j,k} k n_{ij}^k < r - 3$. By \cref{two-secant} it is thus sufficient
to show $(d - 1, g - 1, r; n')$ satisfies our inductive hypothesis, where
\[(n')_{ij}^k = \begin{cases}
n_{ij}^k & \text{if $(i, j; k) \neq (1, 1; 1)$;} \\
n_{11}^1 + 1 & \text{if $(i, j; k) = (1, 1; 1)$.}
\end{cases}\]

For this, we first note that $\sum_{i,j,k} k(n')_{ij}^k < r - 2$;
since $\delta(n') = 5$, our problem is thus reduced to showing that we do not simultaneously have
\begin{align*}
1 + \sum_{i,j,k} k n_{ij}^k = \sum_{i,j,k} k (n')_{ij}^k &= 4r - 2(d - 1) - 2(g - 1) - 5 > \frac{r + 3}{2}, \\
(d - 1) + (g - 1) + r &\equiv 5 + 2 \ \text{or} \ 5 + 4 \ \mod 5.
\end{align*}
Or equivalently, that we do not simultaneously have
\begin{align*}
\sum_{i,j,k} k n_{ij}^k &= 4r - 2d - 2g - 2 > \frac{r + 1}{2}, \\
d + g + r &\equiv 4 \ \text{or} \ 1 \ \mod 5.
\end{align*}
But by assumption (since $\delta(n) = 2$), we do not simultaneously have
\begin{align*}
\sum_{i,j,k} k n_{ij}^k &= 4r - 2d - 2g - 2 > \frac{r}{2}, \\
d + g + r &\equiv 4 \ \text{or} \ 1 \ \mod 5.
\end{align*}

\noindent
Next, we consider the case when
\begin{equation} \label{ij0-high}
\sum_{i,j,k} ((r - 2)i + (r - 3)j - k) \cdot n_{ij}^k \geq 2d + 2g - 3r + 1.
\end{equation}
In this case, we first of all claim that the inequalities of \cref{zzk}
and \cref{lower-d} are satisfied. To check this, we apply \cref{zzk-dg},
which reduces our claim to verifying the inequality
\[\sum_{i,j,k} k n_{ij}^k \leq \frac{3r^2 - 3r - 4}{2r - 4} - \frac{r - 5}{2r - 4}(d + g).\]
On the other hand, our assumption implies
\[- \sum_{i,j,k} k n_{ij}^k = \sum_{i,j,k} ((r - 2)i + (r - 3)j - k) \cdot n_{ij}^k \geq 2d + 2g - 3r + 1 \quad \Rightarrow \quad d + g \leq \frac{3r - 1 - \sum_{i,j,k} k n_{ij}^k}{2}.\]
We are thus reduced to showing
\[\sum_{i,j,k} k n_{ij}^k \leq \frac{3r^2 - 3r - 4}{2r - 4} - \frac{r - 5}{2r - 4} \cdot \frac{3r - 1 - \sum_{i,j,k} k n_{ij}^k}{2}.\]
Or, upon rearrangement, that
\[\sum_{i,j,k} k n_{ij}^k \leq \frac{3r + 13}{3}.\]
For this is it sufficient to note that
\[r - 2 \leq \frac{3r + 13}{3}.\]

Now if $\sum_{i,j,k} n_{ij}^k > 0$, we may iteratively
apply \cref{zzk} --- noting that if $n'$ is as in \cref{zzk},
then $(d, g, r; n')$ also satisfies the inequality of \cref{zzk} --- to
reduce to cases where $\sum_{i,j,k} n_{ij}^k = 0$
(which are good by Condition~\ref{K-small} of \cref{cor:high}).
Similarly, if $\sum_{i,j,k} n_{ij}^k = 0$ and
$d > g + r$, we may apply \cref{lower-d},
again applying Condition~\ref{K-small} of \cref{cor:high}
to check the inductive hypothesis.
It therefore remains to consider the case where $d = g + r$
and $\sum_{i,j,k} n_{ij}^k = 0$.

If the inequality \cref{ij0-high} is strict, then the desired
result follows from \cref{stick}: The $n'$ appearing in \cref{stick}
satisfies
\[\sum_{i,j,k} k(n')_{ij}^k = 0.\]

We may thus suppose additionally that \cref{ij0-high}
is an equality. But in this case, we additionally have
\[2d + 2g - 3r + 1 = \sum_{i,j,k} ((r - 2)i + (r - 3)j - k) \cdot n_{ij}^k = 0.\]
Using the above equation together with $d = g + r$ to solve for $d$ and $r$
in terms of $g$, we obtain
\[d = 5g + 1 \quad \text{and} \quad r = 4g + 1.\]
We are thus done by \cref{last-case}.
\end{proof}

\titlelabel{Appendix \thetitle:\quad}
\section{Code for \cref{S:low-dimensions}}
\label{A:code}

\noindent
In this section, we give python code to do the finite
computations described in \cref{S:low-dimensions}.

\begin{lstlisting}
class Point:
	def __init__(self, i, j, k):
		self.i = i
		self.j = j
		self.k = k

	def __repr__(self):
		return '(' + str(self.i) + ', ' + str(self.j) + '; ' + str(self.k) + ')'

	def as_tuple(self):
		return (self.i, self.j, self.k)

	def __hash__(self):
		return hash(self.as_tuple())
	
	def __eq__(self, other):
		return self.as_tuple() == other.as_tuple() 
	
	def __ne__(self, other):
		return not (self == other)

# Allowed types of markings for points:
P111 = Point(1, 1, 1)
P201 = Point(2, 0, 1)
P102 = Point(1, 0, 2)
P110 = Point(1, 1, 0)
P101 = Point(1, 0, 1)
P200 = Point(2, 0, 0)
P002 = Point(0, 0, 2)
P100 = Point(1, 0, 0)
P001 = Point(0, 0, 1)
P000 = Point(0, 0, 0)

BLANK = P000 # Blank point (to eliminate).
POINTS = [P111, P201, P102, P110, P101, P200, P002, P100, P001]
NO_MARKINGS = {P:0 for P in POINTS}

class Curve:
	def __init__(self, d, g, r, m = NO_MARKINGS.copy()):
		self.d = d
		self.g = g
		self.r = r

		self.m = m

		for i in m.keys():
			if i not in POINTS:
				raise ValueError

		self.nm = sum([self.m[P] for P in POINTS]) # Number of marked points.

		self.I = sum([self.m[P] * P.i for P in POINTS]) # \sum i n_{ij}^k
		self.J = sum([self.m[P] * P.j for P in POINTS]) # \sum j n_{ij}^k
		self.K = sum([self.m[P] * P.k for P in POINTS]) # \sum k n_{ij}^k

		self.lhs = (r - 2) * self.I + (r - 3) * self.J - self.K # \sum [(r - 2)i + (r - 3)j - k] n_{ij}^k
		
		if self.K > r - 2:
			raise ValueError, "K is too big."

	def as_tuple(self):
		return (self.d, self.g, self.r, tuple([self.m[P] for P in POINTS]))

	def __hash__(self):
		return hash(self.as_tuple())

	def __eq__(self, other):
		return self.as_tuple() == other.as_tuple()

	def __ne__(self, other):
		return not (self == other)

	def __repr__(self):
		out = 'Curve of degree ' + str(self.d) + ' and genus ' + str(self.g) + ' in P^' + str(self.r)
		if self.m != NO_MARKINGS:
			out += ', with marked points:'
			for P in POINTS:
				if self.m[P]:
					out += '\n   ' + str(P) + ' x ' + str(self.m[P])
		return out + str('.')

	def delete(self, P, n = 1):
		mprime = self.m.copy()
		if mprime[P] < n:
			raise ValueError, "Cannot delete a point that does not exist."
		mprime[P] -= n
		return Curve(self.d, self.g, self.r, mprime)

	def add(self, P, n = 1):
		mprime = self.m.copy()
		mprime[P] += n
		return Curve(self.d, self.g, self.r, mprime)

	def replace(self, f):
		mprime = NO_MARKINGS.copy()
		for P in POINTS:
			if f(P) != BLANK:
				mprime[f(P)] += self.m[P]
		return Curve(self.d, self.g, self.r, mprime)

	def lower_d(self, n = 1):
		return Curve(self.d - n, self.g, self.r, self.m)

	def lower_g(self, n = 1):
		return Curve(self.d, self.g - n, self.r, self.m)

	def lower_r(self, n = 1):
		return Curve(self.d, self.g, self.r - n, self.m)

def partition(ijk, types = POINTS):
	if len(types) == 0:
		if ijk == [0, 0, 0]:
			yield {}
		return

	t = types[0].as_tuple()
	limits = []

	for l in (0, 1, 2):
		if t[l] != 0:
			limits.append(ijk[l] / t[l])

	for n in xrange(1 + min(limits)):
		for P in partition([ijk[l] - n * t[l] for l in (0, 1, 2)], types[1:]):
			P[types[0]] = n
			yield P.copy()
	return

def all_curves(d, g, r):
	bound = 2 * d + 2 * g - r - 2
	for k in xrange(r - 1):
		for i in xrange(1 + bound + k):
			for j in xrange(1 + i):
				if (r - 2) * i + (r - 3) * j - k <= bound:
					for m in partition([i, j, k]):
						yield Curve(d, g, r, m)

GOOD = {}
def good(C):
	if C in GOOD:
		return GOOD[C]

	if C.lhs > 2 * C.d + 2 * C.g - C.r - 2:
		return False

	if C.g == 0:
		return True
	
	if C.r == 2:
		return True

	if C.r == 3:
		if C.K == 0:
			if (C.J != 0) or (C.I != 2 * C.d + 2 * C.g - 14):
				return True
		else:
			if C.I != 2 * C.d + 2 * C.g - 9:
				return True

	if (C.nm == 0) and (C.d == 5 * C.g + 1) and (C.r == 4 * C.g + 1):
		return True

	if C == Curve(8, 3, 5).add(P101, 2):
		return True

	if (C.nm == 0) and (C.r == 5) and (C.g >= 2):
		if good(C.lower_d(2).lower_g(2).add(P101, 2)):
			GOOD[C] = True; return True

	if (C.r - 1) * C.nm - C.I - 2 * C.J - C.K <= (C.r + 1) * C.d - (2 * C.r - 4) * C.g - 2:
		for P in POINTS:
			if (P.j == P.k == 0) and (C.m[P] > 0):
				if C.K < C.r - 2:
					if good(C.delete(P)) and good(C.replace(lambda P : Point(P.j, 0, P.k)).lower_r().lower_d()):
						GOOD[C] = True; return True
				else:
					if good(C.delete(P)) and good(C.replace(lambda P : Point(P.j, 0, 0)).lower_r().lower_d()):
						GOOD[C] = True; return True

			if (P.i == P.j == 0) and (C.m[P] > 0):
				if good(C.delete(P)) and good(C.replace(lambda P : Point(P.i, P.j, 0))):
					GOOD[C] = True; return True
	
		if C.d > C.g + C.r:
			if (C.r != 3) or (C.J == 0):
				if C.K < C.r - 2:
					if good(C.lower_d()) and good(C.lower_d().lower_r()):
						GOOD[C] = True; return True
				else:
					if good(C.lower_d()) and good(C.replace(lambda P : Point(P.i, P.j, 0)).lower_d().lower_r()):
						GOOD[C] = True; return True
	
	if C.K < C.r - 2:
		if good(C.add(P111).lower_d().lower_g()):
			GOOD[C] = True; return True
	else:
		if good(C.replace(lambda P : Point(P.i, P.j, 0)).add(P110).lower_d().lower_g()):
			GOOD[C] = True; return True
	
	if (C.r != 3):
		if C.K == C.r - 3:
			if good(C.replace(lambda P : Point(P.i, P.j, 0)).add(P101).lower_d().lower_g()):
				GOOD[C] = True; return True
		if C.K == C.r - 2:
			if good(C.replace(lambda P : Point(P.i, P.j, 0)).add(P102).lower_d().lower_g()):
				GOOD[C] = True; return True

	if C.lhs >= 2 * C.d + 2 * C.g - 3 * C.r + 2:
		if C.K < C.r - 2:
			if good(C.replace(lambda P : Point(P.j, 0, P.k)).lower_d().lower_r()):
				GOOD[C] = True; return True
		else:
			if good(C.replace(lambda P : Point(P.j, 0, 0)).lower_d().lower_r()):
				GOOD[C] = True; return True
	
	if (C.r != 3) and (2 * C.d + 2 * C.g - 4 * C.r + 3 <= C.lhs <= 2 * C.d + 2 * C.g - 2 * C.r - 1):
		if C.K < C.r - 3:
			if good(C.replace(lambda P : Point(P.j, 0, P.k)).add(P201).lower_d(2).lower_g().lower_r()):
				GOOD[C] = True; return True
		elif C.K == C.r - 3:
			if good(C.replace(lambda P : Point(P.j, 0, 0)).add(P200).lower_d(2).lower_g().lower_r()):
				GOOD[C] = True; return True
		else:
			if good(C.replace(lambda P : Point(P.j, 0, 0)).add(P201).lower_d(2).lower_g().lower_r()):
				GOOD[C] = True; return True
	
	GOOD[C] = False; return False

def check(r, dg_max):
	largest_dg = 0
	for d in xrange(r, dg_max + 1):
		for g in xrange(min(d - r, dg_max - d) + 1):
			for C in all_curves(d, g, r):
				if not good(C):
					largest_dg = max(largest_dg, d + g)
					if C.m == NO_MARKINGS:
						print 'Potential counterexample:', C
	if largest_dg != 0:
		print 'Largest potentially non-excellent d + g =', largest_dg

\end{lstlisting}

\noindent
The output is as follows:

\begin{lstlisting}[language=]
>>> check(4, 16)
Potential counterexample: Curve of degree 6 and genus 2 in P^4.
Largest potentially non-excellent d + g = 10
>>> check(5, 15)
Potential counterexample: Curve of degree 7 and genus 2 in P^5.
Largest potentially non-excellent d + g = 13
>>> check(6, 16)
Largest potentially non-excellent d + g = 12
>>> check(7, 15)
Largest potentially non-excellent d + g = 13
>>> check(8, 16)
>>> for r in xrange(9, 12):
...     check(r, 2 * r - 2)
... 
>>> 
\end{lstlisting}


\bibliographystyle{amsplain.bst}
\bibliography{Interpolation}

\end{document}